\documentclass{article}%
\usepackage{amsmath}
\usepackage{amsfonts}
\usepackage{amssymb}
\usepackage{graphicx}%
\providecommand{\U}[1]{\protect\rule{.1in}{.1in}}
\newtheorem{theorem}{Theorem}

\newtheorem{corollary}[theorem]{Corollary}

\newtheorem{definition}[theorem]{Definition}
\newtheorem{example}[theorem]{Example}

\newtheorem{lemma}[theorem]{Lemma}
\newtheorem{notation}[theorem]{Notation}

\newtheorem{proposition}[theorem]{Proposition}
\newtheorem{remark}[theorem]{Remark}

\newenvironment{proof}[1][Proof]{\noindent\textbf{#1.} }{\ \rule{0.5em}{0.5em}}
\numberwithin{equation}{section}
\numberwithin{theorem}{section}
\begin{document}

\title{Global Sobolev regularity for nonvariational operators built with homogeneous
H\"{o}rmander vector fields}
\author{Stefano Biagi and Marco Bramanti\\Dipartimento di Matematica - Politecnico di Milano}
\maketitle

\begin{abstract}
We consider a class of nonvariational degenerate elliptic operators of the
kind%
\[
Lu=\sum_{i,j=1}^{m}a_{ij}\left(  x\right)  X_{i}X_{j}u
\]
where $\left\{  a_{ij}\left(  x\right)  \right\}  _{i,j=1}^{m}$ is a symmetric
uniformly positive matrix of bounded measurable functions defined in the whole
$\mathbb{R}^{n}$ ($n>m$), possibly discontinuos but satisfying a $VMO$
assumption, and $X_{1},...,X_{m}$ are real smooth vector fields satisfying
H\"{o}rmander rank condition in the whole $\mathbb{R}^{n}$ and $1$-homogeneous
w.r.t. a family of nonisotropic dilations. We \emph{do not }assume that the
vector fields are left invariant w.r.t. an underlying Lie group of
translations. We prove global $W_{X}^{2,p}$ a-priori estimates, for every
$p\in\left(  1,\infty\right)  $, of the kind:%
\[
\Vert u\Vert_{W_{X}^{2,p}(\mathbb{R}^{n})}\leq c\left\{  \left\Vert
Lu\right\Vert _{L^{p}\left(  \mathbb{R}^{n}\right)  }+\left\Vert u\right\Vert
_{L^{p}\left(  \mathbb{R}^{n}\right)  }\right\}
\]
for every $u\in W_{X}^{2,p}\left(  \mathbb{R}^{n}\right)  .$ We also prove
higher order estimates and corresponding regularity results: if $a_{ij}\in
W_{X}^{k,\infty}\left(  \mathbb{R}^{n}\right)  $, $u\in W_{X}^{2,p}\left(
\mathbb{R}^{n}\right)  $, $Lu\in W_{X}^{k,p}\left(  \mathbb{R}^{n}\right)  $,
then $u\in W_{X}^{k+2,p}\left(  \mathbb{R}^{n}\right)  $ and%
\[
\Vert u\Vert_{W_{X}^{k+2,p}(\mathbb{R}^{n})}\leq c\left\{  \Vert
Lu\Vert_{W_{X}^{k,p}(\mathbb{R}^{n})}+\Vert u\Vert_{L^{p}(\mathbb{R}^{n}%
)}\right\}  .
\]

\end{abstract}

\section{Introduction}

In this paper we consider a class of nonvariational degenerate elliptic
operators of the kind%
\begin{equation}
Lu=\sum_{i,j=1}^{m}a_{ij}\left(  x\right)  X_{i}X_{j}u\label{operators}%
\end{equation}
where $\left\{  a_{ij}\left(  x\right)  \right\}  _{i,j=1}^{m}$ is a symmetric
uniformly positive matrix of bounded measurable functions defined in the whole
$\mathbb{R}^{n}$ ($n>m$), possibly discontinuos but satisfying a $VMO$
assumption, and $X_{1},...,X_{m}$ are real smooth vector fields satisfying
H\"{o}rmander rank condition in the whole $\mathbb{R}^{n}$ and $1$-homogeneous
w.r.t. a family of nonisotropic dilations. (Precise definition will be given
later). We \emph{do not} assume that the vector fields are left invariant
w.r.t. an underlying Lie group of translations, so our context is more general
than that of H\"{o}rmander vector fields on Carnot groups. Our first main
result consists in proving global $W_{X}^{2,p}$ a-priori estimates, for any
$p\in\left(  1,\infty\right)  $, for these operators (see Theorem
\ref{Thm main} for the exact statement). We also extend our result to higher
order derivatives, proving that if $u\in W_{X}^{2,p}\left(  \mathbb{R}%
^{n}\right)  $ and, for some positive integer $k$, the coefficients $a_{ij}$
belong to $W_{X}^{k,\infty}\left(  \mathbb{R}^{n}\right)  $ and $Lu\in
W_{X}^{k,p}\left(  \mathbb{R}^{n}\right)  $, then $u\in W_{X}^{k+2,p}\left(
\mathbb{R}^{n}\right)  $ and the corresponding global a priori estimate holds:%
\[
\Vert u\Vert_{W_{X}^{k+2,p}(\mathbb{R}^{n})}\leq c\left\{  \Vert
Lu\Vert_{W_{X}^{k,p}(\mathbb{R}^{n})}+\Vert u\Vert_{L^{p}(\mathbb{R}^{n}%
)}\right\}  .
\]

For H\"{o}rmander operators, global results are usually known in the case of
Carnot groups while, in the general case, only local results are usually
known. The context of homogeneous (but not left invariant)\ H\"{o}rmander
vector fields represents an interesting intermediate case in which,
notwithstanding the lacking of a group structure, very interesting global
results can be established. The study of (stationary and evolutive) operators
built with homogeneous H\"{o}rmander vector fields has been started in recent
years by Biagi and Bonfiglioli in \cite{BBlift} and continued e.g. in
\cite{BBB1}, \cite{BBB2}, \cite{BiB1}, \cite{BiB2}.

Regularity results and a-priori estimates in the scale of Sobolev spaces have
been proved:

\begin{itemize}
\item[-] for H\"{o}rmander operators on groups, in local form, by Folland in
\cite{Fo2}; for sum of squares of general H\"{o}rmander vector fields, in
local form, by Rothschild-Stein in \cite{RS}; for sum of squares of
H\"{o}rmander vector fields on Carnot groups, in global form, by
Bramanti-Brandolini in \cite[Chap. 8]{BBbook}; for sum of squares of
homogeneous H\"{o}rmander vector fields, in global form, by
Biagi-Bonfiglioli-Bramanti in \cite{BB2};

\item[-] for nonvariational operators of kind (\ref{operators}) with $VMO$
coefficients, on Carnot groups, in local form, by Bramanti-Brandolini in
\cite{BB1}; without the Carnot group assumption, by Bramanti-Brandolini in
\cite{BB2} (and, with more complete results, in \cite[Chap. 12.]{BBbook}).
\end{itemize}

The study of nonvariational operators with $VMO$ coefficients has been
initiated in 1991-93 by Chiarenza-Frasca-Longo \cite{CFL1}, \cite{CFL2} for
uniformly elliptic operators. The classical $W^{2,p}$-theory for uniformly
elliptic operators requires the uniform continuity of the coefficients
$a_{ij}$, and it is well-known that for nonvariational uniformly elliptic
operators with general bounded measurable coefficients, in dimension $n>2$,
the $W^{2,p}$-theory breaks down. The $VMO$ requirement is a kind of
\textquotedblleft uniform continuity in integral sense\textquotedblright%
\ which actually allows for some soft kind of discontinuities, hence
Chiarenza-Frasca-Longo's results represent a significant extension of the
classical elliptic theory of the 1950-1960's. The techique firstly devised by
Chiarenza-Frasca-Longo has been later applied to uniformly parabolic equations
(see \cite{BC}), Kolmogorov-Fokker-Planck ultraparabolic equations (see
\cite{BCM}), and to nonvariational operators built on H\"{o}rmander vector
fields, as already explained. Chiarenza-Frasca-Longo's technique heavily
relies on representation formulas of $u_{x_{i}x_{j}}$ in terms of $Lu$, via a
singular kernel \textquotedblleft with variable coefficients\textquotedblright%
, obtained from the fundamental solution of the model operator with constant
coefficients. This singular kernel needs to be expanded in series of spherical
harmonics, getting singular kernels \textquotedblleft with constant
coefficients\textquotedblright. To these singular integrals \textquotedblleft
of convolution type\textquotedblright\ one needs to apply results of $L^{p}$
continuity of both the corresponding singular integral operator and the
commutator of the singular integral operator with a $BMO$ function. This
technique exploits both translation invariance and homogeneity (in suitable
senses) of the model operator with constant $a_{ij}$ or, in some cases, the
possibility of approximating locally the operator under study with another one
possessing these properties. Moreover, it relies on the knowledge of fine
properties of the fundamental solution of the constant coefficient operator,
with bounds on the derivatives of every order of this fundamental solution,
which have to be uniform as the constant matrix $\left(  a_{ij}\right)  $
ranges in a fixed ellipticity class.

A different technique to prove $W^{2,p}$ a-priori estimates for nonvariational
elliptic or parabolic operators with $VMO$ coefficients has been devised in
2007 by Krylov in \cite{K} (and exploited in a series of subsequent papers). A
key step in Krylov' technique constists in establishing a pointwise estimate
on the sharp maximal function of the second derivatives of $u$, $\left(
u_{x_{i}x_{j}}\right)  ^{\#}$, in terms of $\overline{L}u$ where $\overline
{L}$ is the model operator with constant $a_{ij}$. Once this estimate is
proved, a clever procedure allows to exploit the $VMO$ assumption on $a_{ij}$
for replacing the constant coefficient operator with $L$. In turn, in order to
establish this estimate on $\left(  u_{x_{i}x_{j}}\right)  ^{\#}$, Krylov
makes use of many deep classical results on elliptic or parabolic operators
with constant coefficients. The extension of this part of Krylov' technique to
degenerate operators of H\"{o}rmander type seems very difficult. A first
positive result in this direction has been obtained by Bramanti-Toschi in
\cite{BT} in the case of Carnot groups. More recently, Dong and Yastrzhembskiy
\cite{HY} have successfully employed this technique for (degenerate) kinetic
Kolmogorov-Fokker-Planck operators.

In this paper we present a new approach which combines some ideas of both the
strategies described above and shows some advantages. We follow Krylov in the
idea of exploiting an estimate on the sharp maximal function of the second
order derivatives (in terms of the constant coefficient operator) to prove
$L^{p}$ estimates for the operator with $VMO$ coefficients (see Theorem
\ref{Thm Krylov main step}). But, in order to prove this bound on the sharp
maximal function, we do not follow Krylov' technique but use representation
formulas for second order derivatives in terms of the \emph{constant
coefficient operator} (so that the singular kernel is not a \textquotedblleft
variable kernel\textquotedblright\ in the sense of Calder\'{o}n-Zygmund
theory) and make use of techniques of singular integrals in spaces of
homogeneous type. The link between the two ingredients of our proof (singular
integral operators \emph{and }sharp maximal function) is contained in our
Theorem \ref{Thm sharp}, an abstract real analysis result in spaces of
homogeneous type, which in our opinion can be of independent interest. This
technique, as usual, allows to prove, at first, a-priori estimates for smooth
functions with \emph{small }support. Passing from this result to global
a-priori estimates for any $W_{X}^{2,p}$-function requires the use of suitable
cutoff functions, a covering theorem, and interpolation inequalities for first
order derivatives. These properties are not trivial but, in our setting, the
presence of dilations, the validity of a global doubling property, and some
known properties of the measure of metric balls, allow to prove them in a
relatively direct way.

Our technique, as a whole, seems more flexible than both the techniques to
which it is inspired, and in some sense it is easier than those. For instance,
one could use the technique in this paper to reobtain the main result in
\cite{CFL1} with a shorter proof.

The results in this paper represent the first instance of a complete
regularity theory in Sobolev spaces for nonvariational operators built on
H\"{o}rmander vector fields, with global bounds in $\mathbb{R}^{n}$.

Let us now come to precise definitions and assumptions.

\bigskip

Let us consider a family $X=\{X_{1},\ldots,X_{m}\}$ of \emph{homogeneous
H\"{o}rmander vector fields on} $\mathbb{R}^{n}$. This means that the
following conditions (H.1), (H.2), (H.3) hold:

\begin{itemize}
\item[\textbf{(H.1)}] there exists a family of (non-isotropic) dilations
$\{\delta_{\lambda}\}_{\lambda>0}$ of the form
\begin{equation}
\delta_{\lambda}:\mathbb{R}^{n}\longrightarrow\mathbb{R}^{n}\qquad
\delta_{\lambda}(x)=\left(  \lambda^{\sigma_{1}}x_{1},\ldots,\lambda
^{\sigma_{n}}x_{n}\right)  , \label{intro.dela}%
\end{equation}
where $1=\sigma_{1}\leq\cdots\leq\sigma_{n}$, such that $X_{1},\ldots,X_{m}$
are $\delta_{\lambda}$-homogeneous of degree $1$, i.e.,
\[
X_{j}\left(  f\circ\delta_{\lambda}\right)  =\lambda\,\left(  X_{j}f\right)
\circ\delta_{\lambda},\quad\text{for every $\lambda>0$, $f\in C^{\infty
}\left(  \mathbb{R}^{n}\right)  $ and $j=1,\ldots,m$.}%
\]

\end{itemize}

\noindent In what follows, we denote by
\begin{equation}
q=\sum_{j=1}^{n}\sigma_{j} \label{eq.defq}%
\end{equation}
the so-called $\delta_{\lambda}$-homogeneous dimension of $(\mathbb{R}%
^{n},\delta_{\lambda})$.\medskip

\begin{itemize}
\item[\textbf{(H.2)}] $X_{1},\ldots,X_{m}$ are linearly independent (as vector
fields) and satisfy H\"{o}rmander's rank condition at $0$, i.e.,
\[
\dim\left\{  Y\left(  0\right)  :Y\in\mathrm{Lie}\left(  X\right)  \right\}
=n.
\]
Here $\mathrm{Lie}(X)$ stands for the smallest Lie subalgebra of
$\mathcal{X}(\mathbb{R}^{n})$ containing $X$, where $\mathcal{X}%
(\mathbb{R}^{n})$ is the Lie algebra of all the smooth vector fields on
$\mathbb{R}^{n}.$
\end{itemize}

\noindent Finally, we make the following dimensional assumptions:

\begin{itemize}
\item[\textbf{(H.3)}] we require that $q>2$ and
\begin{equation}
N:=\mathrm{dim}\left(  \mathrm{Lie}\left\{  X\right\}  \right)  >n\geq2.
\label{ipotesidimensinale}%
\end{equation}

\end{itemize}

Throughout the paper we let $p:=N-n\geq1$ and we denote the points of
$\mathbb{R}^{N}\equiv\mathbb{R}^{n}\times\mathbb{R}^{p}$ by
\[
(x,\xi),\quad\text{with }x\in\mathbb{R}^{n}\text{ and $\xi\in\mathbb{R}^{p}$.}%
\]

We stress that the vector fields $X_{i}$ are \emph{not} assumed to be
translation invariant with respect to any underlying structure of Lie group,
so that our context is different from that of Carnot groups.

\begin{remark}
Assumption (H.3) is not restrictive for the following reason. As noted in
\cite[Section 2]{BiB2}, if (H.3) does not hold, that is $N=n$, then our vector
fields $X_{1},...,X_{m}$ are the generators of a Carnot group. Under this
assumption, the main results in this paper, although not explicitly present in
the literature, could be plainly obtained combining the techniques in
\cite[Chap. 8]{BBbook} and \cite[Chap. 12.]{BBbook}. Instead, under our
assumption (H.3), the vector fields $X_{1},...,X_{m}$ \emph{cannot }be the
generators of a Carnot group (again, this is noted in \cite[Section 2]{BiB2}).
Hence the proofs written in this paper cannot be directly specialized to the
case of Carnot groups; however it will be apparent that in the case of Carnot
groups the proof of the present results would be easier, following the same line.
\end{remark}

In this paper we are going to study \emph{nonvariational second order
operators }(\ref{operators}) built with a family of vector fields
$X_{1},...,X_{m}$ satisfying (H.1)-(H.3), where:

\begin{itemize}
\item[\textbf{(H.4)}] $\left\{  a_{ij}\right\}  _{i,j=1}^{m}$ is a symmetric
uniformly positive matrix of bounded measurable functions: there exists
$\nu>0$ such that%
\begin{equation}
\nu\left\vert w\right\vert ^{2}\leq\sum_{i,j=1}^{m}a_{ij}\left(  x\right)
w_{i}w_{j}\leq\nu^{-1}\left\vert w\right\vert ^{2}\text{ for every }%
w\in\mathbb{R}^{m},\text{a.e. }x\in\mathbb{R}^{n}. \label{ellipticity}%
\end{equation}

\end{itemize}

Moreover, we will assume that the coefficients $a_{ij}$, although possibly
discontinuous, belong to $VMO\left(  \mathbb{R}^{n}\right)  $, the space of
functions with \emph{vanishing mean oscillation}, with respect to the
structure of \emph{space of homogeneous type} induced in $\mathbb{R}^{n}$ by
the vector fields $X_{1},X_{2},...,X_{m}$. This space will be defined
precisely later (see Definition \ref{Def VMO} and assumption (H.5) in
\S \ref{sec geometry}).

For these operators, we are interested in proving global $L^{p}\left(
\mathbb{R}^{n}\right)  $-estimates (for $1<p<\infty$) for $X_{i}X_{j}u$ in
terms of $Lu$ and $u$ itself, and then to extend these to higher order
estimates, under suitable regularity assumptions on the coefficients.

Let us first introduce the Sobolev spaces related to our system of
H\"{o}rmander vector fields, fixing the related notation (see \cite[Chap.2]%
{BBbook} for details).

\begin{definition}
[Multiindices, Sobolev spaces]\label{Def Sobolev}Let $X=\{X_{1},\ldots
,X_{m}\}$ be as above. For any multiindex $I=\left(  i_{1},i_{2}%
,...,i_{k}\right)  $ (with $i_{j}\in\left\{  1,2,...,m\right\}  $), we say
that $I$ has \emph{length }$k$, we write $\left\vert I\right\vert =k$, and set%
\[
X_{I}f=X_{i_{1}}X_{i_{2}}...X_{i_{k}}f
\]
for any function $f$ such that these derivatives exist in classical or weak sense.

For any $p\in\left[  1,\infty\right]  $, positive integer $k$, and domain
$\Omega\subseteq\mathbb{R}^{n}$, we define the Sobolev space%
\[
W_{X}^{k,p}\left(  \Omega\right)  =\left\{  f:\Omega\rightarrow\mathbb{R}%
:\left\Vert f\right\Vert _{W_{X}^{k,p}\left(  \Omega\right)  }<\infty\right\}
\]
where%
\[
\left\Vert f\right\Vert _{W_{X}^{k,p}\left(  \Omega\right)  }=\left\Vert
f\right\Vert _{L^{p}\left(  \Omega\right)  }+\sum_{\left\vert I\right\vert
\leq k}\left\Vert X_{I}f\right\Vert _{L^{p}\left(  \Omega\right)  }%
\]
and all the derivatives are meant in weak sense.

We also define $W_{0,X}^{k,p}\left(  \Omega\right)  $ as the closure of
$C_{0}^{\infty}\left(  \Omega\right)  $ in $W_{X}^{k,p}\left(  \Omega\right)
$.
\end{definition}

Let us recall some known approximation result by smooth functions, which holds
for Sobolev spaces defined by any system of H\"{o}rmander vector fields:

\begin{theorem}
\label{Thm approx}(See \cite[Thm. 2.9, Coroll. 2.10]{BBbook}). Let $u\in
W_{X}^{k,p}\left(  \Omega\right)  $ for some open set $\Omega\subseteq
\mathbb{R}^{n}$, some $p\in\lbrack1,\infty)$ and some positive integer $k$.

(a). If $\Omega^{\prime}$ is an open set compactly contained in $\Omega$, then
there exists a sequence $\left\{  u_{j}\right\}  \subset C_{0}^{\infty}\left(
\Omega\right)  $ such that $u_{j}\rightarrow u$ in $W_{X}^{k,p}\left(
\Omega^{\prime}\right)  $.

(b) If $\phi\in C_{0}^{\infty}\left(  \Omega\right)  $ then $u\phi\in
W_{0,X}^{k,p}\left(  \Omega\right)  $.
\end{theorem}

\bigskip

We can now state the first main result of this paper:

\begin{theorem}
\label{Thm main}Under assumptions (H.1)-(H.5) (assumption (H.5) will be stated
in section 2, after Definition \ref{Def VMO}), for every $p\in\left(
1,\infty\right)  $ there exists $c>0$, depending on numbers $p,\nu$ (in
(\ref{ellipticity}))$,$ the vector fields $\left\{  X_{1},...,X_{m}\right\}  $
and the function $a_{\cdot}^{\sharp}$ (that will be defined in
(\ref{mod VMO coeff}), and encodes the $VMO$ assumption on $a_{ij}$) such that
for every $u\in W_{X}^{2,p}\left(  \mathbb{R}^{n}\right)  $ we have%
\begin{equation}
\left\Vert u\right\Vert _{W_{X}^{2,p}\left(  \mathbb{R}^{n}\right)  }\leq
c\left\{  \left\Vert Lu\right\Vert _{L^{p}\left(  \mathbb{R}^{n}\right)
}+\left\Vert u\right\Vert _{L^{p}\left(  \mathbb{R}^{n}\right)  }\right\}  .
\label{k a priori}%
\end{equation}

\end{theorem}

We will also extend the basic estimate contained in the previous theorem to a
higher order regularity result:

\begin{theorem}
\label{Thm main 2}Under assumptions (H.1)-(H.4), assume that, for some
$p\in\left(  1,\infty\right)  $ and positive integer $k$, $a_{ij}\in
W_{X}^{k,\infty}\left(  \mathbb{R}^{n}\right)  $, $u\in W_{X}^{2,p}\left(
\mathbb{R}^{n}\right)  $, and $Lu\in W_{X}^{k,p}\left(  \mathbb{R}^{n}\right)
$. Then $u\in W_{X}^{k+2,p}\left(  \mathbb{R}^{n}\right)  $ and there exists a
constant $c>0$, only depending on the numbers $k,p,\nu$ \emph{(}see
(\ref{ellipticity})\emph{)}, on the vector fields $\{X_{1},...,X_{m}\}$ and on
the number
\[
\Vert a\Vert=\sum_{i,j=1}^{m}\Vert a_{ij}\Vert_{W_{X}^{k,\infty}%
(\mathbb{R}^{n})}<\infty\text{,}%
\]
such that:%
\[
\Vert u\Vert_{W_{X}^{k+2,p}(\mathbb{R}^{n})}\leq c\left\{  \Vert
Lu\Vert_{W_{X}^{k,p}(\mathbb{R}^{n})}+\Vert u\Vert_{L^{p}(\mathbb{R}^{n}%
)}\right\}  .
\]

\end{theorem}

\begin{remark}
We will see in Proposition \ref{Prop VMO W1inf} that
\[
W_{X}^{1,\infty}\left(  \mathbb{R}^{n}\right)  \subset VMO\left(
\mathbb{R}^{n}\right)  ,
\]
hence the assumption on the coefficients $a_{ij}$ in Theorem \ref{Thm main 2}
is actually stronger than that in Theorem \ref{Thm main}.
\end{remark}

Theorem \ref{Thm main} will be proved throughout sections 2-4; Theorem
\ref{Thm main 2} will be proved throughout sections 5-6.

\begin{example}
(1).\thinspace\thinspace Let us consider Franchi-Lanconelli-type vector fields
in $\mathbb{R}^{2}$%
\[
X_{1}=\partial_{x_{1}};X_{2}=x_{1}^{k}\,\partial_{x_{2}}\qquad\text{(with
$k\in\mathbb{N}$)},
\]
associated with the dilations $(\lambda x_{1},\lambda^{k+1}x_{2})$. These are
not left-invariant PDO's on any Lie group on $\mathbb{R}^{2}$ (when $k\geq1$)
as $x_{1}^{k}\partial_{x_{2}}$ vanishes when $x_{1}=0$ without being the null
vector field.\medskip

(2).\thinspace\thinspace Another system of vector fields to which our theory
applies is given by
\[
X_{1}=\partial_{x_{1}};X_{2}=x_{1}\partial_{x_{2}}+x_{2}\partial_{x_{3}%
}+\ldots+x_{n-1}\partial_{x_{n}}\quad\text{in }\mathbb{R}\text{$^{n}$,}%
\]
which are $\delta_{\lambda}$-homogeneous of degree $1$ (but not left invariant
on $\mathbb{R}^{n}$, for the same reason as in (1)) with respect to the
dilations $\delta_{\lambda}(x)=(\lambda x_{1},\lambda^{2}x_{2},\cdots
,\lambda^{n}x_{n})$.\medskip

(3).\thinspace\thinspace\ A further example is the system of vector fields:
\[
X_{1}=\partial_{x_{1}};X_{2}=x_{1}\,\partial_{x_{2}}+x_{1}^{2}\,\partial
_{x_{3}}\quad\text{on }\mathbb{R}\text{$^{3}$,}%
\]
which are homogeneous of degree $1$ (but not left invariant on $\mathbb{R}%
^{3}$) with respect to
\[
\delta_{\lambda}(x)=(\lambda x_{1},\lambda^{2}x_{2},\lambda^{3}x_{3}).
\]

(4).\thinspace\thinspace Finally, the operator
\[
X_{1}=\partial_{x_{1}};X_{2}=x_{1}\,\partial_{x_{2}}+x_{1}^{2}\,\partial
_{x_{3}}+\cdots+x_{1}^{n-1}\,\partial_{x_{n}}\quad\text{on }\mathbb{R}%
\text{$^{n}$}%
\]
are homogeneous of degree $1$ with respect to the same dilations as in (2),
but not left invariant on $\mathbb{R}^{n}$.
\end{example}

\section{Definitions, known results and preliminaries}

\subsection{Geometry of vector fields and real analysis
tools\label{sec geometry}}

Let us start with a couple of abstract definitions.

\begin{definition}
[Quasidistance]Let $X$ be a set. We say that a function $d:X\times
X\rightarrow\lbrack0,+\infty)$ is a \emph{quasidistance }in $X$ if there
exists a constant $C_{d}\geq1$ such that, for every $x,y,z\in X$:

(i) $d\left(  x,y\right)  =0\Longleftrightarrow x=y$

(ii) $d\left(  x,y\right)  =d\left(  y,x\right)  $

(iii) $d\left(  x,y\right)  \leq C_{d}\left(  d\left(  x,z\right)  +d\left(
z,y\right)  \right)  $

In this case we also say that $\left(  X,d\right)  $ is a \emph{quasimetric
space}.
\end{definition}

If $\left(  X,d\right)  $ is a quasimetric space, the $d$-balls, defined, for
$x\in X$ and $r>0$ as
\[
B_{r}\left(  x\right)  =\left\{  y\in X:d\left(  x,y\right)  <r\right\}  ,
\]
induce a topology.

\begin{definition}
[Space of homogeneous type]\label{Def spazio omogeneo}Let $\left(  X,d\right)
$ be a quasimetric space, and assume that:

(a) the $d$-balls are open w.r.t. the topology they induce;

(b) there exists a positive Borel measure $\mu$ on $X$, satisfying the
doubling property: there exists a constant $C_{\mu}\geq1$ such that
\[
0<\mu\left(  B_{2r}\left(  x\right)  \right)  \leq C_{\mu}\cdot\mu\left(
B_{r}\left(  x\right)  \right)  <\infty
\]
for every $x\in X,r>0$.

Then we say that $\left(  X,d,\mu\right)  $ is a \emph{space of homogeneous
type}, in the sense of Coifman-Weiss \cite{CW}. The constants $C_{\mu}$ in (b)
and $C_{d}$ in (iii) of the previous Definition are called the constants of
$X$.
\end{definition}

\bigskip

Let $d_{X}$ be the Carnot-Carath\'{e}odory distance (CC-distance) induced by
the system of H\"{o}rmander vector fields $X=\{X_{1},\ldots,X_{m}\}$, that
is,
\begin{equation}
d_{X}(x,y):=\inf\left\{  r>0:\,\text{there exists $\gamma\in C_{xy}(r)$%
}\right\}  , \label{eq.defdCC}%
\end{equation}
where $C_{xy}(r)$ is the set of the absolutely continuous maps $\gamma
:[0,1]\rightarrow\mathbb{R}^{n}$ satisfying $\gamma(0)=x$, $\gamma(1)=y$ and
(a.e.\thinspace on $[0,1]$):%
\[
\gamma^{\prime}(t)=\sum_{j=1}^{m}a_{j}(t)\,X_{j}(\gamma(t)),\qquad\text{with
$|a_{j}(t)|\leq r$ for all $j=1,\ldots,m$}.
\]
Given any $x\in\mathbb{R}^{n}$ and $r>0$, we denote by $B_{X}(x,r)$, or simply
$B_{r}\left(  x\right)  $, the $d_{X}$-ball of center $x$ and radius $r$. For
general facts about the CC-distance, the reader is referred to \cite[Chap.1]%
{BBbook}. Without risk of confusion, $|\cdot|$ will denote the Lebesgue
measure in $\mathbb{R}^{n}$ (whatever the $n$).

The following basic facts are proved in \cite{BBB1}:

\begin{proposition}
(See \cite[Section 3)]{BBB1}) A \emph{global doubling condition }holds in this
context: there exists $C_{d}>0$ such that
\begin{equation}
\left\vert B_{X}\left(  x,2r\right)  \right\vert \leq C\left\vert B_{X}\left(
x,r\right)  \right\vert \text{ for every }r>0\text{ and }x\in\mathbb{R}^{n}
\label{doubling}%
\end{equation}
In other words, $\left(  \mathbb{R}^{n},d_{X},\left\vert \cdot\right\vert
\right)  $ is a doubling metric space; in particular, it is a space of
homogeneous type, in the sense of Definition \ref{Def spazio omogeneo}.
\end{proposition}

A stronger quantitative bound actually holds:

\begin{proposition}
(See \cite[Thm.B]{BBB1} and \cite[Remark 3.9]{BiB1}) There exist positive
constants $c_{1},c_{2}$ such that for every $x\in\mathbb{R}^{n}$ and $R>r>0$
we have:%
\begin{equation}
c_{1}\left(  \frac{R}{r}\right)  ^{n}\leq\frac{\left\vert B_{X}\left(
x,R\right)  \right\vert }{\left\vert B_{X}\left(  x,r\right)  \right\vert
}\leq c_{2}\left(  \frac{R}{r}\right)  ^{q}, \label{bounds on balls}%
\end{equation}
with $q$ as in (\ref{eq.defq}).
\end{proposition}

We can now make precise our $VMO$ assumptions on the coefficients.

\begin{definition}
\label{Def VMO}For any $f\in L_{loc}^{1}\left(  \mathbb{R}^{n}\right)  $ we
define the $VMO$ modulus of $f$ as the function%
\[
\eta_{f}\left(  r\right)  =\sup_{x\in\mathbb{R}^{n},\rho\leq r}\frac
{1}{\left\vert B_{X}\left(  x,\rho\right)  \right\vert }\int_{B_{X}\left(
x,\rho\right)  }\left\vert f\left(  y\right)  -f_{B_{X}\left(  x,\rho\right)
}\right\vert dy,
\]
for any $r>0,$ where, throughout the following, we let%
\[
f_{B}=\frac{1}{\left\vert B\right\vert }\int_{B}f\left(  y\right)  dy.
\]
We say that $f\in BMO_{X}\left(  \mathbb{R}^{n}\right)  $ if $\eta_{f}$ is
bounded; we say that $f\in VMO_{X}\left(  \mathbb{R}^{n}\right)  $ if,
moreover, $\eta_{f}\left(  r\right)  \rightarrow0$ as $r\rightarrow0^{+}$.
\end{definition}

Note that if $f\in L^{\infty}\left(  \mathbb{R}^{n}\right)  $ then obviously
$f\in BMO_{X}\left(  \mathbb{R}^{n}\right)  $ with $\eta_{f}\left(  r\right)
\leq2\left\Vert f\right\Vert _{L^{\infty}\left(  \mathbb{R}^{n}\right)  }$.
Our last assumption on the variable coefficients $a_{ij}$ will be the following:

\begin{itemize}
\item[\textbf{(H.5)}] We ask that the coefficients $a_{ij}$ in
(\ref{operators}) belong to $VMO_{X}\left(  \mathbb{R}^{n}\right)  $.
\end{itemize}

Letting, for any $R>0,$
\begin{equation}
a_{R}^{\sharp}=\max_{i,j=1,...,q}\eta_{a_{ij}}\left(  R\right)  ,
\label{mod VMO coeff}%
\end{equation}
our bounds will depend quantitatively on the coefficients through the function
$a_{R}^{\sharp}$ (and the number $\nu$ in (\ref{ellipticity})).

It is worthwhile noting the following fact:

\begin{proposition}
\label{Prop VMO W1inf}If $f\in W_{X}^{1,\infty}\left(  \mathbb{R}^{n}\right)
$ then $f\in VMO_{X}\left(  \mathbb{R}^{n}\right)  $, and%
\begin{equation}
\eta_{f}\left(  r\right)  \leq cr\sum_{i=1}^{m}\left\Vert X_{i}f\right\Vert
_{L^{\infty}\left(  \mathbb{R}^{n}\right)  }. \label{W1inf VMO}%
\end{equation}

\end{proposition}

\begin{proof}
If $f\in W_{X}^{1,\infty}\left(  \mathbb{R}^{n}\right)  $, in particular $f\in
L^{\infty}\left(  \mathbb{R}^{n}\right)  $ hence $f\in BMO_{X}\left(
\mathbb{R}^{n}\right)  $. To get the bound (\ref{W1inf VMO}), fix a ball
$B_{X}\left(  x,\rho\right)  $. If $f$ is smooth (for instance, $C^{1}\left(
B_{X}\left(  x,2\rho\right)  \right)  $), then by the global Poincar\'{e}
inequality for homogeneous vector fields (see \cite[Thm.7.6]{BBB1}) we can
write%
\begin{align}
&  \frac{1}{\left\vert B_{X}\left(  x,\rho\right)  \right\vert }\int%
_{B_{X}\left(  x,\rho\right)  }\left\vert f\left(  y\right)  -f_{B_{X}\left(
x,\rho\right)  }\right\vert dy\label{Poincare}\\
&  \leq c\rho\frac{1}{\left\vert B_{X}\left(  x,2\rho\right)  \right\vert
}\int_{B_{X}\left(  x,2\rho\right)  }\left(  \sum_{i=1}^{m}\left\vert
X_{i}f\left(  y\right)  \right\vert ^{2}\right)  ^{\frac{1}{2}}dy\nonumber
\end{align}
for some absolute contant $c$. For a general $f\in W_{X}^{1,\infty}\left(
\mathbb{R}^{n}\right)  $, by Theorem \ref{Thm approx} (a) there exists a
sequence $f_{k}\in C_{0}^{\infty}\left(  \mathbb{R}^{n}\right)  $ such that
\[
f_{k}\rightarrow f\text{ in }W_{X}^{1,1}\left(  B_{X}\left(  x,2\rho\right)
\right)
\]
hence for $f\in W_{X}^{1,\infty}\left(  \mathbb{R}^{n}\right)  $ we can still
write (\ref{Poincare}) and get%
\[
\frac{1}{\left\vert B_{X}\left(  x,\rho\right)  \right\vert }\int%
_{B_{X}\left(  x,\rho\right)  }\left\vert f\left(  y\right)  -f_{B_{X}\left(
x,\rho\right)  }\right\vert dy\leq c\rho\sum_{i=1}^{m}\left\Vert
X_{i}f\right\Vert _{L^{\infty}\left(  B_{X}\left(  x,2\rho\right)  \right)
},
\]
which implies (\ref{W1inf VMO}).
\end{proof}

We will use in the following two different kinds of \emph{maximal functions}.

\begin{definition}
\label{Def Maximal}For $f\in L_{loc}^{1}\left(  \mathbb{R}^{n}\right)  $ we
define the \emph{Hardy-Littlewood} (uncentered)\emph{ maximal function} as:
\begin{equation}
\mathcal{M}f\left(  x\right)  =\sup_{\substack{B_{X}\left(  \overline
{x},r\right)  \ni x\\\overline{x}\in\mathbb{R}^{n},r>0}}\frac{1}{\left\vert
B_{X}\left(  \overline{x},r\right)  \right\vert }\int_{B_{X}\left(
\overline{x},r\right)  }\left\vert f\left(  y\right)  \right\vert dy.
\label{maximal HL}%
\end{equation}

\end{definition}

Since $\left(  \mathbb{R}^{n},d_{X},\left\vert \cdot\right\vert \right)  $ is
a space of homogeneous type, by \cite[Thm.2.1]{CW}, , we have:

\begin{theorem}
\label{Thm Maximal}For every $p\in(1,\infty]$ there exists $c>0$ such that for
every $f\in L^{p}\left(  \mathbb{R}^{n}\right)  $
\begin{equation}
\left\Vert \mathcal{M}f\right\Vert _{L^{p}\left(  \mathbb{R}^{n}\right)  }\leq
c_{p}\left\Vert f\right\Vert _{L^{p}\left(  \mathbb{R}^{n}\right)  }.
\label{HL ineq}%
\end{equation}

\end{theorem}

Another kind of maximal function, which can be introduced in any space of
homogeneous type, is the following:

\begin{definition}
\label{Def sharp}For $f\in L_{loc}^{1}\left(  \mathbb{R}^{n}\right)  $,
$x\in\mathbb{R}^{n}$, we define the \emph{sharp maximal function }of $f$ as:%
\begin{equation}
f^{\#}\left(  x\right)  =\sup_{\substack{B_{X}\left(  \overline{x},r\right)
\ni x\\\overline{x}\in\mathbb{R}^{n},r>0}}\frac{1}{\left\vert B_{X}\left(
\overline{x},r\right)  \right\vert }\int_{B_{X}\left(  \overline{x},r\right)
}\left\vert f\left(  y\right)  -f_{B_{X}\left(  \overline{x},r\right)
}\right\vert dy. \label{sharp maximal}%
\end{equation}

\end{definition}

Applying to our context the result proved in \cite[Prop.3.4]{PS} in the
general setting of spaces of homogeneous type of infinite measure, we have the
following result, generalizing the well-known Fefferman-Stein inequality which
holds in Euclidean spaces:

\begin{theorem}
\label{Thm Fefferman Stein}For every $p\in\left(  1,\infty\right)  $ there
exists $C_{p}$ (depending on $p$ and the doubling constant in (\ref{doubling}%
)) such that for every $f\in L^{\infty}\left(  \mathbb{R}^{n}\right)  $, $f$
with bounded support, we have%
\begin{equation}
\left\Vert f\right\Vert _{L^{p}\left(  \mathbb{R}^{n}\right)  }\leq
C_{p}\left\Vert f^{\#}\right\Vert _{L^{p}\left(  \mathbb{R}^{n}\right)  }.
\label{Fefferman Stein}%
\end{equation}

\end{theorem}

We will sometimes need the notion of (compactly supported) H\"{o}lder
continuous function (with respect to the distance $d_{X}$). The definition is
quite natural:

\begin{definition}
\label{Def Holder}We say that a continuous, compactly supported function
$f:\mathbb{R}^{n}\rightarrow\mathbb{R}$ belongs to $C_{0}^{\alpha}\left(
\mathbb{R}^{n}\right)  $ for some $\alpha\in\left(  0,1\right)  $ if%
\[
\left\vert f\right\vert _{\alpha}=\sup_{x\neq y}\frac{\left\vert f\left(
x\right)  -f\left(  y\right)  \right\vert }{d_{X}\left(  x,y\right)  ^{\alpha
}}<\infty.
\]

\end{definition}

\subsection{Fundamental solution \newline of the constant coefficient
operator}

Next, we need to introduce a key object, the fundamental solution of an
operator of kind (\ref{operators}) but with \emph{constant }$a_{ij}$. This
function has been built by Biagi-Bonfiglioli in \cite{BBlift}, and the study
of its properties has been carried out by Biagi-Bonfiglioli-Bramanti in
\cite{BBB1}. In both papers the operator under study is just the sum of
squares
\[
\sum_{i=1}^{q}X_{i}^{2},
\]
i.e. $a_{ij}=\delta_{ij}$. However, as noted in \cite[Remark 2.3 and Section
3]{BiB2} in the context of heat kernels (but the same reasoning applies to
stationary operators), the same construction works for any operator
(\ref{operators}) with constant matrix $A=\left(  a_{ij}\right)  $, and all
the constants appearing in the quantitative bounds involving this fundamental
solution can be uniformly controlled for $A$ ranging in the ellipticity class
defined by (\ref{ellipticity}). In the following, we will recall the known
results that we need, already stated in the suitable uniform form. In the
Appendix at the end of the paper we will give a more detailed justification of
the results stated here below.

The construction of the fundamental solution of $L$ performed in \cite{BBlift}
exploits a \emph{lifting procedure}: the operator $L$ is lifted to a
sublaplacian $\mathcal{L}_{\mathbb{G}}$ on a (strictly higher dimensional)
Carnot group $(\mathbb{G},\ast)$. After an appropriate change of variable
(performed in \cite{BBlift}), one can suppose that the manifold of
$\mathbb{G}$ takes the product form $\mathbb{G}=\mathbb{R}_{x}^{n}%
\times\mathbb{R}_{\xi}^{p}$, with $p=N-n$. Under assumption (H.3), this $p$ is
at least $1$. We are now going to review this result, which describes a
context which will be important throughout the paper. In what follows, we
refer to \cite[Chaps. 3, 6]{BBbook} for the notions of homogeneous Carnot
group, sublaplacian, and its global homogeneous fundamental solution, first
built by Folland in \cite{Fo2}.

\begin{theorem}
\label{Thm lifting}(See \cite[Thms 3.2 and 4.4]{BBlift}). Assume that
$X=\{X_{1},\ldots,X_{m}\}$ satisfies \emph{(H.1)}-\emph{(H.3)}. Then:

\emph{(1).} There exist a homogeneous Carnot group $\mathbb{G}=(\mathbb{R}%
^{N},\ast,D_{\lambda})$ of homogeneous dimension $Q>q$ and a system
$\{\widetilde{X}_{1},\ldots,\widetilde{X}_{m}\}$ of Lie-generators of
$\mathrm{Lie}(\mathbb{G})$ such that $\widetilde{X}_{i}$ is a lifting of
$X_{i}$ for every $i=1,\ldots,m$; by this we mean that
\begin{equation}
\widetilde{X}_{i}(x,\xi)=X_{i}(x)+R_{i}(x,\xi), \label{lifting}%
\end{equation}
where $R_{i}(x,\xi)$ is a smooth vector field operating only in the variable
$\xi\in\mathbb{R}^{p}$, with coefficients possibly depending on $(x,\xi)$. In
particular, the $\widetilde{X}_{i}$'s are $D_{\lambda}$-homogeneous of degree
$1$. \medskip The dilations in $\mathbb{G}$ have the following structure:
\begin{equation}
D_{\lambda}:\mathbb{R}^{N}=\mathbb{R}^{n}\times\mathbb{R}^{p}\longrightarrow
\mathbb{R}^{N},\qquad D_{\lambda}(x,\xi)=(\delta_{\lambda}(x),E_{\lambda}%
(\xi)), \label{Dilataz RN}%
\end{equation}
where $E_{\lambda}(\xi)=(\lambda^{\tau_{1}}\xi,\ldots,\lambda^{\tau_{p}}%
\xi_{p})$ for suitable integers $1\leq\tau_{1}\leq\cdots\leq\tau_{p}$.

\emph{(2).} If $A=\left(  a_{ij}\right)  $ is a constant matrix satisfying
(\ref{ellipticity}) and $\widetilde{\Gamma}_{A}$ is the \emph{(}unique\emph{)}
smooth fundamental solution of $\widetilde{L}_{A}=\sum_{i,j=1}^{m}%
a_{ij}\widetilde{X}_{i}\widetilde{X}_{j}$ vanishing at infinity
con\-struc\-ted in \cite{Fo2}, then
\[
L_{A}=\sum_{i,j=1}^{m}a_{ij}X_{i}X_{j}%
\]
admits a global fundamental solution $\Gamma_{A}(x;y)$ of the form
\begin{equation}
\Gamma_{A}(x;y)=\int_{\mathbb{R}^{p}}\widetilde{\Gamma}_{A}\left(
(x,0);(y,\eta)\right)  \,d\eta\ (\text{for $x\neq y$ in $\mathbb{R}^{n}$}).
\label{sec.one:mainThm_defGamma}%
\end{equation}
This means that the map $y\mapsto\Gamma_{A}(x;y)$ is locally integrable on
$\mathbb{R}^{n}$ and%
\begin{equation}
\int_{\mathbb{R}^{n}}\Gamma_{A}(x;y)\,L_{A}\varphi(y)\,dy=-\varphi
(x)\qquad\text{for every $\varphi\in C_{0}^{\infty}(\mathbb{R}^{n})$ and every
$x\in\mathbb{R}^{n}$.} \label{repr formula 0}%
\end{equation}
Furthermore, setting $\Gamma_{A,\mathbb{G}}(\cdot):=\widetilde{\Gamma}%
_{A}(0;\cdot)$, the integrand in (\ref{sec.one:mainThm_defGamma}) takes the
convolution form
\begin{equation}
\widetilde{\Gamma}_{A}\left(  (x,0);(y,\eta)\right)  =\Gamma_{A,\mathbb{G}%
}\left(  (x,0)^{-1}\ast(y,\eta)\right)  , \label{sec.one:mainThm_defGamma2}%
\end{equation}
valid for any $x\neq y$, so that (\ref{sec.one:mainThm_defGamma}) becomes
\begin{equation}
\Gamma_{A}(x;y)=\int_{\mathbb{R}^{p}}\Gamma_{A,\mathbb{G}}\left(
(x,0)^{-1}\ast(y,\eta)\right)  \,d\eta\qquad(\text{for $x\neq y$ in
$\mathbb{R}^{n}$}). \label{sec.one:mainThm_defGamma22222}%
\end{equation}

\emph{(3).} $\Gamma_{A}$ is smooth out of the diagonal; it is symmetric in
$x,y$, strictly positive, locally integrable on $\mathbb{R}^{n}\times
\mathbb{R}^{n}$, vanishes when $x$ or $y$ go to infinity, and it is jointly
homogeneous of degree $2-q<0$, i.e.,
\begin{equation}
\Gamma_{A}\left(  \delta_{\lambda}(x);\delta_{\lambda}(y)\right)
=\lambda^{2-q}\,\Gamma(x,y),\qquad x\neq y,\,\,\lambda>0.
\label{sec.one:mainThm_defGamma3}%
\end{equation}

\end{theorem}

\begin{notation}
Note that, throughout the paper, $q$ will always denote the homogeneous
dimension of $\mathbb{R}^{n}$, related to our original set of vector fields
$X_{i}$ (see (\ref{eq.defq})), while $Q$ will always denote the homogeneous
dimension of the Carnot group $\mathbb{G}$ obtained by the lifting procedure
described in Theorem \ref{Thm lifting}. In particular, $Q>q.$
\end{notation}

The next two theorems collects the fundamental properties and (uniform)
estimates proved in \cite{BBB1} for $\Gamma_{A}$.

\begin{theorem}
\label{th.teoremone}(See \cite[Thm.1.3]{BBB1}) Assume that $X=\{X_{1}%
,\ldots,X_{m}\}$ satisfies \emph{(H.1)}-\emph{(H.3)}, the constant matrix
$A=\left(  a_{ij}\right)  $ satisfies (\ref{ellipticity}) and let $\Gamma_{A}$
and $\Gamma_{A,\mathbb{G}}$ be as in Theorem \ref{Thm lifting}. Then:\medskip

\emph{(I).} For any $s,t\geq1$, and any choice of $i_{1},\ldots,i_{s}%
,j_{1},\ldots,j_{t}\in\{1,\ldots,m\}$, we have the following representation
formulas for the $X$-derivatives of $\Gamma_{A}$ \emph{(}holding true for
$x\neq y$ in $\mathbb{R}^{n}$\emph{)}:
\begin{align*}
X_{j_{1}}^{x}\cdots X_{j_{t}}^{x}\left(  \Gamma_{A}(\cdot;y)\right)  (x)  &
=\int_{\mathbb{R}^{p}}\left(  \widetilde{X}_{j_{1}}\cdots\widetilde{X}_{j_{t}%
}\Gamma_{A,\mathbb{G}}\right)  \left(  (y,0)^{-1}\ast(x,\eta)\right)
\,d\eta\,;\\[0.2cm]
X_{i_{1}}^{y}\cdots X_{i_{s}}^{y}\left(  \Gamma_{A}(x;\cdot)\right)  (y)  &
=\int_{\mathbb{R}^{p}}\left(  \widetilde{X}_{i_{1}}\cdots\widetilde{X}_{i_{s}%
}\Gamma_{A,\mathbb{G}}\right)  \left(  (x,0)^{-1}\ast(y,\eta)\right)
\,d\eta\,;
\end{align*}%
\begin{align*}
&  X_{j_{1}}^{x}\cdots X_{j_{t}}^{x}X_{i_{1}}^{y}\cdots X_{i_{s}}^{y}%
\Gamma_{A}(x;y)\\
&  =\int_{\mathbb{R}^{p}}\left(  \widetilde{X}_{j_{1}}\cdots\widetilde{X}%
_{j_{t}}\left(  \left(  \widetilde{X}_{i_{1}}\cdots\widetilde{X}_{i_{s}}%
\Gamma_{A,\mathbb{G}}\right)  \circ\iota\right)  \right)  \left(
(y,0)^{-1}\ast(x,\eta)\right)  \,d\eta\,.
\end{align*}
Here $\iota$ denotes the inversion map of the Lie group $\mathbb{G}$.\medskip

\emph{(II).} For any integer $r\geq1$ there exists $C_{r}>0$ such that
\[
\left\vert Z_{1}\cdots Z_{r}\Gamma_{A}(x;y)\right\vert \leq C_{r}\,\frac
{d_{X}(x,y)^{2-r}}{\left\vert B_{X}(x,d_{X}(x,y))\right\vert },
\]
for any $x,y\in\mathbb{R}^{n}$ \emph{(}with $x\neq y$\emph{)} and any choice
of
\[
Z_{1},\ldots,Z_{r}\in\left\{  X_{1}^{x},\ldots,X_{m}^{x},X_{1}^{y}%
,\ldots,X_{m}^{y}\right\}  .
\]
In particular, for every fixed $x\in\mathbb{R}^{n}$ we have
\[
\lim_{|y|\rightarrow\infty}Z_{1}\cdots Z_{r}\Gamma_{A}(x;y)=0.
\]
The constant $C_{r}$ depends on the matrix $A$ only through the number $\nu$
in (\ref{ellipticity}).
\end{theorem}

In the next section we will establish representation formulas for $X_{i}%
X_{j}u$ in terms of $L_{A}u$; these will involve the second derivatives of the
fundamental solution $\Gamma_{A}$. The \emph{singular kernel}%
\[
K\left(  x,y\right)  =X_{i}^{x}X_{j}^{x}\Gamma_{A}\left(  x;y\right)
\]
(letting $i,j=1,...,m$ and $A$ implicitly understood in the symbol $K$) will
play a crucial role in the following. The next result collects some important
properties already known.

\begin{theorem}
\label{Thm prop sing kern}(See \cite[Thm. 8.1]{BBB1}) Under the same
assumptions of Theorem \ref{th.teoremone}, for some fixed $i,j\in
\{1,2,\ldots,m\}$, let
\[
K(x,y):=X_{i}^{x}X_{j}^{x}\Gamma_{A}(x;y)\qquad\text{for $x,y\in\mathbb{R}%
^{n}$, $x\neq y$.}%
\]
There exist constants $\boldsymbol{A},\boldsymbol{B},\boldsymbol{C}>0$ such that:

\begin{enumerate}
\item[\emph{(i)}] for every $x,y\in\mathbb{R}^{n}$ \emph{(}$x\neq y$\emph{)}
one has
\[
\left\vert K(x,y)\right\vert +\left\vert K(y,x)\right\vert \leq\frac
{\boldsymbol{A}}{\left\vert B(x,d(x,y))\right\vert };
\]

\item[\emph{(ii)}] for every $x,x_{0},y\in\mathbb{R}^{n}$ such that
$d(x_{0},y)\geq2\,d(x_{0},x)>0$, it holds
\[
\left\vert K(x,y)-K(x_{0},y)\right\vert +\left\vert K(y,x)-K(y,x_{0}%
)\right\vert \leq\boldsymbol{B}\,\frac{d(x_{0},x)}{d(x_{0},y)}\cdot\frac
{1}{\left\vert B(x_{0},d(x_{0},y))\right\vert };
\]

\item[\emph{(iii)}] for every $z\in\mathbb{R}^{n}$ and $0<r<R<\infty$, one
has
\[
\left\vert \int_{\{r<d(z,y)<R\}}K(z,y)dy\right\vert +\left\vert \int%
_{\{r<d(z,x)<R\}}K(x,z)dx\right\vert \leq\boldsymbol{C}.
\]
The constants $\boldsymbol{A},\boldsymbol{B},\boldsymbol{C}$ depend on the
matrix $A$ only through the number $\nu$ in (\ref{ellipticity}).
\end{enumerate}
\end{theorem}

Inequalities (i)-(ii) are usually called the \emph{standard estimates} of
singular kernels, while (iii) is a kind of \emph{cancelation property}, and
will be crucial in order to give sense to the principal value integral with
kernel $K$. These three properties are one of the possible sets of reasonable
assumptions to prove that the singular integral operator with kernel $K$ is
continuous on $L^{p}(\mathbb{R}^{n})$ for every $p\in(1,\infty)$.

\bigskip

In the following, $d_{\widetilde{X}}$ will stand for the CC-distance
associated with the system of vector fields $\widetilde{X}=\{\widetilde{X}%
_{1},\ldots,\widetilde{X}_{m}\}$ introduced in point (1) of Theorem
\ref{Thm lifting}. Since the $\widetilde{X}_{j}$'s lift the $X_{j}$'s (and all
these vector fields are homogeneous with respect to appropriate dilations) it
is known that (here, $\pi_{n}$ is the projection of $\mathbb{R}^{N}%
=\mathbb{R}^{n}\times\mathbb{R}^{p}$ onto $\mathbb{R}^{n}$)
\[%
\begin{split}
&  d_{X}(x,y)\leq d_{\widetilde{X}}\left(  (x,\xi),(y,\eta)\right)
,\quad\text{for every $x,y\in\mathbb{R}^{n}$ and $\xi,\eta\in\mathbb{R}^{p}$%
},\\
&  \pi_{n}\left(  B_{\widetilde{X}}\left(  (x,\xi),r\right)  \right)
=B_{X}(x,r),\quad\text{for every $x\in\mathbb{R}^{n}$, $\xi\in\mathbb{R}^{p}$
and $r>0$}.
\end{split}
\]
Furthermore, since the $\widetilde{X}_{j}$'s are left-invariant on the group
$\mathbb{G}=(\mathbb{R}^{N},\ast)$, one has
\[
d_{\widetilde{X}}(z,z^{\prime})=d_{\widetilde{X}}\left(  0,z^{-1}\ast
z^{\prime}\right)  ,\quad\text{for every $z,z^{\prime}\in\mathbb{G}$}.
\]
By an abuse of notation, we shall systematically denote $d_{\widetilde{X}%
}(0,\cdot)$ by $d_{\widetilde{X}}$. Simple arguments on ho\-mo\-ge\-neous
groups also show that
\begin{equation}
\left\vert B_{\widetilde{X}}(z,r)\right\vert =\omega_{Q}\,r^{Q},\quad
\text{where $\omega_{Q}=\left\vert B_{\widetilde{X}}(0,1)\right\vert $}.
\label{eq.misurapallehom}%
\end{equation}

The function $d_{\widetilde{X}}$ has the advantage of being a distance in
$\mathbb{R}^{N}$. On the other hand, we do not have at our disposal regularity
results on the function $d_{\widetilde{X}}\left(  u\right)  $. Hence, in order
to build \emph{smooth }radial cutoff functions, we also need to use a
homogeneous norm \emph{equivalent }to $d_{\widetilde{X}}$. For $u\in
\mathbb{R}^{N}$, let
\begin{equation}
\left\Vert u\right\Vert =\left(  \sum_{i=1}^{N}\left\vert u_{i}\right\vert
^{2\alpha!/\alpha_{i}}\right)  ^{1/2\alpha!} \label{hom norm}%
\end{equation}
where%
\[
\alpha_{i}=\left\{
\begin{tabular}
[c]{ll}%
$\sigma_{i}$ & for $i=1,2,...,n$\\
$\tau_{i-n}$ & for $i=n+1,...,n+p$%
\end{tabular}
\ \ \ \ \right.  \text{and }\alpha=\max_{i}\alpha_{i}.
\]
The function $u\mapsto\left\Vert u\right\Vert $ is a \emph{homogeneous norm},
smooth in $\mathbb{R}^{N}\setminus\left\{  0\right\}  $, and it is equivalent
to $d_{\widetilde{X}}\left(  u\right)  $ (see \cite[Chap.3]{BBbook}): there
exists two constants $\gamma_{2}>\gamma_{1}>0$ such that%
\begin{equation}
\gamma_{1}\left\Vert u\right\Vert \leq d_{\widetilde{X}}\left(  u\right)
\leq\gamma_{2}\left\Vert u\right\Vert \text{ for every }u\in\mathbb{R}^{N}.
\label{equivalence}%
\end{equation}
We recall that a homogeneous norm on a Carnot group $\left(  \mathbb{R}%
^{N},\ast,D\left(  \lambda\right)  \right)  $ is a continuous function%
\[
\left\Vert \cdot\right\Vert :\mathbb{R}^{N}\rightarrow\lbrack0,+\infty)
\]
such that, for some constant $c>0$ and every $u,v\in\mathbb{R}^{N},$%
\[%
\begin{tabular}
[c]{ll}%
$\left(  i\right)  $ & $\left\Vert u\right\Vert =0\Longleftrightarrow x=0$\\
$\left(  ii\right)  $ & $\left\Vert D\left(  \lambda\right)  u\right\Vert
=\left\vert \lambda\right\vert \left\Vert u\right\Vert $ for every $\lambda
>0$\\
$\left(  iii\right)  $ & $\left\Vert u^{-1}\right\Vert \leq c\left\Vert
u\right\Vert $\\
$\left(  iv\right)  $ & $\left\Vert u\ast v\right\Vert \leq c\left(
\left\Vert u\right\Vert +\left\Vert v\right\Vert \right)  .$%
\end{tabular}
\ \
\]

From general properties of homogeneous functions in Carnot groups we also have
the following vanishing property of $\widetilde{X}_{i}\widetilde{X}_{j}%
\Gamma_{A,\mathbb{G}}$ which will be useful several times:

\begin{proposition}
\label{Prop vanishing integral}(See e.g. \cite[Corollary 6.31]{BBbook}). Let
$\left\Vert \cdot\right\Vert $ be a homogeneous norm in the Carnot group
$\mathbb{G}$ and let $\psi$ be any radial continuous function in $\mathbb{G}$,
then:%
\[
\int_{r_{1}<\left\Vert u\right\Vert <r_{2}}\left(  \widetilde{X}%
_{i}\widetilde{X}_{j}\widetilde{\Gamma}\cdot\psi\right)  \left(  u\right)
du=0
\]
for every $r_{2}>r_{1}>0$, $i,j=1,2,...,m.$
\end{proposition}

\subsection{Cutoff functions and coverings of $\mathbb{R}^{n}$}

A tool that we will need in the following is a family of cutoff functions
adapted to balls of fixed radius and centered anywhere in the space,
satisfying uniform bounds w.r.t. the center of the ball. Since we have at our
disposal dilations but not translations adapted to the vector fields, building
these functions is not a trivial task. We need to recall the following deep
known result on the volume of metric balls. It is a global version of a famous
result by Nagel-Stein-Weinger \cite{NSW}, which holds in our context thanks to
the homogeneity of vector fields.

\begin{theorem}
(See \cite[Thm. C]{BBB1}). There exist constants $\kappa\in(0,1)$ and
$c_{1},c_{2}>0$ such that, for every $x\in\mathbb{R}^{n}$, $\xi\in
\mathbb{R}^{p}$ and $r>0$ one has the following estimates:
\begin{align}
\left\vert \{\eta\in\mathbb{R}^{p}:(y,\eta)\in B_{\widetilde{X}}%
((x,\xi),r)\}\right\vert  &  \leq c_{1}\frac{\left\vert B_{\widetilde{X}%
}((x,\xi),r)\right\vert }{\left\vert B_{X}(x,r)\right\vert },\quad\text{for
all $y\in\mathbb{R}^{n}$},\label{eq.SanchezI}\\[0.3cm]
\left\vert \{\eta\in\mathbb{R}^{p}:(y,\eta)\in B_{\widetilde{X}}%
((x,\xi),r)\}\right\vert  &  \geq c_{2}\frac{\left\vert B_{\widetilde{X}%
}((x,\xi),r)\right\vert }{\left\vert B_{X}(x,r)\right\vert },\quad\text{for
all $y\in B_{X}(x,\kappa\,r)$}. \label{eq.SanchezII}%
\end{align}

\end{theorem}

We then have the following:

\begin{lemma}
[Cutoff functions]\label{Lemma cutoff}There exists $H>1$ and, for every fixed
$R>0$, there exist constants $c_{1},c_{2}>0$ and a family of cutoff functions
in $\mathbb{R}^{n}$, $\left\{  \phi^{x}\left(  \cdot\right)  \right\}
_{x\in\mathbb{R}^{n}}$ such that%
\begin{align*}
\phi^{x}  &  \in C_{0}^{\infty}\left(  B_{HR}\left(  x\right)  \right) \\
\phi^{x}\left(  y\right)   &  \geq c_{1}\text{ in }B_{R}\left(  x\right)
\end{align*}%
\[
\sup_{y}\left\vert \phi^{x}\left(  y\right)  \right\vert +\sum_{i=1}^{m}%
\sup_{y}\left\vert X_{i}\phi^{x}\left(  y\right)  \right\vert +\sum
_{i,j=1}^{m}\sup_{y}\left\vert X_{i}X_{j}\phi^{x}\left(  y\right)  \right\vert
\leq c_{2}.
\]
Here the relevant fact is that the constants $c_{1},c_{2}$ are independent of
$x$ (while they may depend on $R$).
\end{lemma}

\begin{proof}
We pick $H=2\gamma_{2}/\left(  \gamma_{1}\kappa\right)  $, where $\kappa$ is
the constant in the previous theorem and $\gamma_{1},\gamma_{2}$ are the
constants in (\ref{equivalence}), and define%
\[
\psi\left(  u\right)  =\varphi\left(  \left\Vert u\right\Vert \right)
\]
where $\varphi:[0,\infty)\rightarrow\left[  0,1\right]  $ is a smooth
function, $\varphi\left(  t\right)  =1$ for $t\leq R/\left(  \gamma_{1}%
\kappa\right)  $, $\varphi\left(  t\right)  =0$ for $t\geq2R/\left(
\gamma_{1}\kappa\right)  $, and $\left\Vert \cdot\right\Vert $ is the
homogeneous norm defined in (\ref{hom norm}). Then $\psi\in C_{0}^{\infty
}\left(  \widetilde{B}_{HR}\left(  0\right)  \right)  $ and $\psi=1$ in
$\widetilde{B}_{R/\kappa}\left(  0\right)  .$

Let%
\[
\phi^{x}\left(  y\right)  =\left\vert B_{X}\left(  x,R\right)  \right\vert
\int_{\mathbb{R}^{p}}\psi\left(  \left(  x,0\right)  ^{-1}\circ\left(
y,\eta\right)  \right)  d\eta.
\]
Then:%
\[
X_{i}^{y}\phi^{x}\left(  y\right)  =\left\vert B_{X}\left(  x,R\right)
\right\vert \int_{\mathbb{R}^{p}}\left(  \widetilde{X}_{i}\psi\right)  \left(
\left(  x,0\right)  ^{-1}\circ\left(  y,\eta\right)  \right)  d\eta
\]
and this relation can be iterated to higher order derivatives; in particular,
$\phi^{x}\left(  \cdot\right)  \in C^{\infty}\left(  \mathbb{R}^{n}\right)  $.
Moreover, since $\widetilde{d}\left(  \left(  x,0\right)  ,\left(
y,\eta\right)  \right)  \geq d\left(  x,y\right)  $, if $d\left(  x,y\right)
\geq HR$ then $\psi\left(  \left(  x,0\right)  ^{-1}\circ\left(
y,\eta\right)  \right)  =0$ for every $\eta\in\mathbb{R}^{p},$ hence $\phi
^{x}\left(  y\right)  =0$. Therefore%
\[
\phi^{x}\in C_{0}^{\infty}\left(  B_{HR}\left(  x\right)  \right)  .
\]
Also, note that, by (\ref{eq.SanchezI})%
\begin{align*}
0  &  \leq\phi^{x}\left(  y\right)  \leq\left\vert B_{X}\left(  x,R\right)
\right\vert \int_{\left\{  \eta\in\mathbb{R}^{p}:\widetilde{d}\left(  \left(
x,0\right)  ,\left(  y,\eta\right)  \right)  <HR\right\}  }\psi\left(  \left(
x,0\right)  ^{-1}\circ\left(  y,\eta\right)  \right)  d\eta\\
&  \leq\left\vert B_{X}\left(  x,HR\right)  \right\vert \left\vert \left\{
\eta\in\mathbb{R}^{p}:\widetilde{d}\left(  \left(  x,0\right)  ,\left(
y,\eta\right)  \right)  <HR\right\}  \right\vert \\
&  \leq c\left\vert B_{\widetilde{X}}\left(  \left(  x,0\right)  ,HR\right)
\right\vert \leq c\left(  HR\right)  ^{Q}%
\end{align*}
for every $x,y\in\mathbb{R}^{n}$. Analogously%
\begin{align*}
\left\vert X_{i}^{y}\phi^{x}\left(  y\right)  \right\vert  &  \leq\left\vert
B_{X}\left(  x,R\right)  \right\vert \int_{\left\{  \eta\in\mathbb{R}%
^{p}:\widetilde{d}\left(  \left(  x,0\right)  ,\left(  y,\eta\right)  \right)
<HR\right\}  }\left\vert \left(  \widetilde{X}_{i}\psi\right)  \left(  \left(
x,0\right)  ^{-1}\circ\left(  y,\eta\right)  \right)  \right\vert d\eta\\
&  \leq c\left\vert B_{X}\left(  x,HR\right)  \right\vert \left\vert \left\{
\eta\in\mathbb{R}^{p}:\widetilde{d}\left(  \left(  x,0\right)  ,\left(
y,\eta\right)  \right)  <HR\right\}  \right\vert \leq c\left(  HR\right)  ^{Q}%
\end{align*}
and the same reasoning applies to $X_{i}^{y}X_{j}^{y}\phi^{x}$.

Let $d\left(  x,y\right)  <R$, then by (\ref{eq.SanchezII})%
\begin{align*}
\phi^{x}\left(  y\right)   &  \geq\left\vert B_{X}\left(  x,R\right)
\right\vert \int_{\left\{  \eta\in\mathbb{R}^{p}:\widetilde{d}\left(  \left(
x,0\right)  ,\left(  y,\eta\right)  \right)  <R/\kappa\right\}  }\psi\left(
\left(  x,0\right)  ^{-1}\circ\left(  y,\eta\right)  \right)  d\eta\\
&  =\left\vert B_{X}\left(  x,R\right)  \right\vert \left\vert \left\{
\eta\in\mathbb{R}^{p}:\widetilde{d}\left(  \left(  x,0\right)  ,\left(
y,\eta\right)  \right)  <R/\kappa\right\}  \right\vert \\
&  \geq c\left\vert B_{\widetilde{X}}\left(  \left(  x,0\right)  ,R\right)
\right\vert \geq cR^{Q}.
\end{align*}
The constant $cR^{Q}$ is independent of $x$, so we are done.
\end{proof}

We will need also a suitable covering of $\mathbb{R}^{n}$ by metric balls with
fixed small radii. The required property can be established in every space of
homogeneous type.

\begin{proposition}
[bounded overlapping]\label{Prop overlapping}Let $\left(  X,d,\mu\right)  $ be
a space of homogeneous type. For every $R>0$ there exists a family of balls
$\mathcal{B=}\left\{  B\left(  x_{\alpha},R\right)  \right\}  _{\alpha\in A}$
such that:

(i) $%
{\displaystyle\bigcup\limits_{_{\alpha\in A}}}
B\left(  x_{\alpha},R\right)  =X$

(ii) for every $H>1$, the family of dilated balls%
\[
\mathcal{B}^{H}=\left\{  B\left(  x_{\alpha},HR\right)  \right\}  _{\alpha\in
A}%
\]
has the bounded overlapping property, that is: there exists a constant $N$
(depending on $H$ and the constants of $X$, but independent of $R$) such that
every point of $X$ belongs to at most $N$ balls $B\left(  x_{\alpha
},HR\right)  $.
\end{proposition}

\begin{proof}
To simplify the notation, this proof is written assuming that $d$ is a
distance. If $d$ is a general quasidistance, some inessential changes have to
be made in the constants.

For a fixed number $r>0$, let $\mathcal{F}_{r}$ be the collection of all the
families of pairwise disjoint balls of radius $r$ in $X$. The set
$\mathcal{F}_{r}$ is partially ordered by the set inclusion, hence by Zorn's
lemma it possesses a maximal element. This means that there exists a family
$\mathcal{B}_{0}\mathcal{=}\left\{  B\left(  x_{\alpha},r\right)  \right\}
_{\alpha\in A}$ of pairwise disjoint balls in $X$, such that for every $x\in
X$ the ball $B\left(  x,r\right)  $ intersects at least one ball $B\left(
x_{\alpha},r\right)  $. Therefore every $x\in X$ has distance $<2r$ from some
$x_{\alpha}$, so that, taking $r=R/2$,%
\[
\mathcal{B}=\left\{  B\left(  x_{\alpha},R\right)  \right\}  _{\alpha\in A}%
\]
is a covering of $X$.

Now, fix a number $H>1$ and for any $x\in X$ let us compute an upper bound for
the number $N_{x}$ of balls $\left\{  B\left(  x_{\alpha},HR\right)  \right\}
_{\alpha\in A_{x}}$ containing $x$ (with $A_{x}\subset A$). For every ball
$B\left(  x_{\alpha},HR\right)  $ of this family we have, by the triangle
inequality:%
\[
B\left(  x_{\alpha},HR\right)  \subset B\left(  x,2HR\right)  \subset B\left(
x_{\alpha},3HR\right)  .
\]
Hence%
\[
B\left(  x,2HR\right)  \supset%
{\textstyle\bigcup\limits_{\alpha\in A_{x}}}
B\left(  x_{\alpha},HR\right)  \supset%
{\textstyle\bigcup\limits_{\alpha\in A_{x}}}
B\left(  x_{\alpha},R/2\right)
\]
and since the balls $B\left(  x_{\alpha},R/2\right)  $ are pairwise disjoint,%
\[
\mu\left(  B\left(  x,2HR\right)  \right)  \geq%
{\textstyle\sum\limits_{\alpha\in A_{x}}}
\mu\left(  B\left(  x_{\alpha},R/2\right)  \right)
\]
by the doubling condition%
\[
\geq c\left(  \mu,H\right)
{\textstyle\sum\limits_{\alpha\in A_{x}}}
\mu\left(  B\left(  x_{\alpha},3HR\right)  \right)  \geq c\left(
\mu,H\right)
{\textstyle\sum\limits_{\alpha\in A_{x}}}
\mu\left(  B\left(  x,2HR\right)  \right)  .
\]
We have proved that
\[
c\left(  \mu,H\right)
{\textstyle\sum\limits_{\alpha\in A_{x}}}
\mu\left(  B\left(  x,2HR\right)  \right)  \leq\mu\left(  B\left(
x,2HR\right)  \right)  <\infty
\]
so we conclude that $\operatorname{card}\left(  A_{x}\right)  =N_{x}$ is
finite and%
\[
N_{x}\leq\frac{1}{c\left(  \mu,H\right)  },
\]
independently from $x$ and $R.$
\end{proof}

\section{Operators with constant coefficients\label{sec constant}}

Throughout this section we keep considering operators (\ref{operators}) with
\emph{constant }coefficients $a_{ij}$, satisfying assumptions (H.1)-(H.4),
that is:%
\[
L_{A}u=\sum_{i,j=1}^{m}a_{ij}X_{i}X_{j}u.
\]

Our aim is twofold:

(1) to prove a representation formula for $X_{i}X_{j}u$ in terms of $L_{A}u$,
holding true for $u\in C_{0}^{\infty}\left(  \mathbb{R}^{n}\right)  $,
involving a suitable singular integral operator in $\mathbb{R}^{n}$;

(2) to prove $L^{p}\left(  \mathbb{R}^{n}\right)  $ continuity of this
singular integral operator, obtaining $L^{p}\left(  \mathbb{R}^{n}\right)  $
estimates for $X_{i}X_{j}u$, with constant depending on the coefficients
$a_{ij}$ only through the number $\nu$ appearing in (\ref{ellipticity}).

We wish to stress that, in view of the overall strategy of this paper, points
(1)-(2) are necessary steps. If we were interested just in the validity of
$L^{p}\left(  \mathbb{R}^{n}\right)  $ estimates for $X_{i}X_{j}u$ in terms of
$L_{A}u$, we could derive them almost immediately from the result in
\cite{BBB2}. But this would not allow to extend the result to variable
coefficient operators.

\subsection{Representation formulas for $X_{i}X_{j}u$\label{sec repr form}}

The next result contains a first form of the representation formula for second
derivatives $X_{i}X_{j}u$.

\begin{theorem}
\label{Thm repr formula 1}For $u\in C_{0}^{\infty}\left(  \mathbb{R}%
^{n}\right)  ,$ let us write%
\[
u\left(  x\right)  =-\int_{\mathbb{R}^{n}}\Gamma_{A}\left(  x,y\right)
f\left(  y\right)  dy
\]
with $f=L_{A}u$ as in (\ref{repr formula 0}). Then, for $i,j=1,2,...,m,$ and
every $x\in\mathbb{R}^{n}$,%
\begin{align}
&  X_{i}X_{j}u\left(  x\right) \label{repr form XX}\\
&  =-\lim_{\varepsilon\rightarrow0}\int_{\widetilde{d}\left(  \left(
y,\eta\right)  ,\left(  x,0\right)  \right)  \geq\varepsilon}f\left(
y\right)  \widetilde{X}_{i}\widetilde{X}_{j}\widetilde{\Gamma}_{A}\left(
\left(  y,\eta\right)  ^{-1}\ast\left(  x,0\right)  \right)  d\eta
dy+c_{ij}^{A}f\left(  x\right) \nonumber
\end{align}
where the limit is uniform and the constants $c_{ij}^{A}$ are given by
\[
c_{ij}^{A}=\int_{\left\{  u\in\mathbb{R}^{N}:\left\Vert u\right\Vert
=1\right\}  }\left(  \widetilde{X}_{j}\widetilde{\Gamma}_{A}\right)  \left(
\widetilde{X}_{i}\cdot\nu\right)  d\sigma\left(  u\right)
\]
where $\nu$ is the outer normal to the surface $\left\{  \left\Vert
u\right\Vert =1\right\}  $, and satisfy a uniform bound%
\[
\left\vert c_{ij}^{A}\right\vert \leq c
\]
for some constant $c$ depending on $A$ only through the number $\nu$.

Moreover, the limit for $\varepsilon\rightarrow0$ in the representation
formula exists uniformly in $x$, for every $f\in C_{0}^{\alpha}\left(
\mathbb{R}^{n}\right)  $ (see Definition \ref{Def Holder}).
\end{theorem}

For the proof we need the following know result:

\begin{lemma}
(See \cite[Lemma 4.5]{BBB1}). \label{Lemma diffeomorfismo}For any fixed
$x,y\in\mathbb{R}^{n}$, the change of variables $\eta\mapsto\zeta$ in
$\mathbb{R}^{p}$ implicitly defined by the identity%
\begin{equation}
(x,0)^{-1}\ast(y,\zeta)=(x,\eta)^{-1}\ast(y,0)
\label{eq.tousechangeofvariables}%
\end{equation}
has jacobian determinant of modulus $1.$
\end{lemma}

\bigskip

\begin{proof}
[Proof of Theorem \ref{Thm repr formula 1}]This proof is an adaptation of the
computation which, for Carnot groups, is written in \cite[proof of Thm.
6.33]{BBbook}.

First, we can write%
\begin{equation}
X_{j}u\left(  x\right)  =\int_{\mathbb{R}^{n}}X_{j}\Gamma_{A}\left(
x,y\right)  f\left(  y\right)  dy \label{start repr}%
\end{equation}
since by the growth estimate on $X_{j}\Gamma\left(  x,y\right)  $ the last
integral is absolutely convergent. Next, let us fix a cutoff function
$\omega_{\varepsilon}$ in $\mathbb{R}^{N}$ such that (letting $\left\Vert
\cdot\right\Vert $ be the smooth homogeneous norm defined in (\ref{hom norm}))%
\begin{align*}
\omega_{\varepsilon}\left(  u\right)   &  =1\text{ for }\left\Vert
u\right\Vert \geq\varepsilon\\
\omega_{\varepsilon}\left(  u\right)   &  =0\text{ for }\left\Vert
u\right\Vert \leq\frac{\varepsilon}{2}\\
\left\vert \widetilde{X}_{i}\omega_{\varepsilon}\right\vert  &  \leq\frac
{c}{\varepsilon},
\end{align*}
actually, let us choose%
\begin{align*}
\omega_{\varepsilon}\left(  u\right)   &  =\omega\left(  D_{1/\varepsilon
}\left(  u\right)  \right)  \text{ with}\\
\omega\left(  u\right)   &  =1\text{ for }\left\Vert u\right\Vert \geq1\\
\omega\left(  u\right)   &  =0\text{ for }\left\Vert u\right\Vert \leq\frac
{1}{2}%
\end{align*}
To simplify notation, throughout this proof we will write $\widetilde{\Gamma}$
in place of $\widetilde{\Gamma}_{A}$.

By (\ref{start repr}) and Theorem \ref{th.teoremone}, point (i), we can write:%
\[
X_{j}u\left(  x\right)  =\int_{\mathbb{R}^{n}}f\left(  y\right)  \left(
\int_{\mathbb{R}^{p}}\widetilde{X}_{j}\widetilde{\Gamma}\left(  (y,0)^{-1}%
\ast(x,\eta)\right)  d\eta\right)  dy
\]
by the change of variables in Lemma \ref{Lemma diffeomorfismo}
\begin{align*}
&  =\int_{\mathbb{R}^{n}}f\left(  y\right)  \left(  \int_{\mathbb{R}^{p}%
}\widetilde{X}_{j}\widetilde{\Gamma}\left(  \left(  y,\eta\right)  ^{-1}%
\ast\left(  x,0\right)  \right)  d\eta\right)  dy\\
&  =\lim_{\varepsilon\rightarrow0}\int_{\mathbb{R}^{n}}f\left(  y\right)
\left(  \int_{\mathbb{R}^{p}}\left(  \omega_{\varepsilon}\widetilde{X}%
_{j}\widetilde{\Gamma}\right)  \left(  \left(  y,\eta\right)  ^{-1}\ast\left(
x,0\right)  \right)  d\eta\right)  dy\\
&  \equiv\lim_{\varepsilon\rightarrow0}A_{\varepsilon}\left(  x\right)  ,
\end{align*}
with uniform convergence. Actually,%
\begin{align*}
\left\vert X_{j}u\left(  x\right)  -A_{\varepsilon}\left(  x\right)
\right\vert  &  \leq\left\vert \int_{\mathbb{R}^{n}}f\left(  y\right)  \left(
\int_{\mathbb{R}^{p}}\left(  \left(  \omega_{\varepsilon}-1\right)
\widetilde{X}_{j}\widetilde{\Gamma}\right)  \left(  \left(  y,\eta\right)
^{-1}\ast\left(  x,0\right)  \right)  d\eta\right)  dy\right\vert \\
&  \leq\sup\left\vert f\right\vert \cdot\int_{\left\{  u\in\mathbb{R}%
^{N}:\left\Vert u\right\Vert <\varepsilon\right\}  }\frac{du}{d_{\widetilde{X}%
}\left(  u\right)  ^{Q-1}}\leq c\varepsilon\cdot\sup\left\vert f\right\vert .
\end{align*}
We will compute%
\[
\lim_{\varepsilon\rightarrow0}X_{i}A_{\varepsilon}\left(  x\right)  ,
\]
showing that the limit is uniform and coincides with the right hand side of
(\ref{repr form XX}). A distributional reasoning then implies that the
identity (\ref{repr form XX}) actually holds. To compute $X_{i}A_{\varepsilon
}\left(  x\right)  $, we exploit twice the change of variables in Lemma
\ref{Lemma diffeomorfismo}:%
\[
A_{\varepsilon}\left(  x\right)  =\int_{\mathbb{R}^{n}}f\left(  y\right)
\left(  \int_{\mathbb{R}^{p}}\left(  \omega_{\varepsilon}\widetilde{X}%
_{j}\widetilde{\Gamma}\right)  \left(  (y,0)^{-1}\ast(x,\eta)\right)
d\eta\right)  dy;
\]%
\begin{align*}
X_{i}A_{\varepsilon}\left(  x\right)   &  =\int_{\mathbb{R}^{n}}f\left(
y\right)  \left(  \int_{\mathbb{R}^{p}}\widetilde{X}_{i}\left(  \omega
_{\varepsilon}\widetilde{X}_{j}\widetilde{\Gamma}\right)  \left(
(y,0)^{-1}\ast(x,\eta)\right)  d\eta\right)  dy\\
&  =\int_{\left\Vert \left(  y,\eta\right)  ^{-1}\ast\left(  x,0\right)
\right\Vert <\varepsilon}f\left(  y\right)  \widetilde{X}_{i}\left(
\omega_{\varepsilon}\widetilde{X}_{j}\widetilde{\Gamma}\right)  \left(
\left(  y,\eta\right)  ^{-1}\ast\left(  x,0\right)  \right)  d\eta dy\\
&  +\int_{\left\Vert \left(  y,\eta\right)  ^{-1}\ast\left(  x,0\right)
\right\Vert \geq\varepsilon}f\left(  y\right)  \widetilde{X}_{i}%
\widetilde{X}_{j}\widetilde{\Gamma}\left(  \left(  y,\eta\right)  ^{-1}%
\ast\left(  x,0\right)  \right)  d\eta dy\\
&  \equiv A_{1}^{\varepsilon}\left(  x\right)  +A_{2}^{\varepsilon}\left(
x\right)  .
\end{align*}
Let us show that for $\varepsilon\rightarrow0$ the function $A_{2}%
^{\varepsilon}\left(  x\right)  $ converges uniformly to some function
$A_{2}\left(  x\right)  .$ For $0<\varepsilon_{1}<\varepsilon_{2}<1$ let us
consider%
\begin{align*}
&  A_{2}^{\varepsilon_{1}}\left(  x\right)  -A_{2}^{\varepsilon_{2}}\left(
x\right)  =\int_{\varepsilon_{1}\leq\left\Vert \left(  y,\eta\right)
^{-1}\ast\left(  x,0\right)  \right\Vert <\varepsilon_{2}}f\left(  y\right)
\widetilde{X}_{i}\widetilde{X}_{j}\widetilde{\Gamma}\left(  \left(
y,\eta\right)  ^{-1}\ast\left(  x,0\right)  \right)  d\eta dy\\
&  =\int_{\varepsilon_{1}\leq\left\Vert \left(  y,\eta\right)  ^{-1}%
\ast\left(  x,0\right)  \right\Vert <\varepsilon_{2}}\left[  f\left(
y\right)  -f\left(  x\right)  \right]  \widetilde{X}_{i}\widetilde{X}%
_{j}\widetilde{\Gamma}\left(  \left(  y,\eta\right)  ^{-1}\ast\left(
x,0\right)  \right)  d\eta dy
\end{align*}
where we have exploited the vanishing property of $\widetilde{X}%
_{i}\widetilde{X}_{j}\widetilde{\Gamma}$ on spherical shells (Proposition
\ref{Prop vanishing integral}). Then, since $f\in C_{0}^{\alpha}\left(
\mathbb{R}^{n}\right)  $ and%
\[
d_{X}\left(  x,y\right)  \leq\widetilde{d}\left(  \left(  y,\eta\right)
,\left(  x,0\right)  \right)  ,
\]
with $\widetilde{d}\left(  \left(  y,\eta\right)  ,\left(  x,0\right)
\right)  $ equivalent to $\left\Vert \left(  y,\eta\right)  ^{-1}\ast\left(
x,0\right)  \right\Vert ,$%
\begin{align*}
\left\vert A_{2}^{\varepsilon_{1}}\left(  x\right)  -A_{2}^{\varepsilon_{2}%
}\left(  x\right)  \right\vert  &  \leq c\left\vert f\right\vert _{\alpha}%
\int_{\varepsilon_{1}\leq\left\Vert \left(  y,\eta\right)  ^{-1}\ast\left(
x,0\right)  \right\Vert <\varepsilon_{2}}\frac{d\left(  x,y\right)  ^{\alpha}%
}{\widetilde{d}\left(  \left(  y,\eta\right)  ,\left(  x,0\right)  \right)
^{Q}}d\eta dy\\
&  \leq c\left\vert f\right\vert _{\alpha}\int_{\left\Vert \left(
y,\eta\right)  ^{-1}\ast\left(  x,0\right)  \right\Vert <\varepsilon_{2}}%
\frac{1}{\widetilde{d}\left(  \left(  y,\eta\right)  ,\left(  x,0\right)
\right)  ^{Q-\alpha}}d\eta dy\\
&  \leq c\left\vert f\right\vert _{\alpha}\varepsilon_{2}^{\alpha}%
\rightarrow0\text{ as }\varepsilon_{2}\rightarrow0.
\end{align*}
Then $A_{2}^{\varepsilon}\left(  x\right)  $ satisfies the Cauchy condition
for uniform convergence.

As to $A_{1}^{\varepsilon}\left(  x\right)  $,%
\begin{align*}
A_{1}^{\varepsilon}\left(  x\right)   &  =\int_{\left\Vert \left(
y,\eta\right)  ^{-1}\ast\left(  x,0\right)  \right\Vert <\varepsilon}\left[
f\left(  y\right)  -f\left(  x\right)  \right]  \widetilde{X}_{i}\left(
\omega_{\varepsilon}\widetilde{X}_{j}\widetilde{\Gamma}\right)  \left(
\left(  y,\eta\right)  ^{-1}\ast\left(  x,0\right)  \right)  d\eta dy\\
&  +f\left(  x\right)  \int_{\left\Vert \left(  y,\eta\right)  ^{-1}%
\ast\left(  x,0\right)  \right\Vert <\varepsilon}\widetilde{X}_{i}\left(
\omega_{\varepsilon}\widetilde{X}_{j}\widetilde{\Gamma}\right)  \left(
\left(  y,\eta\right)  ^{-1}\ast\left(  x,0\right)  \right)  d\eta dy\\
&  \equiv A_{1,1}^{\varepsilon}\left(  x\right)  +f\left(  x\right)
A_{1,2}^{\varepsilon}\left(  x\right)  .
\end{align*}

Next, we claim that $A_{1,1}^{\varepsilon}\left(  x\right)  \rightarrow0$
uniformly as $\varepsilon\rightarrow0$. Actually,%
\begin{align*}
A_{1,1}^{\varepsilon}\left(  x\right)   &  =\int_{\left\Vert \left(
y,\eta\right)  ^{-1}\ast\left(  x,0\right)  \right\Vert <\varepsilon}\left[
f\left(  y\right)  -f\left(  x\right)  \right]  \cdot\\
&  \cdot\left(  \widetilde{X}_{i}\omega_{\varepsilon}\cdot\widetilde{X}%
_{j}\widetilde{\Gamma}+\omega_{\varepsilon}\widetilde{X}_{i}\widetilde{X}%
_{j}\widetilde{\Gamma}\right)  \left(  \left(  y,\eta\right)  ^{-1}\ast\left(
x,0\right)  \right)  d\eta dy.
\end{align*}
Then, since%
\begin{align*}
\widetilde{X}_{i}\omega_{\varepsilon}\left(  u\right)   &  =\frac
{1}{\varepsilon}\left(  \widetilde{X}_{i}\omega\right)  \left(
D_{1/\varepsilon}u\right)  \neq0\text{ only if }\frac{\varepsilon}%
{2}<\left\Vert u\right\Vert <\varepsilon\text{ and}\\
\omega_{\varepsilon}\left(  u\right)   &  \neq0\text{ only if }\frac
{\varepsilon}{2}<\left\Vert u\right\Vert ,
\end{align*}%
\begin{align*}
\left\vert A_{1,1}^{\varepsilon}\left(  x\right)  \right\vert  &  \leq
c\left\vert f\right\vert _{\alpha}\int_{\frac{\varepsilon}{2}<\left\Vert
\left(  y,\eta\right)  ^{-1}\ast\left(  x,0\right)  \right\Vert <\varepsilon
}d_{X}\left(  x,y\right)  ^{\alpha}\cdot\\
&  \cdot\left(  \frac{c}{\varepsilon\widetilde{d}\left(  \left(
y,\eta\right)  ,\left(  x,0\right)  \right)  ^{Q-1}}+\frac{c}{\widetilde{d}%
\left(  \left(  y,\eta\right)  ,\left(  x,0\right)  \right)  ^{Q}}\right)
d\eta dy
\end{align*}
since $d_{X}\left(  x,y\right)  \leq\widetilde{d}\left(  \left(
y,\eta\right)  ,\left(  x,0\right)  \right)  \leq\varepsilon$
\begin{align*}
&  \leq c\left\vert f\right\vert _{\alpha}\int_{\left\Vert \left(
y,\eta\right)  ^{-1}\ast\left(  x,0\right)  \right\Vert <\varepsilon}%
\frac{d\eta dy}{\widetilde{d}\left(  \left(  y,\eta\right)  ,\left(
x,0\right)  \right)  ^{Q-\alpha}}\\
&  =c\left\vert f\right\vert _{\alpha}\int_{\left\Vert u\right\Vert
<\varepsilon}\frac{du}{\left\Vert u\right\Vert ^{Q-\alpha}}\leq c\left\vert
f\right\vert _{\alpha}\varepsilon^{\alpha}\rightarrow0\text{ as }%
\varepsilon\rightarrow0.
\end{align*}

We have now to study
\begin{align*}
A_{1,2}^{\varepsilon}\left(  x\right)   &  =\int_{\left\Vert \left(
y,\eta\right)  ^{-1}\ast\left(  x,0\right)  \right\Vert <\varepsilon
}\widetilde{X}_{i}\left(  \omega_{\varepsilon}\widetilde{X}_{j}%
\widetilde{\Gamma}\right)  \left(  \left(  y,\eta\right)  ^{-1}\ast\left(
x,0\right)  \right)  d\eta dy\\
&  =\int_{\left\Vert u\right\Vert <\varepsilon}\widetilde{X}_{i}\left(
\omega_{\varepsilon}\widetilde{X}_{j}\widetilde{\Gamma}\right)  \left(
u\right)  du=A_{1,2}^{\varepsilon},
\end{align*}
that is, the term $A_{1,2}^{\varepsilon}$ is actually independent of $x$. We
are going to show that it is also independent of $\varepsilon$. Since%
\[
\widetilde{X}_{i}\left(  \omega_{\varepsilon}\widetilde{X}_{j}%
\widetilde{\Gamma}\right)  \left(  u\right)  =\omega\left(  D_{1/\varepsilon
}\left(  u\right)  \right)  \left(  \widetilde{X}_{i}\widetilde{X}%
_{j}\widetilde{\Gamma}\right)  \left(  u\right)  +\frac{1}{\varepsilon}\left(
\widetilde{X}_{i}\omega\right)  \left(  D_{1/\varepsilon}\left(  u\right)
\right)  \widetilde{X}_{j}\widetilde{\Gamma}\left(  u\right)  ,
\]
letting $u=D_{\varepsilon}\left(  v\right)  $ in the integral we get%
\begin{align*}
A_{1,2}^{\varepsilon}  &  =\int_{\frac{1}{2}<\left\Vert v\right\Vert
<1}\left[  \omega\left(  v\right)  \varepsilon^{-Q}\left(  \widetilde{X}%
_{i}\widetilde{X}_{j}\widetilde{\Gamma}\right)  \left(  v\right)  +\frac
{1}{\varepsilon}\left(  \widetilde{X}_{i}\omega\right)  \left(  v\right)
\varepsilon^{1-Q}\widetilde{X}_{j}\widetilde{\Gamma}\left(  v\right)  \right]
\varepsilon^{Q}dv\\
&  =\int_{\frac{1}{2}<\left\Vert v\right\Vert <1}\left[  \omega\widetilde{X}%
_{i}\widetilde{X}_{j}\widetilde{\Gamma}+\widetilde{X}_{i}\omega\cdot
\widetilde{X}_{j}\widetilde{\Gamma}\right]  \left(  v\right)  dv\\
&  =\int_{\frac{1}{2}<\left\Vert v\right\Vert <1}\widetilde{X}_{i}\left(
\omega\widetilde{X}_{j}\widetilde{\Gamma}\right)  \left(  v\right)
dv=A_{1,2},
\end{align*}
that is, the term $A_{1,2}^{\varepsilon}$ is actually independent of
$\varepsilon.$ Hence we have proved (\ref{repr form XX}) with $c_{ij}%
=A_{1,2}.$

Note that%
\begin{align*}
\left\vert c_{ij}\right\vert  &  \leq\int_{\frac{1}{2}<\left\Vert v\right\Vert
<1}\left\vert \widetilde{X}_{i}\omega\right\vert \left\vert \widetilde{X}%
_{j}\widetilde{\Gamma}\right\vert dv+\int_{\frac{1}{2}<\left\Vert v\right\Vert
<1}\left\vert \omega\right\vert \left\vert \widetilde{X}_{i}\widetilde{X}%
_{j}\widetilde{\Gamma}\right\vert dv\\
&  \leq c\left\{  \int_{\frac{1}{2}<\left\Vert u\right\Vert <1}\frac
{1}{\left\Vert u\right\Vert ^{Q-1}}dv+\int_{\frac{1}{2}<\left\Vert
u\right\Vert <1}\frac{1}{\left\Vert u\right\Vert ^{Q}}dv\right\}  =c^{\prime},
\end{align*}
where $c^{\prime}$ depends on the matrix $A$ only through the number $\nu.$ In
the last inequality, we have exploited the estimates (uniform w.r.t. the
matrix $A$)%
\[
\left\vert \widetilde{X}_{j}\widetilde{\Gamma}\left(  u\right)  \right\vert
\leq\frac{c}{\left\Vert u\right\Vert ^{Q-1}};\left\vert \widetilde{X}%
_{i}\widetilde{X}_{j}\widetilde{\Gamma}\left(  u\right)  \right\vert \leq
\frac{c}{\left\Vert u\right\Vert ^{Q}}%
\]
which are actually contained in the \emph{proof }of Theorem \ref{th.teoremone}%
, point (ii).

The value of the constant $A_{1,2}$ can be rewritten more transparently if we
apply the divergence theorem on the (smooth) ellipsoidal shell $\frac{1}%
{2}<\left\Vert u\right\Vert <1$, and exploit the fact that $\omega\left(
v\right)  =1$ for $\left\Vert v\right\Vert =1$ and $\omega\left(  v\right)
=0$ for $\left\Vert v\right\Vert =1/2$, which shows that the constant is also
independent from the cutoff function $\omega$. Writing%
\[
\widetilde{X}_{i}=\sum_{l=1}^{N}a_{il}\left(  w\right)  \partial_{w_{l}}%
\]
and recalling that, by the structure of homogeneous vector fields (see
\cite[\S 1]{BBB2}),%
\[
\widetilde{X}_{i}f\left(  w\right)  =\sum_{l=1}^{N}\partial_{w_{l}}\left(
a_{il}f\right)  \left(  w\right)  ,
\]
we have%

\[
A_{1,2}=\sum_{l=1}^{N}\int_{\left\Vert u\right\Vert =1}\left(  a_{il}%
\widetilde{X}_{j}\widetilde{\Gamma}\right)  \left(  v\right)  \nu_{l}%
d\sigma\left(  v\right)  =\int_{\left\Vert u\right\Vert =1}\left(
\widetilde{X}_{j}\widetilde{\Gamma}\right)  \left(  \widetilde{X}_{i}\cdot
\nu\right)  d\sigma\left(  v\right)
\]
where $\left(  \widetilde{X}_{i}\cdot\nu\right)  \ $is the scalar product
between $\widetilde{X}_{i}$ and outer normal to $\left\{  v:\left\Vert
v\right\Vert =1\right\}  $.
\end{proof}

Next, we need to rewrite the representation formula for second order
derivatives in terms of some singular integral living in the space
$\mathbb{R}^{n}$, involving the singular kernel
\begin{equation}
K_{A}\left(  x,y\right)  =X_{i}^{x}X_{j}^{x}\Gamma_{A}\left(  x;y\right)  ,
\label{sing kernel}%
\end{equation}
whose basic properties have been stated in Proposition
\ref{Thm prop sing kern}. To this aim we need to define, in an unusual way, a
\emph{smoothed version }of $K_{A}$.

Let $\psi_{\varepsilon,R}$ be a family of radial cutoff function in the group
$\mathbb{G}=\mathbb{R}^{N}$:%
\begin{equation}
\psi_{\varepsilon,R}\left(  u\right)  =\phi_{\varepsilon,R}\left(
d_{\widetilde{X}}\left(  u\right)  \right)  \label{cutoff psi}%
\end{equation}
where $\phi_{\varepsilon,R}$ is a compactly supported piecewise linear
function such that%
\begin{align*}
\phi_{\varepsilon,R}\left(  t\right)   &  =1\text{ for }2\varepsilon\leq t\leq
R\\
\phi_{\varepsilon,R}\left(  t\right)   &  =0\text{ for }t\leq\varepsilon\text{
and for }t\geq2R.
\end{align*}
(This time it is not necessary to build the cutoff function using a smooth
homogeneous norm, like in the proof of Theorem \ref{Thm repr formula 1},
because we will not need to differentiate $\psi_{\varepsilon,R}$). Let%
\begin{equation}
K_{\varepsilon,R}^{A}\left(  x,y\right)  =\int_{\mathbb{R}^{p}}\left(
\widetilde{X}_{i}\widetilde{X}_{j}\widetilde{\Gamma}_{A}\cdot\psi
_{\varepsilon,R}\right)  \left(  \left(  y,\eta\right)  ^{-1}\ast\left(
x,0\right)  \right)  d\eta\label{K troncato}%
\end{equation}
(note that $K_{\varepsilon,R}^{A}$ is not just $K_{A}$ times the cutoff
function, but is built saturating $\widetilde{X}_{i}\widetilde{X}%
_{j}\widetilde{\Gamma}_{A}$ times a cutoff function).

Note that%
\begin{align*}
d\left(  x,y\right)   &  >2R\Longrightarrow d_{\widetilde{X}}\left(  \left(
y,\eta\right)  ,\left(  x,0\right)  \right)  >2R\Longrightarrow\\
&  \Longrightarrow\psi_{\varepsilon,R}\left(  \left(  y,\eta\right)  ^{-1}%
\ast\left(  x,0\right)  \right)  =0\Longrightarrow K_{\varepsilon,R}%
^{A}\left(  x,y\right)  =0.
\end{align*}
Instead,
\[
d\left(  x,y\right)  <\varepsilon\text{ does \emph{not} imply }\psi
_{\varepsilon,R}\left(  \left(  y,\eta\right)  ^{-1}\ast\left(  x,0\right)
\right)  =0,
\]
so the kernel $K_{\varepsilon,R}^{A}\left(  x,y\right)  $ does not
\emph{vanish }near the diagonal. Nevertheless, it is absolutely integrable
w.r.t. $x$, because:%
\begin{align}
&  \int_{\mathbb{R}^{n}}\left\vert K_{\varepsilon,R}^{A}\left(  x,y\right)
\right\vert dy\leq\int_{\mathbb{R}^{n}}\int_{\mathbb{R}^{p}}\left\vert \left(
\widetilde{X}_{i}\widetilde{X}_{j}\widetilde{\Gamma}_{A}\cdot\psi
_{\varepsilon,R}\right)  \left(  \left(  y,\eta\right)  ^{-1}\ast\left(
x,0\right)  \right)  \right\vert d\eta dy\nonumber\\
&  =\int_{\mathbb{R}^{N}}\left\vert \left(  \widetilde{X}_{i}\widetilde{X}%
_{j}\widetilde{\Gamma}_{A}\cdot\psi_{\varepsilon,R}\right)  \left(  u\right)
\right\vert du\leq c\int_{\varepsilon<d_{\widetilde{X}}\left(  u\right)
<2R}\frac{du}{d_{\widetilde{X}}\left(  u\right)  ^{Q}}=c\log\left(  \frac
{R}{\varepsilon}\right)  . \label{K loc integr}%
\end{align}

We can now prove the desired representation formula:

\begin{theorem}
\label{Thm repr formula2}For every $u\in C_{0}^{\infty}\left(  \mathbb{R}%
^{n}\right)  $ the following representation formula holds, for $i,j=1,2,...,m$%
\begin{align*}
X_{i}X_{j}u\left(  x\right)   &  =-\lim_{\varepsilon\rightarrow0^{+}}%
\lim_{R\rightarrow+\infty}\int_{\mathbb{R}^{n}}K_{\varepsilon,R}^{A}\left(
x,y\right)  L_{A}u\left(  y\right)  dy+c_{ij}^{A}L_{A}u\left(  x\right) \\
&  \equiv-\lim_{\varepsilon\rightarrow0^{+}}\lim_{R\rightarrow+\infty
}T_{\varepsilon,R}^{A}\left(  L_{A}u\right)  \left(  x\right)  +c_{ij}%
^{A}L_{A}u\left(  x\right)  ,
\end{align*}
where the limit exists locally uniformly. More generally, for every $f\in
C_{0}^{\alpha}\left(  \mathbb{R}^{n}\right)  $ the limit%
\[
\lim_{\varepsilon\rightarrow0^{+}}\lim_{R\rightarrow+\infty}T_{\varepsilon
,R}^{A}\left(  f\right)  \left(  x\right)
\]
exists, locally uniformly. The constants $c_{ij}^{A}$ are those appearing in
Theorem \ref{Thm repr formula 1}.
\end{theorem}

\begin{proof}
We start writing, by (\ref{repr formula 0})%
\[
u\left(  x\right)  =-\int_{\mathbb{R}^{n}}\Gamma_{A}\left(  x;y\right)
L_{A}u\left(  y\right)  dy.
\]
Then, by Theorem \ref{Thm repr formula 1} we have:%
\begin{align*}
X_{i}X_{j}u\left(  x\right)   &  =-\lim_{\varepsilon\rightarrow0}%
\int_{\widetilde{d}\left(  \left(  y,\eta\right)  ,\left(  x,0\right)
\right)  \geq\varepsilon}L_{A}u\left(  y\right)  \widetilde{X}_{i}%
\widetilde{X}_{j}\widetilde{\Gamma}_{A}\left(  \left(  y,\eta\right)
^{-1}\ast\left(  x,0\right)  \right)  d\eta dy\\
&  +c_{ij}^{A}L_{A}u\left(  x\right)  .
\end{align*}
To simplify notation, from now on, throughout this proof, we will drop the
sub- or super-script $A$, simply writing $L,\widetilde{\Gamma},c_{ij}$ instead
of $L_{A},\widetilde{\Gamma}_{A},c_{ij}^{A}.$

Let us consider the operator:%
\[
Tf\left(  x\right)  =\lim_{\varepsilon\rightarrow0}\int_{\widetilde{d}\left(
\left(  y,\eta\right)  ,\left(  x,0\right)  \right)  \geq\varepsilon
}\widetilde{X}_{i}\widetilde{X}_{j}\widetilde{\Gamma}\left(  \left(
y,\eta\right)  ^{-1}\ast\left(  x,0\right)  \right)  f\left(  y\right)  d\eta
dy
\]
defined, as we have seen from the proof of the representation formula, for
$f\in C_{0}^{\alpha}\left(  \mathbb{R}^{n}\right)  $ for any $\alpha\in(0,1]$.
We have proved that this limit is uniform in $x$. We have to show that%
\[
Tf\left(  x\right)  =\lim_{\varepsilon\rightarrow0^{+}}\lim_{R\rightarrow
+\infty}T_{\varepsilon,R}\left(  f\right)  \left(  x\right)  .
\]

Let $\psi_{\varepsilon,R}$ be the cutoff function defined as in
(\ref{cutoff psi}), for $R>\varepsilon>0$, and let%
\begin{align*}
S_{\varepsilon}f\left(  x\right)   &  =\int_{\widetilde{d}\left(  \left(
y,\eta\right)  ,\left(  x,0\right)  \right)  \geq\varepsilon}\widetilde{X}%
_{i}\widetilde{X}_{j}\widetilde{\Gamma}\left(  \left(  y,\eta\right)
^{-1}\ast\left(  x,0\right)  \right)  f\left(  y\right)  d\eta dy\\
&  =\int_{\mathbb{R}^{N}}\left(  \widetilde{X}_{i}\widetilde{X}_{j}%
\widetilde{\Gamma}\cdot\psi_{\varepsilon,R}\right)  \left(  \left(
y,\eta\right)  ^{-1}\ast\left(  x,0\right)  \right)  f\left(  y\right)  d\eta
dy\\
&  +\int_{\widetilde{d}\left(  \left(  y,\eta\right)  ,\left(  x,0\right)
\right)  \geq\varepsilon}\left(  \widetilde{X}_{i}\widetilde{X}_{j}%
\widetilde{\Gamma}\cdot\left(  1-\psi_{\varepsilon,R}\right)  \right)  \left(
\left(  y,\eta\right)  ^{-1}\ast\left(  x,0\right)  \right)  f\left(
y\right)  d\eta dy
\end{align*}%
\begin{align*}
&  \equiv T_{\varepsilon,R}f\left(  x\right) \\
&  +\int_{\varepsilon\leq\widetilde{d}\left(  \left(  y,\eta\right)  ,\left(
x,0\right)  \right)  <2\varepsilon}\left(  \widetilde{X}_{i}\widetilde{X}%
_{j}\widetilde{\Gamma}\cdot\left(  1-\psi_{\varepsilon,R}\right)  \right)
\left(  \left(  y,\eta\right)  ^{-1}\ast\left(  x,0\right)  \right)  f\left(
y\right)  d\eta dy\\
&  +\int_{\widetilde{d}\left(  \left(  y,\eta\right)  ,\left(  x,0\right)
\right)  >R}\left(  \widetilde{X}_{i}\widetilde{X}_{j}\widetilde{\Gamma}%
\cdot\left(  1-\psi_{\varepsilon,R}\right)  \right)  \left(  \left(
y,\eta\right)  ^{-1}\ast\left(  x,0\right)  \right)  f\left(  y\right)  d\eta
dy\\
&  \equiv T_{\varepsilon,R}f\left(  x\right)  +S_{\varepsilon}^{1}f\left(
x\right)  +S_{R}^{2}f\left(  x\right)  .
\end{align*}
Now, we exploit the vanishing property in Proposition
\ref{Prop vanishing integral}:%
\[
\int_{r_{1}\leq d_{\widetilde{X}}\left(  u\right)  <r_{2}}\left(
\widetilde{X}_{i}\widetilde{X}_{j}\widetilde{\Gamma}\cdot\left(
1-\psi_{\varepsilon,R}\right)  \right)  \left(  u\right)  du=0.
\]
Therefore, for every $f\in C_{0}^{\alpha}\left(  \mathbb{R}^{n}\right)  $ we
can write:%
\begin{align*}
&  \left\vert S_{\varepsilon}^{1}f\left(  x\right)  \right\vert =\left\vert
\int_{\varepsilon\leq\widetilde{d}\left(  \left(  y,\eta\right)  ,\left(
x,0\right)  \right)  <2\varepsilon}\left(  \widetilde{X}_{i}\widetilde{X}%
_{j}\widetilde{\Gamma}\cdot\left(  1-\psi_{\varepsilon,R}\right)  \right)
\left(  \left(  y,\eta\right)  ^{-1}\ast\left(  x,0\right)  \right)  f\left(
y\right)  d\eta dy\right\vert \\
&  =\left\vert \int_{\varepsilon\leq\widetilde{d}\left(  \left(
y,\eta\right)  ,\left(  x,0\right)  \right)  <2\varepsilon}\left(
\widetilde{X}_{i}\widetilde{X}_{j}\widetilde{\Gamma}\cdot\left(
1-\psi_{\varepsilon,R}\right)  \right)  \left(  \left(  y,\eta\right)
^{-1}\ast\left(  x,0\right)  \right)  \left[  f\left(  y\right)  -f\left(
x\right)  \right]  d\eta dy\right\vert \\
&  \leq\int_{\varepsilon\leq\widetilde{d}\left(  \left(  y,\eta\right)
,\left(  x,0\right)  \right)  <2\varepsilon}\frac{c}{\widetilde{d}\left(
\left(  y,\eta\right)  ,\left(  x,0\right)  \right)  ^{Q}}\left\vert
f\right\vert _{\alpha}d\left(  x,y\right)  ^{\alpha}dyd\eta\\
&  \leq c\left\vert f\right\vert _{\alpha}\int_{\widetilde{d}\left(  \left(
y,\eta\right)  ,\left(  x,0\right)  \right)  <2\varepsilon}\frac
{1}{\widetilde{d}\left(  \left(  y,\eta\right)  ,\left(  x,0\right)  \right)
^{Q-\alpha}}dyd\eta\leq c\left\vert f\right\vert _{\alpha}\varepsilon^{\alpha
}\rightarrow0
\end{align*}
so that
\[
S_{\varepsilon}^{1}f\left(  x\right)  \rightarrow0\text{ uniformly in }x\text{
as }\varepsilon\rightarrow0.
\]
As to $S_{R}^{2}f$, we can write%
\begin{align*}
\left\vert S_{R}^{2}f\left(  x\right)  \right\vert  &  \leq\int_{\widetilde{d}%
\left(  \left(  y,\eta\right)  ,\left(  x,0\right)  \right)  >R}\frac
{c}{\widetilde{d}\left(  \left(  y,\eta\right)  ,\left(  x,0\right)  \right)
^{Q}}\left\vert f\left(  y\right)  \right\vert d\eta dy\\
&  \leq\frac{c}{R}\int_{\widetilde{d}\left(  \left(  y,\eta\right)  ,\left(
x,0\right)  \right)  >R}\frac{\left\vert f\left(  y\right)  \right\vert
}{\widetilde{d}\left(  \left(  y,\eta\right)  ,\left(  x,0\right)  \right)
^{Q-1}}d\eta dy\\
&  \leq\frac{c}{R}\int_{\mathbb{R}^{n}}\left\vert f\left(  y\right)
\right\vert \left(  \int_{\mathbb{R}^{p}}\frac{d\eta}{\widetilde{d}\left(
\left(  y,\eta\right)  ,\left(  x,0\right)  \right)  ^{Q-1}}\right)  dy\\
&  \leq\frac{c}{R}\int_{\mathbb{R}^{n}}\left\vert f\left(  y\right)
\right\vert \frac{d\left(  x,y\right)  }{\left\vert B\left(  x,d\left(
x,y\right)  \right)  \right\vert }dy.
\end{align*}
If $f$ is compactly supported in $B\left(  0,r_{1}\right)  $ and $x\in
B\left(  0,r_{2}\right)  $
\begin{align*}
\left\vert S_{R}^{2}f\left(  x\right)  \right\vert  &  \leq\frac{c}{R}%
\int_{B\left(  x,r_{1}+r_{2}\right)  }\left\vert f\left(  y\right)
\right\vert \frac{d\left(  x,y\right)  }{\left\vert B\left(  x,d\left(
x,y\right)  \right)  \right\vert }dy\\
&  \leq\frac{c}{R}\left\Vert f\right\Vert _{\infty}\int_{d\left(  x,y\right)
<r_{1}+r_{2}}\frac{d\left(  x,y\right)  }{\left\vert B\left(  x,d\left(
x,y\right)  \right)  \right\vert }dy\\
&  \leq\frac{c}{R}\left\Vert f\right\Vert _{\infty}\left(  r_{1}+r_{2}\right)
\end{align*}
hence $\left\vert S_{R}^{2}f\left(  x\right)  \right\vert \rightarrow0$
locally uniformly as $R\rightarrow\infty$. Hence:%
\[
S_{\varepsilon}f\left(  x\right)  =\lim_{R\rightarrow\infty}T_{\varepsilon
,R}f\left(  x\right)  +S_{\varepsilon}^{1}f\left(  x\right)
\]
and%
\[
Tf\left(  x\right)  =\lim_{\varepsilon\rightarrow0}S_{\varepsilon}f\left(
x\right)  =\lim_{\varepsilon\rightarrow0}\lim_{R\rightarrow\infty
}T_{\varepsilon,R}f\left(  x\right)  ,
\]
where the operator $Tf\ $\ is well defined for every H\"{o}lder continuous
compactly supported $f$, and is also equal to%
\[
\lim_{\varepsilon\rightarrow0}\lim_{R\rightarrow\infty}T_{\varepsilon
,R}\left(  f\right)  \left(  x\right)
\]
for the same class of functions $f$, where the limit exists locally uniformly
in $\mathbb{R}^{n}$.
\end{proof}

\subsection{$L^{p}$ estimates for the constant coefficient operator}

The aim of this section is to prove the following

\begin{theorem}
[$L^{p}$ estimates for $L_{A}$]\label{Thm Lp constant}Keeping the assumptions
stated at the beginning of section \ref{sec constant}, for every $p\in\left(
1,\infty\right)  $ there exists a constant $c>0$, depending on $A$ only
through the number $\nu$ in (\ref{ellipticity}), such that
\[
\left\Vert X_{i}X_{j}u\right\Vert _{L^{p}\left(  \mathbb{R}^{n}\right)  }\leq
c\left\Vert L_{A}u\right\Vert _{L^{p}\left(  \mathbb{R}^{n}\right)  }%
\]
for every $u\in C_{0}^{\infty}\left(  \mathbb{R}^{n}\right)  $.

Moreover, the operator $T$ defined as
\[
Tf\left(  x\right)  =\lim_{\varepsilon\rightarrow0^{+}}\lim_{R\rightarrow
+\infty}T_{\varepsilon,R}\left(  f\right)  \left(  x\right)
\]
for $f\in C_{0}^{\alpha}\left(  \mathbb{R}^{n}\right)  $, can be extended to a
linear continuous operator on $L^{p}\left(  \mathbb{R}^{n}\right)  $, with
norm bounded by the same constant $c$.
\end{theorem}

To prove the above result, we will show the following:

\begin{theorem}
\label{Thm Lp uniform truncated}Under the same assumptions of the previous
theorem, for every $p\in\left(  1,\infty\right)  $ there exists a constant
$c>0$, depending on $A$ only through the number $\nu$ in (\ref{ellipticity}),
such that%
\[
\left\Vert T_{\varepsilon,R}f\right\Vert _{L^{p}\left(  \mathbb{R}^{n}\right)
}\leq c\left\Vert f\right\Vert _{L^{p}\left(  \mathbb{R}^{n}\right)  }%
\]
for every $f\in C_{0}^{\alpha}\left(  \mathbb{R}^{n}\right)  $ (and, by
density, every $f\in L^{p}\left(  \mathbb{R}^{n}\right)  $), every
$R>\varepsilon>0$ (and $c$ independent of $R,\varepsilon$).
\end{theorem}

Let us first show how Theorem \ref{Thm Lp constant} follows from Theorem
\ref{Thm Lp uniform truncated}.

By Fatou's Lemma we have, for every $u\in C_{0}^{\infty}\left(  \mathbb{R}%
^{n}\right)  $ and $1<p<\infty,$%
\begin{align*}
&  \int_{\mathbb{R}^{n}}\lim_{\varepsilon\rightarrow0}\lim_{R\rightarrow
\infty}\left\vert T_{\varepsilon,R}\left(  L_{A}u\right)  \left(  x\right)
\right\vert ^{p}dx\\
&  \leq\liminf_{\varepsilon\rightarrow0}\liminf_{R\rightarrow\infty}%
\int_{\mathbb{R}^{n}}\left\vert T_{\varepsilon,R}\left(  L_{A}u\right)
\left(  x\right)  \right\vert ^{p}dx\leq c\left\Vert L_{A}u\right\Vert
_{p}^{p},
\end{align*}
hence by the representation formula in Theorem \ref{Thm repr formula2}, and
keeping into account the uniform bound on the constants $c_{ij}^{A}$ in
Theorem \ref{Thm repr formula 1}, we conclude%
\[
\left\Vert X_{i}X_{j}u\right\Vert _{L^{p}\left(  \mathbb{R}^{n}\right)  }\leq
c\left\Vert L_{A}u\right\Vert _{L^{p}\left(  \mathbb{R}^{n}\right)  }%
\]
for every $p\in\left(  1,\infty\right)  $ and $u\in C_{0}^{\infty}\left(
\mathbb{R}^{n}\right)  $, with constant $c$ depending on $A$ only through the
number $\nu$ in (\ref{ellipticity}).

We can also conclude that, since $C_{0}^{\infty}\left(  \mathbb{R}^{n}\right)
$ is dense in $L^{p}\left(  \mathbb{R}^{n}\right)  $, the operator $T$, for
which we know that%
\[
Tf\left(  x\right)  =\lim_{\varepsilon\rightarrow0^{+}}\lim_{R\rightarrow
+\infty}T_{\varepsilon,R}\left(  f\right)  \left(  x\right)
\]
for every $f\in C_{0}^{\alpha}\left(  \mathbb{R}^{n}\right)  $ and every $x$,
can be extended to a linear continuous operator
\[
T:L^{p}\left(  \mathbb{R}^{n}\right)  \rightarrow L^{p}\left(  \mathbb{R}%
^{n}\right)  .
\]
So Theorem \ref{Thm Lp constant} follows.

In order to prove Theorem \ref{Thm Lp uniform truncated}, we will apply the
following abstract result in spaces of homogeneous type.

\begin{theorem}
\label{Thm spazio omogeneo}(See \cite[Thm. 4.1]{BCsing}). Let $\left(
X,d,\mu\right)  $ be a space of homogeneous type (see Definition
\ref{Def spazio omogeneo}), $\mu$ a regular measure. Let $K:X\times
X\setminus\left\{  x=y\right\}  \rightarrow\mathbb{R}$ be a kernel satisfying
the following conditions: there exist constants $\boldsymbol{A},\boldsymbol{B}%
,\boldsymbol{C}>0$ such that:

\begin{enumerate}
\item[\emph{(i)}] for every $x,y\in X$ \emph{(}$x\neq y$\emph{)} one has
\[
\left\vert K(x,y)\right\vert +\left\vert K(y,x)\right\vert \leq\frac
{\boldsymbol{A}}{\left\vert B(x,d(x,y))\right\vert };
\]

\item[\emph{(ii)}] for every $x,x_{0},y\in X$ such that $d(x_{0},y)\geq
M\,d(x_{0},x)>0$, it holds
\[
\left\vert K(x,y)-K(x_{0},y)\right\vert +\left\vert K(y,x)-K(y,x_{0}%
)\right\vert \leq\boldsymbol{B}\,\left(  \frac{d(x_{0},x)}{d(x_{0},y)}\right)
^{\beta}\cdot\frac{1}{\left\vert B(x_{0},d(x_{0},y))\right\vert }%
\]
for some $\beta>0$ and some constant $M>1$ such that if $d(x_{0},y)\geq
M\,d(x_{0},x)$, then $d\left(  x_{0},y\right)  $ and $d\left(  x,y\right)  $
are equivalent.

\item[\emph{(iii)}] for every $z\in X$ and $0<r<R<\infty$, one has
\[
\left\vert \int_{\{r<d(z,y)<R\}}K(z,y)d\mu\left(  y\right)  \right\vert
+\left\vert \int_{\{r<d(z,x)<R\}}K(x,z)d\mu\left(  x\right)  \right\vert
\leq\boldsymbol{C}.
\]
Let $K_{\varepsilon}$ (for $\varepsilon>0$) be a \textquotedblleft regularized
kernel\textquotedblright\ defined in such a way that the following properties holds:

\item[(a)] $K_{\varepsilon}\left(  x,\cdot\right)  $ and $K_{\varepsilon
}\left(  \cdot,x\right)  $ are locally integrable for every $x\in X$

\item[(b)] $K_{\varepsilon}$ satisfies the standard estimates (i), (ii), with
constant bounded by $c^{\prime}\left(  \boldsymbol{A}+\boldsymbol{B}\right)  $
and $c^{\prime}$ absolute constant (independent of $\varepsilon$).

\item[(c)] There exists an absolute constant $c^{\prime}$ such that for every
$z\in X$%
\[
\left\vert \int_{\{r<d(z,y)<R\}}K_{\varepsilon}(z,y)d\mu\left(  y\right)
\right\vert +\left\vert \int_{\{r<d(z,x)<R\}}K_{\varepsilon}(x,z)d\mu\left(
x\right)  \right\vert \leq c^{\prime}\boldsymbol{C}.
\]

For every $f\in C_{0}^{\alpha}\left(  X\right)  $, set
\[
T_{\varepsilon}f\left(  x\right)  =\int_{X}K_{\varepsilon}\left(  x,y\right)
f\left(  y\right)  d\mu\left(  y\right)  .
\]
Then, for every $p\in\left(  1,\infty\right)  $, $T_{\varepsilon}$ can be
extended to a linear continuous operator on $L^{p}\left(  X\right)  $, and%
\[
\left\Vert T_{\varepsilon}f\right\Vert _{L^{p}\left(  X\right)  }\leq
c\left\Vert f\right\Vert _{L^{p}\left(  X\right)  }%
\]
for every $f\in L^{p}\left(  X\right)  $, with $c>0$ independent of
$\varepsilon.$ Moreover, the constant $c$ in the last estimate depends on the
quantities involved in the assumptions as follows:%
\[
c\leq c^{\prime}\left(  \boldsymbol{A}+\boldsymbol{B}+\boldsymbol{C}\right)
,
\]
with $c^{\prime}$ \textquotedblleft absolute\textquotedblright\ constant.
\end{enumerate}
\end{theorem}

In the previous statement, \textquotedblleft absolute\textquotedblright%
\ constant means a constant depending on $\left(  X,d,\mu\right)  $ but not on
the kernel $K\left(  x,y\right)  $ (neither on the parameter $\varepsilon$).

\begin{remark}
The previous theorem has been stated and proved in \cite[Thm. 4.1]{BCsing}
defining $K_{\varepsilon}$ in a specific way (as a linear truncation of $K$
near the pole). Since we need to apply this result to smoothed kernels
$K_{\varepsilon}$ defined with a different procedure, in the present statement
we have pointed out explicitly which are the only properties of
$K_{\varepsilon}$ which are actually used in the proof \cite[Thm. 4.1]%
{BCsing}. The present statement is therefore slightly different (and more
general) than that in \cite{BCsing}. On the other hand, this forces us to
check that our smoothed kernels $K_{\varepsilon}$ actually satisfy properties (a)-(b)-(c).
\end{remark}

So, to prove Theorem \ref{Thm Lp uniform truncated} we are left to show that
the kernels $K_{\varepsilon,R}^{A}\left(  x,y\right)  $ defined in
(\ref{K troncato}) satisfy properties (a)-(b)-(c) in Theorem
\ref{Thm spazio omogeneo}.

As to property (a), we have already checked in (\ref{K loc integr}) the local
integrability of $K_{\varepsilon,R}^{A}\left(  x,\cdot\right)  $. To show the
analogous property for $K_{\varepsilon,R}^{A}\left(  \cdot,x\right)  $ we need
to apply Lemma \ref{Lemma diffeomorfismo}. We can write:%

\begin{align*}
\int_{\mathbb{R}^{n}}\left\vert K_{\varepsilon,R}^{A}\left(  y,x\right)
\right\vert dy  &  \leq\int_{\mathbb{R}^{n}}\int_{\mathbb{R}^{p}}\left\vert
\left(  \widetilde{X}_{i}\widetilde{X}_{j}\widetilde{\Gamma}_{A}\cdot
\psi_{\varepsilon,R}\right)  \left(  \left(  x,\eta\right)  ^{-1}\ast\left(
y,0\right)  \right)  \right\vert d\eta dy\\
&  =\int_{\mathbb{R}^{n}}\int_{\mathbb{R}^{p}}\left\vert \left(
\widetilde{X}_{i}\widetilde{X}_{j}\widetilde{\Gamma}_{A}\cdot\psi
_{\varepsilon,R}\right)  \left(  (x,0)^{-1}\ast(y,\zeta)\right)  \right\vert
d\zeta dy\\
&  =\int_{\mathbb{R}^{N}}\left\vert \left(  \widetilde{X}_{i}\widetilde{X}%
_{j}\widetilde{\Gamma}_{A}\cdot\psi_{\varepsilon,R}\right)  \left(  u\right)
\right\vert du\\
&  \leq c\int_{\varepsilon<d_{\widetilde{X}}\left(  u\right)  <2R}\frac
{du}{d_{\widetilde{X}}\left(  u\right)  ^{Q}}=c\log\left(  \frac
{R}{\varepsilon}\right)
\end{align*}
where in the first identity we have performed the change of variables
(\ref{eq.tousechangeofvariables}).

So we have checked property (a) for the kernels $K_{\varepsilon,R}^{A}\left(
x,y\right)  $ (note that this property is \emph{not }uniform w.r.t.
$\varepsilon$). The next proposition states that also property (b) is satisfied.

\begin{proposition}
\emph{(i).} There exists $c>0$ such that for every $R>\varepsilon>0$ and every
$x,y\in\mathbb{R}^{n},x\neq y,$%
\[
\left\vert K_{\varepsilon,R}^{A}\left(  x,y\right)  \right\vert +\left\vert
K_{\varepsilon,R}^{A}\left(  y,x\right)  \right\vert \leq\frac{c}{\left\vert
B\left(  x,d\left(  x,y\right)  \right)  \right\vert }.
\]

\emph{(ii).} There exists $c>0$ such that for every $R>\varepsilon>0$ and
every $x,x_{0},y\in\mathbb{R}^{n}$ such that $d\left(  x_{0},y\right)
\geq2d\left(  x_{0},y\right)  $,%
\begin{align*}
&  \left\vert K_{\varepsilon,R}^{A}\left(  x_{0},y\right)  -K_{\varepsilon
,R}^{A}\left(  x,y\right)  \right\vert +\left\vert K_{\varepsilon,R}%
^{A}\left(  y,x_{0}\right)  -K_{\varepsilon,R}^{A}\left(  y,x\right)
\right\vert \\
&  \leq c\frac{d\left(  x_{0},x\right)  }{d\left(  x_{0},y\right)  \left\vert
B\left(  x_{0},d\left(  x_{0},y\right)  \right)  \right\vert }.
\end{align*}

\emph{(iii). }There exists $c>0$ such that for every $r_{2}>r_{1}%
>0,x\in\mathbb{R}^{n},$%
\[
\left\vert \int_{r_{1}<d\left(  x,y\right)  <r_{2}}K_{\varepsilon,R}%
^{A}\left(  x,y\right)  dy\right\vert +\left\vert \int_{r_{1}<d\left(
x,y\right)  <r_{2}}K_{\varepsilon,R}^{A}\left(  y,x\right)  dy\right\vert \leq
c
\]
with $c$ independent of $\varepsilon$. The constants depend on the matrix $A$
only through the number $\nu$.
\end{proposition}

\begin{proof}
The proof of this Proposition is an easy adaptation of the proof of the
analogous properties of $K\left(  x,y\right)  $ given in \cite[\S 8]{BBB1}. As
in \cite[\S 8]{BBB1}, the proof of (ii) exploits the upper bound on
\emph{third order derivatives} of $\Gamma^{A}$ of the kind $X_{i}^{x}X_{j}%
^{x}X_{k}^{x}\Gamma^{A}\left(  x,y\right)  $ and $X_{i}^{x}X_{j}^{x}X_{k}%
^{y}\Gamma^{A}\left(  x,y\right)  $. Here we implicitly exploit the generality
of the estimates in Theorem \ref{th.teoremone}, (ii). To prove (iii) for
$K_{\varepsilon,R}^{A}$ we also need to use Proposition
\ref{Prop vanishing integral}, to assert that%
\[
\int_{r_{1}<d_{\widetilde{X}}\left(  \left(  y,\eta\right)  ,\left(
x,0\right)  \right)  <r_{2}}\left(  \widetilde{X}_{i}\widetilde{X}%
_{j}\widetilde{\Gamma}\cdot\psi_{\varepsilon,R}\right)  \left(  \left(
y,\eta\right)  ^{-1}\ast\left(  x,0\right)  \right)  d\eta dy=0
\]
for every $r_{2}>r_{1}>0$.
\end{proof}

Summarizing, since the kernels $K_{\varepsilon,R}^{A}$ satisfy assumptions
(a), (b), (c) in Theorem \ref{Thm spazio omogeneo}, we can infer that Theorem
\ref{Thm Lp uniform truncated} holds. Then also Theorem \ref{Thm Lp constant}
is completely proved.

\subsection{Estimates on the mean oscillation of $X_{i}X_{j}u$ in terms of
$L_{A}u$}

Throughout this section we will show how, combining the representation formula
in Theorem \ref{Thm repr formula2}, the $L^{p}$-continuity result in Theorem
\ref{Thm Lp constant}, and some real analysis in spaces of homogeneous type,
we can prove an estimate on the mean oscillation of $X_{i}X_{j}u$ over balls,
in terms of $L_{A}u$, which will be the key point to extend to operators with
variable coefficients $a_{ij}\left(  x\right)  $ the $L^{p}$ estimates proved.
The result is the following:

\begin{theorem}
\label{Thm Krylov main step}Under the assumptions (H.1)-(H.4) stated in \S 1,
for every $p\in\left(  1,\infty\right)  $ there exists $c>0$ such that for
every $u\in C_{0}^{\infty}\left(  \mathbb{R}^{n}\right)  $, $r>0$,
$x_{0},\overline{x}\in\mathbb{R}^{n}$ with $x_{0}\in B\left(  \overline
{x},r\right)  $, we have, for $i,j=1,...,m$ and every $k\geq2$:%
\begin{align}
&  \frac{1}{\left\vert B_{r}\left(  \overline{x}\right)  \right\vert }%
\int_{B_{r}\left(  \overline{x}\right)  }\left\vert X_{i}X_{j}u\left(
x\right)  -\left(  X_{i}X_{j}u\right)  _{B_{r}}\right\vert dx\label{Krylov}\\
&  \leq c\left\{  \frac{1}{k}\sum_{h,l=1}^{m}\mathcal{M}\left(  X_{h}%
X_{l}u\right)  \left(  x_{0}\right)  +k^{\frac{q}{p}}\left(  \frac
{1}{\left\vert B_{kr}\left(  \overline{x}\right)  \right\vert }\int%
_{B_{kr}\left(  \overline{x}\right)  }\left\vert L_{A}u\left(  x\right)
\right\vert ^{p}dx\right)  ^{1/p}\right\}  .\nonumber
\end{align}
The constant $c$ depends on the matrix $A$ only through the number $\nu$ in
(\ref{ellipticity}).
\end{theorem}

Here $\mathcal{M}$ is the Hardy-Littlewood maximal operator which has been
defined in (\ref{maximal HL}).

To get the above theorem we will exploit an abstract result which we think can
be of independent interest, and therefore we state in the general context of
spaces of homogeneous type.

\bigskip

Let $T$ be a singular integral operator which we already know to be bounded on
$L^{p}\left(  X\right)  $ on some space of homogeneous type $\left(
X,d,\mu\right)  $ and some (or all) $p\in\left(  1,\infty\right)  $. Just to
simplify notation, in the integrals we will always write $dx$ instead of
$d\mu\left(  x\right)  $. The operator $T$ has kernel $K\left(  x,y\right)  $,
which means that%
\[
Tf\left(  x\right)  =\int_{X}K\left(  x,y\right)  f\left(  y\right)  dy
\]
at least when $f$ is compactly supported and $x$ does not belong to
$\operatorname{sprt}f.$ We assume that the kernel $K$ satisfies the mean value
inequality%
\begin{equation}
\left\vert K\left(  x_{0},y\right)  -K\left(  x,y\right)  \right\vert \leq
C\frac{d\left(  x_{0},x\right)  }{d\left(  x_{0},y\right)  B\left(
x;y\right)  }\text{ if }d\left(  x_{0},y\right)  \geq Md\left(  x_{0}%
,x\right)  \label{mean value ineq}%
\end{equation}
where
\[
B\left(  x;y\right)  =\mu\left(  B\left(  x,d\left(  x,y\right)  \right)
\right)
\]
and $M>1$ is such that the condition $d\left(  x_{0},y\right)  \geq Md\left(
x_{0},x\right)  $ implies the equivalence of $d\left(  x_{0},y\right)  $ and
$d\left(  x,y\right)  $ (if $d$ is a distance, like in our case, $M=2$ is a
good choice).

Let $q>1$ be an exponent such that%
\begin{equation}
\left\vert B\left(  x,kr\right)  \right\vert \leq ck^{q}\left\vert B\left(
x,r\right)  \right\vert \label{growth balls}%
\end{equation}
for every $k\geq1,r>0,x\in X.$ Such exponent exists in every space of
homogeneous type; in our concrete application, by (\ref{bounds on balls}),
this exponent will be exactly the number that we have called $q$, i.e. the
homogeneous dimension of $\mathbb{R}^{n}$.

Under the above assumptions we prove the following:

\begin{theorem}
\label{Thm sharp}There exists $c>0$ such that for every $f\in L^{p}\left(
X\right)  $, $x_{0}\in X$, ball $B_{r}=B\left(  \overline{x},r\right)  \ni
x_{0}$ (for some $r>0$, $\overline{x}\in X$), $k\geq M$, we have:%
\begin{align*}
&  \frac{1}{\left\vert B_{r}\left(  \overline{x}\right)  \right\vert }%
\int_{B_{r}}\left\vert Tf\left(  x\right)  -\left(  Tf\right)  _{B_{r}\left(
\overline{x}\right)  }\right\vert dx\\
&  \leq c\left\{  \frac{1}{k}\mathcal{M}f\left(  x_{0}\right)  +k^{\frac{q}%
{p}}\left(  \frac{1}{\left\vert B_{kr}\left(  \overline{x}\right)  \right\vert
}\int_{B_{kr}\left(  \overline{x}\right)  }\left\vert f\left(  x\right)
\right\vert ^{p}dx\right)  ^{1/p}\right\}  .
\end{align*}

\end{theorem}

\begin{proof}
To simplify notation, this proof is written assuming $d$ a distance (and not
just a quasidistance), in particular $M=2$. The proof for a general
quasidistance has some inessential changes.

For fixed $r>0$, $k\geq2\left(  =M\right)  $ and $f\in L^{p}\left(  X\right)
$, let%
\[
f=f\chi_{B_{kr}\left(  \overline{x}\right)  }+f\chi_{\left(  B_{kr}\left(
\overline{x}\right)  \right)  ^{c}}\equiv f_{1}+f_{2},
\]
then for any $c\in\mathbb{R}$ we have%
\begin{align*}
&  \frac{1}{\left\vert B_{r}\left(  \overline{x}\right)  \right\vert }%
\int_{B_{r}\left(  \overline{x}\right)  }\left\vert Tf\left(  x\right)
-\left(  Tf\right)  _{B_{r}\left(  \overline{x}\right)  }\right\vert
dx\leq2\frac{1}{\left\vert B_{r}\left(  \overline{x}\right)  \right\vert }%
\int_{B_{r}\left(  \overline{x}\right)  }\left\vert Tf\left(  x\right)
-c\right\vert dx\\
&  \leq2\frac{1}{\left\vert B_{r}\left(  \overline{x}\right)  \right\vert
}\int_{B_{r}\left(  \overline{x}\right)  }\left\vert Tf_{1}\left(  x\right)
\right\vert dx+2\frac{1}{\left\vert B_{r}\left(  \overline{x}\right)
\right\vert }\int_{B_{r}\left(  \overline{x}\right)  }\left\vert Tf_{2}\left(
x\right)  -c\right\vert dx\equiv I+II.
\end{align*}
By H\"{o}lder inequality and $L^{p}$ continuity of $T$ we have%
\begin{align*}
I  &  \leq2\left(  \frac{1}{\left\vert B_{r}\left(  \overline{x}\right)
\right\vert }\int_{B_{r}\left(  \overline{x}\right)  }\left\vert Tf_{1}\left(
x\right)  \right\vert ^{p}dx\right)  ^{1/p}\leq2C\left(  \frac{1}{\left\vert
B_{r}\left(  \overline{x}\right)  \right\vert }\int_{X}\left\vert f_{1}\left(
x\right)  \right\vert ^{p}dx\right)  ^{1/p}\\
&  =2C\left(  \frac{1}{\left\vert B_{r}\left(  \overline{x}\right)
\right\vert }\int_{B_{kr}\left(  \overline{x}\right)  }\left\vert f\left(
x\right)  \right\vert ^{p}dx\right)  ^{1/p}\leq ck^{\frac{q}{p}}\left(
\frac{1}{\left\vert B_{kr}\left(  \overline{x}\right)  \right\vert }%
\int_{B_{kr}\left(  \overline{x}\right)  }\left\vert f\left(  x\right)
\right\vert ^{p}dx\right)  ^{1/p},
\end{align*}
where we have exploited (\ref{growth balls}).

Next, since, in the integral $II$, $x\in B_{r}\left(  \overline{x}\right)  $
and $\operatorname{sprt}f_{2}\subset\left(  B_{kr}\left(  \overline{x}\right)
\right)  ^{c}$,%
\[
Tf_{2}\left(  x\right)  =\int K\left(  x,y\right)  f_{2}\left(  y\right)
dy=\int_{\left(  B_{kr}\left(  \overline{x}\right)  \right)  ^{c}}K\left(
x,y\right)  f\left(  y\right)  dy\text{.}%
\]
Then, in $II$, pick $c=Tf_{2}\left(  x^{\ast}\right)  =\int_{\left(
B_{kr}\left(  \overline{x}\right)  \right)  ^{c}}K\left(  x^{\ast},y\right)
f\left(  y\right)  dy$ for some $x^{\ast}\in B_{r}\left(  \overline{x}\right)
$ such that $Tf_{2}\left(  x^{\ast}\right)  $ is finite ($Tf_{2}$ is an
$L^{p}$ function). Then, by the mean value inequality (\ref{mean value ineq})
for $K$,%
\begin{align}
II  &  =2\frac{1}{\left\vert B_{r}\left(  \overline{x}\right)  \right\vert
}\int_{B_{r}\left(  \overline{x}\right)  }\left\vert \int_{\left(
B_{kr}\left(  \overline{x}\right)  \right)  ^{c}}\left[  K\left(  x,y\right)
-K\left(  x^{\ast},y\right)  \right]  f\left(  y\right)  dy\right\vert
dx\nonumber\\
&  \leq c\frac{1}{\left\vert B_{r}\left(  \overline{x}\right)  \right\vert
}\int_{B_{r}\left(  \overline{x}\right)  }\int_{\left(  B_{kr}\left(
\overline{x}\right)  \right)  ^{c}}\frac{d\left(  x,x^{\ast}\right)
}{d\left(  x,y\right)  \left\vert B\left(  x;y\right)  \right\vert }\left\vert
f\left(  y\right)  \right\vert dydx\nonumber\\
&  \leq cr\frac{1}{\left\vert B_{r}\left(  \overline{x}\right)  \right\vert
}\int_{B_{r}\left(  \overline{x}\right)  }\left(  \int_{\left(  B_{kr}\left(
\overline{x}\right)  \right)  ^{c}}\frac{\left\vert f\left(  y\right)
\right\vert }{d\left(  x,y\right)  \left\vert B\left(  x;y\right)  \right\vert
}dy\right)  dx. \label{M1}%
\end{align}
Note that under our assumptions on $\overline{x},x,y$ the distances $d\left(
\overline{x},y\right)  $ and $d\left(  x,y\right)  $ are equivalent. Since
$x_{0}\in B_{r}\left(  \overline{x}\right)  $, also $d\left(  x_{0},y\right)
$ and $d\left(  x,y\right)  $ are equivalent. Moreover,%
\[
\left(  B_{kr}\left(  \overline{x}\right)  \right)  ^{c}\subset\left(
B_{\left(  k-1\right)  r}\left(  x_{0}\right)  \right)  ^{c},
\]
so that%
\begin{align}
&  \int_{\left(  B_{kr}\left(  \overline{x}\right)  \right)  ^{c}}%
\frac{\left\vert f\left(  y\right)  \right\vert }{d\left(  x,y\right)
\left\vert B\left(  x;y\right)  \right\vert }dy\leq c\int_{\left(  B_{\left(
k-1\right)  r}\left(  x_{0}\right)  \right)  ^{c}}\frac{\left\vert f\left(
y\right)  \right\vert }{d\left(  x_{0},y\right)  \left\vert B\left(
x_{0};y\right)  \right\vert }dy\nonumber\\
&  \leq c\sum_{h=0}^{\infty}\int_{2^{h}\left(  k-1\right)  r\leq d\left(
x_{0},y\right)  <2^{h+1}\left(  k-1\right)  r}\frac{\left\vert f\left(
y\right)  \right\vert }{d\left(  x_{0},y\right)  \left\vert B\left(
x_{0};y\right)  \right\vert }dy\nonumber\\
&  \leq c\sum_{h=0}^{\infty}\frac{1}{2^{h}\left(  k-1\right)  r\left\vert
B_{2^{h}\left(  k-1\right)  r}\left(  x_{0}\right)  \right\vert }%
\int_{B_{2^{h+1}\left(  k-1\right)  r}\left(  x_{0}\right)  }\left\vert
f\left(  y\right)  \right\vert dy\nonumber\\
&  \leq c_{1}\sum_{h=0}^{\infty}\frac{1}{2^{h}\left(  k-1\right)  r}\frac
{1}{\left\vert B_{2^{h+1}\left(  k-1\right)  r}\left(  x_{0}\right)
\right\vert }\int_{B_{2^{h+1}\left(  k-1\right)  r}\left(  x_{0}\right)
}\left\vert f\left(  y\right)  \right\vert dy\nonumber\\
&  \leq\frac{c_{1}}{\left(  k-1\right)  r}\mathcal{M}f\left(  x_{0}\right)  .
\label{M2}%
\end{align}
where in the up to last inequality we have exploited the doubling property$.$

By (\ref{M1})-(\ref{M2}) we conclude%
\[
II\leq cr\frac{1}{\left\vert B_{r}\left(  \overline{x}\right)  \right\vert
}\int_{B_{r}\left(  \overline{x}\right)  }\frac{c_{1}}{\left(  k-1\right)
r}\mathcal{M}f\left(  x_{0}\right)  dx=\frac{c}{k-1}\mathcal{M}f\left(
x_{0}\right)  \leq\frac{c^{\prime}}{k}\mathcal{M}f\left(  x_{0}\right)
\]
since for, $k\geq2$, $k$ and $k-1$ are equivalent. So we are done.
\end{proof}

\bigskip

\begin{proof}
[Proof of Theorem \ref{Thm Krylov main step}]By Theorem
\ref{Thm repr formula2} we know that for every $u\in C_{0}^{\infty}\left(
\mathbb{R}^{n}\right)  $ we can write:
\begin{align}
X_{i}X_{j}u\left(  x\right)   &  =-\lim_{\varepsilon\rightarrow0^{+}}%
\lim_{R\rightarrow+\infty}\int_{\mathbb{R}^{n}}K_{\varepsilon,R}^{A}\left(
x,y\right)  L_{A}u\left(  y\right)  dy+c_{ij}^{A}L_{A}u\left(  x\right)
\nonumber\\
&  \equiv-\lim_{\varepsilon\rightarrow0^{+}}\lim_{R\rightarrow+\infty
}T_{\varepsilon,R}^{A}\left(  L_{A}u\right)  \left(  x\right)  +c_{ij}%
^{A}L_{A}u\left(  x\right)  , \label{repr1}%
\end{align}
where the operator
\[
T^{A}f\equiv\lim_{\varepsilon\rightarrow0^{+}}\lim_{R\rightarrow+\infty
}T_{\varepsilon,R}^{A}\left(  f\right)
\]
with kernel%
\[
K_{A}\left(  x;y\right)  =X_{i}^{x}X_{j}^{x}\Gamma_{A}\left(  x;y\right)
\]
satisfies the assumptions of Theorem \ref{Thm sharp} by Theorems
\ref{Thm Lp constant} and \ref{Thm prop sing kern}. Therefore by Theorem
\ref{Thm sharp} we have, with the same notation%
\begin{align}
&  \frac{1}{\left\vert B_{r}\left(  \overline{x}\right)  \right\vert }%
\int_{B_{r}}\left\vert T^{A}\left(  L_{A}u\right)  \left(  x\right)
-T^{A}\left(  L_{A}u\right)  _{B_{r}\left(  \overline{x}\right)  }\right\vert
dx\nonumber\\
&  \leq c\left\{  \frac{1}{k}\mathcal{M}\left(  L_{A}u\right)  \left(
x_{0}\right)  +k^{\frac{q}{p}}\left(  \frac{1}{\left\vert B_{kr}\left(
\overline{x}\right)  \right\vert }\int_{B_{kr}\left(  \overline{x}\right)
}\left\vert L_{A}u\left(  x\right)  \right\vert ^{p}dx\right)  ^{1/p}\right\}
. \label{K1}%
\end{align}
Also, if $c_{ij}^{A}$ are the constants appearing in (\ref{repr1}),
\begin{align}
&  \frac{1}{\left\vert B_{r}\left(  \overline{x}\right)  \right\vert }%
\int_{B_{r}\left(  \overline{x}\right)  }\left\vert \left(  c_{ij}^{A}%
L_{A}u\right)  \left(  x\right)  -\left(  c_{ij}^{A}L_{A}u\right)  _{B_{r}%
}\right\vert dx\nonumber\\
&  \leq2\frac{1}{\left\vert B_{r}\left(  \overline{x}\right)  \right\vert
}\int_{B_{r}\left(  \overline{x}\right)  }\left\vert \left(  c_{ij}^{A}%
L_{A}u\right)  \left(  x\right)  \right\vert dx\nonumber\\
&  \leq c\frac{1}{\left\vert B_{r}\left(  \overline{x}\right)  \right\vert
}\int_{B_{r}\left(  \overline{x}\right)  }\left\vert L_{A}u\left(  x\right)
\right\vert dx\leq c\left(  \frac{1}{\left\vert B_{r}\left(  \overline
{x}\right)  \right\vert }\int_{B_{r}\left(  \overline{x}\right)  }\left\vert
L_{A}u\left(  x\right)  \right\vert ^{p}dx\right)  ^{1/p}\nonumber\\
&  \leq ck^{\frac{q}{p}}\left(  \frac{1}{\left\vert B_{kr}\left(  \overline
{x}\right)  \right\vert }\int_{B_{kr\left(  \overline{x}\right)  }}\left\vert
L_{A}u\left(  x\right)  \right\vert ^{p}dx\right)  ^{1/p} \label{molt part}%
\end{align}
where in the last inequality we have exploited (\ref{growth balls}) and, by
Theorem \ref{Thm repr formula2}, $c$ depends on $A$ only through $\nu.$

Finally, since, for every ball $B_{r}\left(  \overline{x}\right)  \ni x_{0}$
\[
\frac{1}{\left\vert B_{r}\left(  \overline{x}\right)  \right\vert }\int%
_{B_{r}\left(  \overline{x}\right)  }\left\vert L_{A}u\left(  x\right)
\right\vert dx\leq c\sum_{h,l=1}^{m}\frac{1}{\left\vert B_{r}\left(
\overline{x}\right)  \right\vert }\int_{B_{r}\left(  \overline{x}\right)
}\left\vert X_{h}X_{l}u\left(  x\right)  \right\vert dx,
\]
we have%
\begin{equation}
\mathcal{M}\left(  L_{A}u\right)  \left(  x_{0}\right)  \leq c\sum_{h,l=1}%
^{m}\mathcal{M}\left(  X_{h}X_{l}u\right)  \left(  x_{0}\right)  \label{K2}%
\end{equation}
with $c=c\left(  \nu\right)  $. Combining (\ref{repr1}) with (\ref{K1}),
(\ref{molt part}) and (\ref{K2}), we get (\ref{Krylov}).
\end{proof}

\section{$L^{p}$ estimates for operators \newline with $VMO$ coefficients}

\label{sec:estimatesVMO}

In this section we come at last to consider operators (\ref{operators}), i.e.%
\[
Lu=\sum_{i,j=1}^{m}a_{ij}\left(  x\right)  X_{i}X_{j}u
\]
with \emph{variable }coefficients $a_{ij}\left(  x\right)  $, satisfying
assumptions (H.1)-(H.5).

\subsection{Estimates on the mean oscillation of $X_{i}X_{j}u$ \newline in
terms of $Lu$}

The first step is the proof of a control on the mean oscillation of
$X_{i}X_{j}u$ for a smooth function $u$ with small support, in terms of $Lu$.
Here we combine Theorem \ref{Thm Krylov main step} with the $VMO$ assumption
on the coefficients $a_{ij}$, following as close as possible Krylov' technique
in \cite{K}.

\begin{theorem}
\label{Thm 2}Let $p,\alpha,\beta\in(1,\infty)$ with $\alpha^{-1}+\beta^{-1}%
=1$. There exists $c>0$, depending on $\left\{  X_{1},...,X_{m}\right\}
,p,\alpha,\nu$, such that for every $R,r>0$, $x^{\ast},x_{0},\overline{x}%
\in\mathbb{R}^{n}$ with $x_{0}\in B_{r}\left(  \overline{x}\right)  $ and
$u\in C_{0}^{\infty}(B_{R}\left(  x^{\ast}\right)  )$, and every $k\geq M$
(with $M$ the number in Theorem \ref{Thm sharp})%
\begin{align*}
&  \frac{1}{\left\vert B_{r}\left(  \overline{x}\right)  \right\vert }%
\int_{B_{r}\left(  \overline{x}\right)  }\left\vert X_{i}X_{j}u\left(
x\right)  -\left(  X_{i}X_{j}u\right)  _{B_{r}}\right\vert dx\\
&  \leqslant\frac{c}{k}\sum_{h,l=1}^{m}\mathcal{M}(X_{h}X_{l}u)\left(
x_{0}\right)  +ck^{q/p}\left(  \mathcal{M}\left(  |Lu|^{p}\right)  \left(
x_{0}\right)  \right)  ^{1/p}\ \\
&  +ck^{q/p}\left(  a_{R}^{\sharp}\right)  ^{1/p\beta}\sum_{h,l=1}^{m}\left(
\mathcal{M}(\left\vert X_{h}X_{l}u\right\vert ^{p\alpha})\left(  x_{0}\right)
\right)  ^{1/p\alpha}%
\end{align*}
for $i,j=1,2,...,m$. Recall that $a_{R}^{\sharp}$ has been defined in
(\ref{mod VMO coeff}).
\end{theorem}

\begin{proof}
We can assume that
\[
B_{r}\left(  \overline{x}\right)  \cap B_{R}\left(  x^{\ast}\right)
\neq\varnothing
\]
because otherwise%
\[
\int_{B_{r}\left(  \overline{x}\right)  }\left\vert X_{i}X_{j}u\left(
x\right)  -\left(  X_{i}X_{j}u\right)  _{B_{r}}\right\vert dx=0,
\]
and there is nothing to prove.

By Theorem \ref{Thm Krylov main step} we can write, for every constant matrix
$A$ satisfying (H.4) and every $k$ large enough:%
\begin{align}
&  \frac{1}{\left\vert B_{r}\left(  \overline{x}\right)  \right\vert }%
\int_{B_{r}\left(  \overline{x}\right)  }|X_{i}X_{j}u\left(  x\right)
-(X_{i}X_{j}u)_{B_{r}}|dx\nonumber\\
&  \leqslant\frac{c}{k}\sum_{i,j=1}^{q}\mathcal{M}(X_{i}X_{j}u)\left(
x_{0}\right)  +ck^{q/p}\left(  \frac{1}{|B_{kr}\left(  \overline{x}\right)
|}\int_{B_{kr}\left(  \overline{x}\right)  }|L_{A}u\left(  x\right)
|^{p}dx\right)  ^{1/p}. \label{proof 2}%
\end{align}
To bound the right hand side, we can write%
\[
\left\Vert L_{A}u\right\Vert _{L^{p}\left(  B_{kr}\left(  \overline{x}\right)
\right)  }\leqslant\left\Vert Lu\right\Vert _{L^{p}\left(  B_{kr}\left(
\overline{x}\right)  \right)  }+\left\Vert L_{A}u-Lu\right\Vert _{L^{p}\left(
B_{kr}\left(  \overline{x}\right)  \right)  }%
\]
and exploit the fact that, since $x_{0}\in B_{r}\left(  \overline{x}\right)
$
\[
\left(  \frac{1}{|B_{kr}\left(  \overline{x}\right)  |}\int_{B_{kr}\left(
\overline{x}\right)  }|Lu\left(  x\right)  |^{p}dx\right)  ^{1/p}%
\leqslant\left(  \mathcal{M}(|Lu|^{p})\left(  x_{0}\right)  \right)  ^{1/p}%
\]
so that%
\begin{align}
&  \frac{1}{\left\vert B_{r}\left(  \overline{x}\right)  \right\vert }%
\int_{B_{r}\left(  \overline{x}\right)  }|X_{i}X_{j}u\left(  x\right)
-(X_{i}X_{j}u)_{B_{r}}|dx\leqslant\frac{c}{k}\sum_{i,j=1}^{q}\mathcal{M}%
(X_{i}X_{j}u)\left(  x_{0}\right) \nonumber\\
&  +ck^{q/p}\left\{  \left(  \mathcal{M}(|Lu|^{p})\left(  x_{0}\right)
\right)  ^{1/p}+\frac{1}{|B_{kr}\left(  \overline{x}\right)  |^{1/p}%
}\left\Vert L_{A}u-Lu\right\Vert _{L^{p}\left(  B_{kr}\left(  \overline
{x}\right)  \right)  }\right\}  . \label{proof 3}%
\end{align}

Next, let us write, denoting by $\left(  \overline{a}_{ij}\right)  $ the
entries of the constant matrix $A$,%
\begin{align}
&  \int_{B_{kr}\left(  \overline{x}\right)  }|L_{A}u(x)-Lu(x)|^{p}%
dx\nonumber\\
&  \leqslant c\sum_{i,j=1}^{m}\left(  \int_{B_{kr}\left(  \overline{x}\right)
\cap B_{R}\left(  x^{\ast}\right)  }|\overline{a}_{ij}-a_{ij}(x)|^{p\beta
}dx\right)  ^{1/\beta}\left(  \int_{B_{kr}\left(  \overline{x}\right)  \cap
B_{R}\left(  x^{\ast}\right)  }|X_{i}X_{j}u|^{p\alpha}dx\right)  ^{1/\alpha
}\ . \label{proof 4}%
\end{align}
Since the coefficients $\overline{a}_{ij},a_{ij}$ are bounded by $\nu^{-1}$ we
have%
\begin{equation}
\int_{B_{kr}\left(  \overline{x}\right)  \cap B_{R}\left(  x^{\ast}\right)
}|\overline{a}_{ij}-a_{ij}(x)|^{p\beta}dx\leqslant\left(  2\nu^{-1}\right)
^{\beta p-1}\int_{B_{kr}\left(  \overline{x}\right)  \cap B_{R}\left(
x^{\ast}\right)  }\left\vert a_{ij}\left(  x\right)  -\overline{a}%
_{ij}\right\vert dx. \label{proof 5}%
\end{equation}
Recalling that all our estimates hold for every constant matrix $A$ in the
fixed ellipticity class (\ref{ellipticity}), we now choose a particular
constant matrix $\left\{  \overline{a}_{ij}\right\}  $, depending on $r,k,R$:%
\[
\overline{a}_{ij}=\left\{
\begin{array}
[c]{ll}%
(a_{ij})_{B_{R}\left(  x^{\ast}\right)  } & \text{if }kr\geq R\\
(a_{ij})_{B_{kr}\left(  \overline{x}\right)  } & \text{if }kr\leqslant R
\end{array}
\right.
\]
and with this definition we easily get%
\[
\int_{B_{kr}\left(  \overline{x}\right)  \cap B_{R}\left(  x^{\ast}\right)
}\left\vert a_{ij}\left(  x\right)  -\overline{a}_{ij}\right\vert
dx\leqslant\left\{
\begin{array}
[c]{ll}%
c\left\vert B_{R}\left(  x^{\ast}\right)  \right\vert a_{R}^{\sharp} &
\text{if }kr\geq R\\
c\left\vert B_{kr}\left(  \overline{x}\right)  \right\vert a_{R}^{\sharp} &
\text{if }kr\leqslant R
\end{array}
\right.
\]
Next, since $B_{r}\left(  \overline{x}\right)  $ and $B_{R}\left(  x^{\ast
}\right)  \ $intersect, in the case $kr\geq R$ we can write, by the doubling
condition,%
\[
\left\vert B_{R}\left(  x^{\ast}\right)  \right\vert \leq\left\vert
B_{kr}\left(  x^{\ast}\right)  \right\vert \leq c\left\vert B_{kr}\left(
\overline{x}\right)  \right\vert
\]
so that in any case%
\[
\int_{B_{kr}\left(  \overline{x}\right)  \cap B_{R}\left(  x^{\ast}\right)
}\left\vert a_{ij}\left(  x\right)  -\overline{a}_{ij}\right\vert dx\leqslant
c\left\vert B_{kr}\left(  \overline{x}\right)  \right\vert a_{R}^{\sharp}%
\]
and by (\ref{proof 3}), (\ref{proof 4}), (\ref{proof 5}), we have%
\begin{align}
&  \frac{1}{\left\vert B_{r}\left(  \overline{x}\right)  \right\vert }%
\int_{B_{r}\left(  \overline{x}\right)  }\left\vert X_{i}X_{j}u\left(
x\right)  -\left(  X_{i}X_{j}u\right)  _{B_{r}}\right\vert dx\nonumber\\
&  \leqslant\frac{c}{k}\sum_{i,j=1}^{q}\mathcal{M}(X_{i}X_{j}u)\left(
x_{0}\right)  +ck^{q/p}\left(  \mathcal{M}(|Lu|^{p})\left(  x_{0}\right)
\right)  ^{1/p}\nonumber\\
&  +ck^{q/p}\frac{1}{|B_{kr}\left(  \overline{x}\right)  |^{1/p}}\sum
_{i,j=1}^{m}\left(  \left\vert B_{kr}\left(  \overline{x}\right)  \right\vert
a_{R}^{\sharp}\right)  ^{1/p\beta}\left(  \int_{B_{kr}\left(  \overline
{x}\right)  \cap B_{R}\left(  x^{\ast}\right)  }|X_{i}X_{j}u|^{p\alpha
}dx\right)  ^{1/p\alpha}. \label{proof 7}%
\end{align}
Note that the last line is bounded by%
\begin{align*}
&  ck^{q/p}\left(  a_{R}^{\sharp}\right)  ^{1/p\beta}\sum_{i,j=1}^{m}\left(
\frac{1}{|B_{kr}\left(  \overline{x}\right)  |}\int_{B_{kr}\left(
\overline{x}\right)  }|X_{i}X_{j}u|^{p\alpha}dx\right)  ^{1/p\alpha}\\
&  \leqslant ck^{q/p}\left(  a_{R}^{\sharp}\right)  ^{1/p\beta}\sum
_{i,j=1}^{m}\left(  \mathcal{M}(\left\vert X_{i}X_{j}u\right\vert ^{p\alpha
})\left(  x_{0}\right)  \right)  ^{1/p\alpha}%
\end{align*}
with $c$ also depending on $\nu.$ In the last inequality we have exploited the
fact that $x_{0}\in B_{kr}\left(  \overline{x}\right)  $. Inserting this bound
into (\ref{proof 7}), we finally get the statement of Theorem \ref{Thm 2}.
\end{proof}

\subsection{Local $L^{p}$ estimates}

\begin{theorem}
\label{Thm local small Lp}For every $p\in\left(  1,\infty\right)  $ there
exist $R,c>0$ such that for every ball $B_{R}\left(  x^{\ast}\right)  $ and
$u\in C_{0}^{\infty}\left(  B_{R}\left(  x^{\ast}\right)  \right)  $
\begin{equation}
\sum_{h,l=1}^{m}\left\Vert X_{h}X_{l}u\right\Vert _{L^{p}\left(  B_{R}\left(
x^{\ast}\right)  \right)  }\leq c\left\Vert Lu\right\Vert _{L^{p}\left(
B_{R}\left(  x^{\ast}\right)  \right)  }. \label{final local}%
\end{equation}
The constant $c$ in (\ref{final local}) depends on the numbers $p,\nu,$ the
vector fields $\left\{  X_{1},...,X_{m}\right\}  $ and the function $a_{\cdot
}^{\sharp}$, but does does not depend on $R$ or $x^{\ast}.$
\end{theorem}

\begin{proof}
Let $u\in C_{0}^{\infty}(B_{R}\left(  x^{\ast}\right)  )$, with $R$ to be
chosen later, then for every $x_{0}$, applying Theorem \ref{Thm 2} and taking
the supremum over all the balls $B_{r}\left(  \overline{x}\right)  $
containing $x_{0}$ we get, for every $k\geq M$,
\begin{align}
\left(  X_{i}X_{j}u\right)  ^{\#}\left(  x_{0}\right)   &  \leqslant\frac
{c}{k}\sum_{h,l=1}^{m}\mathcal{M}(X_{h}X_{l}u)\left(  x_{0}\right)
+ck^{q/p}\left(  \mathcal{M}\left(  |Lu|^{p}\right)  \left(  x_{0}\right)
\right)  ^{1/p}\ \nonumber\\
&  +ck^{q/p}\left(  a_{R}^{\sharp}\right)  ^{1/p\beta}\sum_{h,l=1}^{m}\left(
\mathcal{M}(\left\vert X_{h}X_{l}u\right\vert ^{p\alpha})\left(  x_{0}\right)
\right)  ^{1/p\alpha}. \label{Lp1}%
\end{align}
Note that this inequality holds for every $p,\alpha,\beta\in(1,\infty)$ with
$\alpha^{-1}+\beta^{-1}=1$, with the constants $c$ in (\ref{Lp1}) also
depending on these three numbers. For a fixed $p\in(1,\infty)$, let us choose
three numbers $p_{1},\alpha,\beta\in(1,\infty)$ with $\alpha^{-1}+\beta
^{-1}=1$ and $\alpha p_{1}<p$. Let us rewrite (\ref{Lp1}) with $p$ replaced by
$p_{1}\ $and then take $L^{p}\left(  \mathbb{R}^{n}\right)  $ norms of both
sides. We get:%
\begin{align*}
&  \left\Vert \left(  X_{i}X_{j}u\right)  ^{\#}\right\Vert _{L^{p}\left(
\mathbb{R}^{n}\right)  }\leq\frac{c}{k}\sum_{h,l=1}^{m}\left\Vert
\mathcal{M}(X_{h}X_{l}u)\right\Vert _{L^{p}\left(  \mathbb{R}^{n}\right)  }\\
&  +ck^{q/p_{1}}\left(  \int_{\mathbb{R}^{n}}\left\vert \mathcal{M}\left(
|Lu|^{p_{1}}\right)  \left(  y\right)  \right\vert ^{p/p_{1}}dy\right)
^{1/p}\\
&  +ck^{q/p_{1}}\left(  a_{R}^{\sharp}\right)  ^{1/p_{1}\beta}\sum_{h,l=1}%
^{m}\left(  \int_{\mathbb{R}^{n}}\left\vert \mathcal{M}(\left\vert X_{h}%
X_{l}u\right\vert ^{p_{1}\alpha})\left(  y\right)  \right\vert ^{p/p_{1}%
\alpha}dy\right)  ^{1/p}\\
&  =\frac{c}{k}\sum_{h,l=1}^{m}\left\Vert \mathcal{M}(X_{h}X_{l}u)\right\Vert
_{L^{p}\left(  \mathbb{R}^{n}\right)  }+ck^{q/p_{1}}\left\Vert \mathcal{M}%
\left(  |Lu|^{p_{1}}\right)  \right\Vert _{L^{p/p_{1}}\left(  \mathbb{R}%
^{n}\right)  }^{1/p_{1}}\\
&  +ck^{q/p_{1}}\left(  a_{R}^{\sharp}\right)  ^{1/p_{1}\beta}\sum_{h,l=1}%
^{m}\left\Vert \mathcal{M}(\left\vert X_{h}X_{l}u\right\vert ^{p_{1}\alpha
}\right\Vert _{L^{p/p_{1}\alpha}\left(  \mathbb{R}^{n}\right)  }%
^{1/p_{1}\alpha}.
\end{align*}
Since $p/p_{1}$ and $p/\left(  p_{1}\alpha\right)  $ belong to $\left(
1,\infty\right)  $, by the $L^{p}$ bound on the maximal operator $\mathcal{M}%
$, Theorem \ref{Thm Maximal}, we get, also recalling that $u\in C_{0}^{\infty
}\left(  B_{R}\left(  x^{\ast}\right)  \right)  ,$%
\begin{align*}
\left\Vert \left(  X_{i}X_{j}u\right)  ^{\#}\right\Vert _{L^{p}\left(
\mathbb{R}^{n}\right)  }  &  \leq\frac{c}{k}\sum_{h,l=1}^{m}\left\Vert
X_{h}X_{l}u\right\Vert _{L^{p}\left(  \mathbb{R}^{n}\right)  }+ck^{q/p_{1}%
}\left\Vert |Lu|^{p_{1}}\right\Vert _{L^{p/p_{1}}\left(  \mathbb{R}%
^{n}\right)  }^{1/p_{1}}\\
&  +ck^{q/p_{1}}\left(  a_{R}^{\sharp}\right)  ^{1/p_{1}\beta}\sum_{h,l=1}%
^{m}\left\Vert \left\vert X_{h}X_{l}u\right\vert ^{p_{1}\alpha}\right\Vert
_{L^{p/p_{1}\alpha}\left(  \mathbb{R}^{n}\right)  }^{1/p_{1}\alpha}\\
&  =\frac{c}{k}\sum_{h,l=1}^{m}\left\Vert X_{h}X_{l}u\right\Vert
_{L^{p}\left(  B_{R}\left(  x^{\ast}\right)  \right)  }+ck^{q/p_{1}}\left\Vert
Lu\right\Vert _{L^{p}\left(  B_{R}\left(  x^{\ast}\right)  \right)  }\\
&  +ck^{q/p_{1}}\left(  a_{R}^{\sharp}\right)  ^{1/p_{1}\beta}\sum_{h,l=1}%
^{m}\left\Vert X_{h}X_{l}u\right\Vert _{L^{p}\left(  B_{R}\left(  x^{\ast
}\right)  \right)  }.
\end{align*}
Note that, since $u\in C_{0}^{\infty}\left(  \mathbb{R}^{n}\right)  $, we are
entitled to apply the Fefferman-Stein inequality (Theorem
\ref{Thm Fefferman Stein}) and write:%
\begin{align*}
&  \left\Vert X_{i}X_{j}u\right\Vert _{L^{p}\left(  B_{R}\left(  x^{\ast
}\right)  \right)  }\leq c\left\{  \frac{1}{k}\sum_{h,l=1}^{m}\left\Vert
X_{h}X_{l}u\right\Vert _{L^{p}\left(  B_{R}\left(  x^{\ast}\right)  \right)
}+\right. \\
&  \left.  +k^{q/p_{1}}\left[  \left\Vert Lu\right\Vert _{L^{p}\left(
B_{R}\left(  x^{\ast}\right)  \right)  }+\left(  a_{R}^{\sharp}\right)
^{1/p_{1}\beta}\sum_{h,l=1}^{m}\left\Vert X_{h}X_{l}u\right\Vert
_{L^{p}\left(  B_{R}\left(  x^{\ast}\right)  \right)  }\right]  \right\}  .
\end{align*}
Recall that by now $p,p_{1},\alpha,\beta$, and therefore the constant $c$, are
fixed, and the above inequality holds for every $k$ large enough. Choosing now
$k$ large enough we can write:%
\begin{align}
&  \sum_{i,j=1}^{m}\left\Vert X_{i}X_{j}u\right\Vert _{L^{p}\left(
B_{R}\left(  x^{\ast}\right)  \right)  }\label{Lp2}\\
&  \leq ck^{q/p_{1}}\left\{  \left\Vert Lu\right\Vert _{L^{p}\left(
B_{R}\left(  x^{\ast}\right)  \right)  }+\left(  a_{R}^{\sharp}\right)
^{1/p_{1}\beta}\sum_{h,l=1}^{m}\left\Vert X_{h}X_{l}u\right\Vert
_{L^{p}\left(  B_{R}\left(  x^{\ast}\right)  \right)  }\right\}  .\nonumber
\end{align}
Recall that in Theorem \ref{Thm 2} the constant $c$ does not depend on $R$,
hence the same is true in (\ref{Lp2}). Since by assumptions (H.5) the function
$a_{R}^{\sharp}$ vanishes as $R\rightarrow0^{+}$, we can finally choose $R$
small enough so that $ck^{q/p_{1}}\left(  a_{R}^{\sharp}\right)
^{1/p_{1}\beta}<1/2$ and conclude that for that small $R$ we have%
\[
\sum_{h,l=1}^{m}\left\Vert X_{h}X_{l}u\right\Vert _{L^{p}\left(  B_{R}\left(
x^{\ast}\right)  \right)  }\leq c\left\Vert Lu\right\Vert _{L^{p}\left(
B_{R}\left(  x^{\ast}\right)  \right)  }.
\]

\end{proof}

\subsection{Global $L^{p}$ estimates}

The next step consists in deriving the following version of global a priori estimates:

\begin{theorem}
\label{Thm stima globale ver1}For every $p\in\left(  1,\infty\right)  $ there
exists $c>0$, depending on numbers $p,\nu,$ the vector fields $\left\{
X_{1},...,X_{m}\right\}  $ and the function $R\mapsto a_{R}^{\sharp}$, such
that for every $u\in C_{0}^{\infty}\left(  \mathbb{R}^{n}\right)  $ we have%
\[
\left\Vert u\right\Vert _{W_{X}^{2,p}\left(  \mathbb{R}^{n}\right)  }\leq
c\left\{  \left\Vert Lu\right\Vert _{L^{p}\left(  \mathbb{R}^{n}\right)
}+\left\Vert u\right\Vert _{W_{X}^{1,p}\left(  \mathbb{R}^{n}\right)
}\right\}  .
\]

\end{theorem}

\begin{proof}
Let $H>1$ be as in Lemma \ref{Lemma cutoff} and fix $R>0$ small enough so that
Theorem \ref{Thm local small Lp} is applicable to balls $B_{HR}$. We can then
build a family of cutoff functions $\left\{  \phi^{x}\left(  \cdot\right)
\right\}  _{x\in\mathbb{R}^{n}}$ as in Lemma \ref{Lemma cutoff}. Next, by
Proposition \ref{Prop overlapping} we can cover $\mathbb{R}^{n}$ with a family
$\left\{  B\left(  x_{\alpha},R\right)  \right\}  _{\alpha\in A}$ of balls
such that the family of dilated balls $\left\{  B\left(  x_{\alpha},HR\right)
\right\}  _{\alpha\in A}$ has the bounded overlapping property. Let us
consider the family of cutoff functions adapted to these balls; letting let
$\phi_{\alpha}=\phi^{x_{\alpha}}$ we have:%
\begin{align*}
\phi_{\alpha}  &  \in C_{0}^{\infty}\left(  B_{HR}\left(  x_{\alpha}\right)
\right) \\
\phi_{\alpha}\left(  y\right)   &  \geq c_{1}\text{ in }B_{R}\left(
x_{\alpha}\right)
\end{align*}%
\[
\sup_{y}\left\vert \phi_{\alpha}\left(  y\right)  \right\vert +\sum_{h=1}%
^{m}\sup_{y}\left\vert X_{h}\phi_{\alpha}\left(  y\right)  \right\vert
+\sum_{h,l=1}^{m}\sup_{y}\left\vert X_{h}X_{l}\phi_{\alpha}\left(  y\right)
\right\vert \leq c_{2}%
\]
with $c_{1},c_{2}$ independent of $\alpha$ (while $R$ is by now fixed).

Then we can write, for every $\alpha\in A$,%
\begin{align*}
&  \sum_{h,l=1}^{m}\left\Vert X_{h}X_{l}u\right\Vert _{L^{p}\left(
B_{R}\left(  x_{\alpha}\right)  \right)  }\leq\sum_{h,l=1}^{m}\frac{1}{c_{1}%
}\left\Vert \phi_{\alpha}X_{h}X_{l}u\right\Vert _{L^{p}\left(  B_{R}\left(
x_{\alpha}\right)  \right)  }\\
&  \leq c\sum_{h,l=1}^{m}\left\{  \left\Vert X_{h}X_{l}\left(  u\phi_{\alpha
}\right)  \right\Vert _{L^{p}\left(  B_{HR}\left(  x_{\alpha}\right)  \right)
}+\left\Vert X_{h}uX_{l}\phi_{\alpha}\right\Vert _{L^{p}\left(  B_{HR}\left(
x_{\alpha}\right)  \right)  }\right. \\
&  \left.  +\left\Vert uX_{h}X_{l}\phi_{\alpha}\right\Vert _{L^{p}\left(
B_{HR}\left(  x_{\alpha}\right)  \right)  }\right\} \\
&  \leq c\left\{  \sum_{h,l=1}^{m}\left\Vert X_{h}X_{l}\left(  u\phi_{\alpha
}\right)  \right\Vert _{L^{p}\left(  B_{HR}\left(  x_{\alpha}\right)  \right)
}+\sum_{h=1}^{m}\left\Vert X_{h}u\right\Vert _{L^{p}\left(  B_{HR}\left(
x_{\alpha}\right)  \right)  }+\left\Vert u\right\Vert _{L^{p}\left(
B_{HR}\left(  x_{\alpha}\right)  \right)  }\right\}
\end{align*}
by Theorem \ref{Thm local small Lp}
\begin{align*}
&  \leq c\left\{  \left\Vert L\left(  u\phi_{\alpha}\right)  \right\Vert
_{L^{p}\left(  B_{HR}\left(  x_{\alpha}\right)  \right)  }+\sum_{h=1}%
^{m}\left\Vert X_{h}u\right\Vert _{L^{p}\left(  B_{HR}\left(  x_{\alpha
}\right)  \right)  }+\left\Vert u\right\Vert _{L^{p}\left(  B_{HR}\left(
x_{\alpha}\right)  \right)  }\right\} \\
&  \leq c\left\{  \left\Vert Lu\right\Vert _{L^{p}\left(  B_{HR}\left(
x_{\alpha}\right)  \right)  }+\sum_{h=1}^{m}\left\Vert X_{h}u\right\Vert
_{L^{p}\left(  B_{HR}\left(  x_{\alpha}\right)  \right)  }+\left\Vert
u\right\Vert _{L^{p}\left(  B_{HR}\left(  x_{\alpha}\right)  \right)
}\right\}
\end{align*}

Next, we sum up over $\alpha\in A$, finding, since $\left\{  B\left(
x_{\alpha},R\right)  \right\}  _{\alpha\in A}$ are a covering of
$\mathbb{R}^{n}$,
\begin{align*}
&  \sum_{h,l=1}^{m}\left\Vert X_{h}X_{l}u\right\Vert _{L^{p}\left(
\mathbb{R}^{n}\right)  }\leq\sum_{\alpha\in A}\sum_{h,l=1}^{m}\left\Vert
X_{h}X_{l}u\right\Vert _{L^{p}\left(  B_{R}\left(  x_{\alpha}\right)  \right)
}\\
&  \leq c\sum_{\alpha\in A}\left\{  \left\Vert Lu\right\Vert _{L^{p}\left(
B_{HR}\left(  x_{\alpha}\right)  \right)  }+\sum_{h=1}^{m}\left\Vert
X_{h}u\right\Vert _{L^{p}\left(  B_{HR}\left(  x_{\alpha}\right)  \right)
}+\left\Vert u\right\Vert _{L^{p}\left(  B_{HR}\left(  x_{\alpha}\right)
\right)  }\right\}
\end{align*}
applying Proposition \ref{Prop overlapping}
\[
\leq c\left\{  \left\Vert Lu\right\Vert _{L^{p}\left(  \mathbb{R}^{n}\right)
}+\sum_{h=1}^{m}\left\Vert X_{h}u\right\Vert _{L^{p}\left(  \mathbb{R}%
^{n}\right)  }+\left\Vert u\right\Vert _{L^{p}\left(  \mathbb{R}^{n}\right)
}\right\}  .
\]
Adding $\sum_{h=1}^{m}\left\Vert X_{h}u\right\Vert _{L^{p}\left(
\mathbb{R}^{n}\right)  }+\left\Vert u\right\Vert _{L^{p}\left(  \mathbb{R}%
^{n}\right)  }$ to both sides of the last inequality we get the conclusion.
\end{proof}

\begin{corollary}
\label{Thm stima globale ver2}For every $p\in\left(  1,\infty\right)  $ there
exists $c>0$, depending on numbers $p,\nu$, the vector fields $\left\{
X_{1},...,X_{m}\right\}  $ and the function $R\mapsto a_{R}^{\sharp}$, such
that for every $u\in C^{\infty}\left(  \mathbb{R}^{n}\right)  $ and for every
$R>1$ we have%
\[
\left\Vert u\right\Vert _{W_{X}^{2,p}\left(  B_{R}^{\ast}\left(  0\right)
\right)  }\leq c\left\{  \left\Vert Lu\right\Vert _{L^{p}\left(  B_{2R}^{\ast
}\left(  0\right)  \right)  }+\left\Vert u\right\Vert _{W_{X}^{1,p}\left(
B_{2R}^{\ast}\left(  0\right)  \right)  }\right\}  ,
\]
with $c$ independent of $R$. Here
\[
B_{R}^{\ast}\left(  0\right)  =\left\{  x\in\mathbb{R}^{n}:\left\Vert
x\right\Vert \leq R\right\}
\]
and, by an abuse of notation, for $x\in\mathbb{R}^{n}$ we let $\left\Vert
x\right\Vert =\left\Vert \left(  x,0\right)  \right\Vert $ with $\left\Vert
\left(  x,\xi\right)  \right\Vert $ the usual homogeneous norm in
$\mathbb{R}^{N}$.
\end{corollary}

\begin{proof}
Let $\phi\in C_{0}^{\infty}\left(  B_{2}^{\ast}\left(  0\right)  \right)
,\phi=1$ in $B_{1}^{\ast}\left(  0\right)  $ and let%
\[
\phi_{R}\left(  x\right)  =\phi\left(  \delta_{1/R}\left(  x\right)  \right)
.
\]
Then:%
\[
\phi_{R}\in C_{0}^{\infty}\left(  B_{2R}^{\ast}\left(  0\right)  \right)
,\phi_{R}=1\text{ in }B_{R}^{\ast}\left(  0\right)  ,
\]%
\[
\left\vert \phi_{R}\left(  x\right)  \right\vert +R\left\vert \left(
X_{i}\phi_{R}\right)  \left(  x\right)  \right\vert +R^{2}\left\vert \left(
X_{i}X_{j}\phi_{R}\right)  \left(  x\right)  \right\vert \leq c\text{ for
every }x\in\mathbb{R}^{N},R>0.
\]
Let $u\in C^{\infty}\left(  \mathbb{R}^{n}\right)  $ and for a fixed $R>0$ let
us apply Theorem \ref{Thm stima globale ver1} to the function $u\phi_{R}\in
C_{0}^{\infty}\left(  \mathbb{R}^{n}\right)  $. Then:%
\begin{align*}
\left\Vert u\right\Vert _{W_{X}^{2,p}\left(  B_{R}^{\ast}\left(  0\right)
\right)  }  &  \leq\left\Vert u\phi_{R}\right\Vert _{W_{X}^{2,p}\left(
\mathbb{R}^{n}\right)  }\leq c\left\{  \left\Vert L\left(  u\phi_{R}\right)
\right\Vert _{L^{p}\left(  \mathbb{R}^{n}\right)  }+\left\Vert u\phi
_{R}\right\Vert _{W_{X}^{1,p}\left(  \mathbb{R}^{n}\right)  }\right\} \\
&  \leq c\left\{  \left\Vert Lu\right\Vert _{L^{p}\left(  B_{2R}^{\ast}\left(
0\right)  \right)  }+\left\Vert u\right\Vert _{W_{X}^{1,p}\left(  B_{2R}%
^{\ast}\left(  0\right)  \right)  }\right\}  .
\end{align*}

\end{proof}

To conclude the proof of our main result, we also need an interpolation
inequality for intermediate norms which holds for \emph{homogeneous }vector fields:

\begin{theorem}
\label{Thm interpolation}(See \cite[Prop.2.9]{BBB2}). For every $p\in\left(
1,\infty\right)  $ there exists $c_{p}$ such that for every $u\in W_{X}%
^{2,p}\left(  \mathbb{R}^{n}\right)  $ we have%
\[
\left\Vert X_{i}u\right\Vert _{L^{p}\left(  \mathbb{R}^{n}\right)  }%
\leq\varepsilon\left\Vert X_{i}^{2}u\right\Vert _{L^{p}\left(  \mathbb{R}%
^{n}\right)  }+\frac{c_{p}}{\varepsilon}\left\Vert u\right\Vert _{L^{p}\left(
\mathbb{R}^{n}\right)  }%
\]
for $i=1,2,...,m$ and every $\varepsilon>0$.
\end{theorem}

We can now finally conclude the:

\bigskip

\begin{proof}
[Proof of Theorem \ref{Thm main}]Let $u\in W_{X}^{2,p}\left(  \mathbb{R}%
^{n}\right)  $ and, for a fixed $R,$ take a sequence $\left\{  u_{i}\right\}
_{i=1}^{\infty}\subset C_{0}^{\infty}\left(  \mathbb{R}^{n}\right)  $ such
that on the compact set $B_{2R}^{\ast}\left(  0\right)  $ we have
$u_{i}\rightarrow u$ in $W_{X}^{2,p}$-norm. This is possible in view of
Theorem \ref{Thm approx}. By Corollary \ref{Thm stima globale ver2} we can
write%
\[
\left\Vert u_{i}\right\Vert _{W_{X}^{2,p}\left(  B_{R}^{\ast}\left(  0\right)
\right)  }\leq c\left\{  \left\Vert Lu_{i}\right\Vert _{L^{p}\left(
B_{2R}^{\ast}\left(  0\right)  \right)  }+\left\Vert u_{i}\right\Vert
_{W_{X}^{1,p}\left(  B_{2R}^{\ast}\left(  0\right)  \right)  }\right\}
\]
and, passing to the limit as $i\rightarrow+\infty$, we conclude%
\[
\left\Vert u\right\Vert _{W_{X}^{2,p}\left(  B_{R}^{\ast}\left(  0\right)
\right)  }\leq c\left\{  \left\Vert Lu\right\Vert _{L^{p}\left(  B_{2R}^{\ast
}\left(  0\right)  \right)  }+\left\Vert u\right\Vert _{W_{X}^{1,p}\left(
B_{2R}^{\ast}\left(  0\right)  \right)  }\right\}
\]
for every $u\in W_{X}^{2,p}\left(  \mathbb{R}^{n}\right)  $, with $c$
independent of $R$. Passing to the limit for $R\rightarrow+\infty$, we get%
\[
\left\Vert u\right\Vert _{W_{X}^{2,p}\left(  \mathbb{R}^{n}\right)  }\leq
c\left\{  \left\Vert Lu\right\Vert _{L^{p}\left(  \mathbb{R}^{n}\right)
}+\left\Vert u\right\Vert _{W_{X}^{1,p}\left(  \mathbb{R}^{n}\right)
}\right\}
\]
for every $u\in W_{X}^{2,p}\left(  \mathbb{R}^{n}\right)  $. Finally, we apply
the interpolation inequaltiy in Theorem \ref{Thm interpolation} with
$\varepsilon$ small enough, and conclude%
\[
\left\Vert u\right\Vert _{W_{X}^{2,p}\left(  \mathbb{R}^{n}\right)  }\leq
c\left\{  \left\Vert Lu\right\Vert _{L^{p}\left(  \mathbb{R}^{n}\right)
}+\left\Vert u\right\Vert _{L^{p}\left(  \mathbb{R}^{n}\right)  }\right\}  ,
\]
that is our desired global estimate.
\end{proof}

\section{Higher order estimates\label{sec higher order}}

Taking into account all the machinery developed so far to prove Theorem
\ref{Thm main}, we now turn to establish \emph{global higher-order a-priori
estimates} for the solutions of the equation $Lu=f$ (where $L$ is as in
\eqref{operators}), under the assumption that the (variable) coefficients
$a_{ij}$ of $L$ have some extra regularity. Together with the regularization
result that we will prove in section 6, this will imply our second main
result, that is Theorem \ref{Thm main 2}.

Throughout the section we will use the notation fixed in Definition
\ref{Def Sobolev}; multiindices will be denoted by letters $I,J,J^{\prime}.$

\begin{theorem}
\label{thm:mainHO}Let the assumptions \emph{(H.1)\thinspace-\thinspace(H.4)}
(stated in section 1) be in force, and suppose that there exists some positive
integer $k\ $such that
\[
\Vert a\Vert=\sum_{i,j=1}^{m}\Vert a_{ij}\Vert_{W_{X}^{k,\infty}%
(\mathbb{R}^{n})}<\infty.
\]
Then, for every $p\in(1,\infty)$ there exists a constant $c>0$, only depending
on the numbers $k,p,\nu$ \emph{(}see \eqref{ellipticity}\emph{)}, the vector
fields $\{X_{1},...,X_{m}\}$ and the number $\Vert a\Vert$, such that, for
every $u\in W_{X}^{k+2,p}(\mathbb{R}^{n})$, we have
\begin{equation}
\Vert u\Vert_{W_{X}^{k+2,p}(\mathbb{R}^{n})}\leq c\left\{  \Vert
Lu\Vert_{W_{X}^{k,p}(\mathbb{R}^{n})}+\Vert u\Vert_{L^{p}(\mathbb{R}^{n}%
)}\right\}  . \label{eq:HOEstimate}%
\end{equation}

\end{theorem}

\begin{proof}
In order to prove \eqref{eq:HOEstimate} we essentially go through the whole
argument used for the proof of Theorem \ref{Thm main}, but we replace the
representation formula \eqref{repr form XX} with an analogous one for
\emph{higher-orders derivatives}.

To ease the readability, throughout what follows we assume that:

\begin{itemize}
\item[i)] $I=(i_{1},\ldots,i_{k})\ $is a fixed multiindex of length $k$;

\item[ii)] $1\leq i,j\leq m$ are fixed indices;

\item[iii)] $1<p<\infty$ is a fixed exponent.
\end{itemize}

Furthermore, if $A$ is any \emph{constant} $m\times m$ matrix satisfying
\eqref{ellipticity}, we inherit (and adopt) the associated notation in Theorem
\ref{Thm lifting}. \vspace*{0.1cm}

We then consider the following steps. \medskip

\noindent\textbf{Step 1 (Representation formula for }$X_{i}X_{j}X_{I}%
u$\textbf{).} Let $A$ be a fixed $m\times m$ matrix \emph{with constant
entries}, and satisfying \eqref{ellipticity}. Given any $u\in C_{0}^{\infty
}(\mathbb{R}^{n})$, we first prove a representation formula for $X_{i}%
X_{j}X_{I}u$ in terms of $X_{J}L_{A}u$ (for finitely many $J$ of lenght $k$),
which is a \emph{higher-or\-der version} of \eqref{repr form XX}. To this end,
we fix $1\leq l\leq m$ and we notice that, by arguing as in the proof of
Theorem \ref{Thm repr formula 1}, we can write
\begin{align*}
X_{l}u(x)  &  =-\int_{\mathbb{R}^{n}}\left(  \int_{\mathbb{R}^{p}%
}\widetilde{X}_{l}\widetilde{\Gamma}_{A}((y,\eta)^{-1}\ast(x,0))d\eta\right)
L_{A}u(y)\,dy\\
&  (\text{setting $v=(y,\eta)$ and applying Fubini's Theorem})\\
&  =-\int_{\mathbb{R}^{N}}\widetilde{X}_{l}\widetilde{\Gamma}_{A}(v^{-1}%
\ast(x,0))\widetilde{L}_{A}u(v)\,dv
\end{align*}
From this, since $\widetilde{\Gamma}_{A},\,\widetilde{X}_{l}$ and
$\widetilde{L}_{A}u$ are defined on the \emph{Carnot group} $\mathbb{G}%
=(\mathbb{R}^{N},\ast)$, we can proceed exactly as in the proof of
\cite[Proposition 8.32]{BBbook}, obtaining
\begin{equation}%
\begin{split}
X_{l}u(x)  &  =-\sum_{k=1}^{m}\int_{\mathbb{R}^{N}}\widetilde{\Gamma}%
_{A,k}(v^{-1}\ast(x,0))\widetilde{X}_{k}\widetilde{L}_{A}u(v)\,dv\\
&  =-\sum_{k=1}^{m}\int_{\mathbb{R}^{n}}{X}_{k}{L}_{A}u(y)\left(
\int_{\mathbb{R}^{p}}\widetilde{\Gamma}_{A,k}((y,\eta)^{-1}\ast(x,0))\,d\eta
\right)  dy,
\end{split}
\label{eq:reprXluByPart}%
\end{equation}
where the $\widetilde{\Gamma}_{A,k}$'s are homogeneous functions of degree
$2-Q$, smooth outside the origin. We explicitly stress that, \emph{differently
from the case considered} in the cited \cite[Proposition 8.32]{BBbook}, here
\emph{we do not have} $f=\widetilde{L}_{A}u\in C_{0}^{\infty}(\mathbb{R}%
^{N});$ however, we can still apply \cite[Proposition 3.47]{BBbook} in order
to obtain \eqref{eq:reprXluByPart}. Indeed, since the functions
$\widetilde{\Gamma}_{A,k}$ are homogeneous of degree $2-Q$, by proceeding
exactly as in the proof of \cite[Theorem 2.1\,-\,(ii)]{BBlift} we see that
\[
\int_{\mathbb{R}^{N}}\left\vert \widetilde{\Gamma}_{A,k}(v^{-1}\ast
(x,0))\widetilde{X}_{k}\widetilde{L}_{A}u(v)\right\vert \,dv<\infty.
\]
With \eqref{eq:reprXluByPart} at hand, by a classical iteration argument we
then get
\begin{equation}%
\begin{split}
&  X_{I}u=X_{i_{1}}\cdots X_{i_{k}}u\\
&  \qquad=-\sum_{\left\vert J\right\vert =k}\int_{\mathbb{R}^{n}}X_{J}%
L_{A}u(y)\left(  \int_{\mathbb{R}^{p}}\widetilde{\Gamma}_{A,J}((y,\eta
)^{-1}\ast(x,0))\,d\eta\right)  dy,
\end{split}
\label{eq:ReprXalfau}%
\end{equation}
where the functions $\widetilde{\Gamma}_{A,J}$ are homogeneous of degree
$2-Q$, and smooth outside the origin. Starting from \eqref{eq:ReprXalfau}
(which is the analog of \eqref{repr formula 0}), and proceeding exactly as in
the proof of Theorem \ref{Thm repr formula 1}, we finally obtain%
\begin{equation}%
\begin{split}
&  X_{i}X_{j}X_{I}u\left(  x\right) \\
&  =\sum_{\left\vert J\right\vert =k}\Big\{-\lim_{\varepsilon\rightarrow0}%
\int_{\widetilde{d}((y,\eta),(x,0))\geq\varepsilon}\!\!\!\widetilde{X}%
_{i}\widetilde{X}_{j}\widetilde{\Gamma}_{A,J}(y,\eta)^{-1}\ast(x,0))X_{J}%
L_{A}u(y)\,dy\,d\eta\\
&  \qquad+c_{ij}^{A,J}X_{J}L_{A}u(x)\Big\}\\
&  =\sum_{\left\vert J\right\vert =k}\Big\{-\lim_{\varepsilon\rightarrow0}%
\lim_{R\rightarrow\infty}\int_{\mathbb{R}^{n}}K_{J,\varepsilon,R}%
^{A,i,j}(x,y)X_{J}L_{A}u(y)\,dy+c_{ij}^{A,J}X_{J}L_{A}u(x)\Big\},
\end{split}
\label{eq:ReprXiXjXalfau}%
\end{equation}
where the kernels $K_{J,\varepsilon,R}^{A,i,j}$ are defined by
\[
K_{J,\varepsilon,R}^{A,i,j}(x,y)=\int_{\mathbb{R}^{p}}\left(  \widetilde{X}%
_{i}\widetilde{X}_{j}\widetilde{\Gamma}_{A,J}\cdot\psi_{\varepsilon,R}\right)
((y,\eta)^{-1}\ast(x,0))\,d\eta,
\]
while the constants $c_{ij}^{A,J}$ are explicitly given by
\[
c_{ij}^{A,J}=\int_{\left\{  u\in\mathbb{R}^{N}:\left\Vert u\right\Vert
=1\right\}  }\left(  \widetilde{X}_{j}\widetilde{\Gamma}_{A,J}\right)  \left(
\widetilde{X}_{i}\cdot\nu\right)  d\sigma\left(  u\right)
\]
We explicitly stress that the key facts allowing to obtain
\eqref{eq:ReprXiXjXalfau} are the following: on the one hand, the functions
$\widetilde{\Gamma}_{A,J}$ are homogeneous of degree $2-Q$ and smooth outside
the origin; hence, the second-order derivatives
\[
\widetilde{X}_{i}\widetilde{X}_{j}\widetilde{\Gamma}_{A,J}%
\]
(which are clearly homogeneous of degree $-Q$ and smooth outside the origin)
satisfy the \emph{cancellation property} in Proposition
\ref{Prop vanishing integral} (see \cite[Corollary 6.31]{BBbook}).

On the other hand, by the construction of the kernels $\widetilde{\Gamma
}_{A,J}$ performed in \cite[proof of Proposition 8.32]{BBbook}, we see that
any $X$-derivative of the maps $\widetilde{\Gamma}_{A,J}$ can be estimated in
terms of the $X$-derivatives of $\widetilde{\Gamma}_{A}$, and thus we have
\emph{uniform estimates w.r.t.\thinspace the matrix $A$}. Thus, we see that

\begin{itemize}
\item[i)] the kernels $K_{J,\varepsilon,R}^{A,i,j}$ satisfy \emph{the very
same properties} of the analogous kernels $K_{\varepsilon,R}^{A}$ appearing in
the representation formula \eqref{repr form XX};

\item[ii)] the constants $c_{ij}^{A,J}$ satisfies \emph{uniform estimates
w.r.t.\thinspace$A$}.
\end{itemize}

\noindent\textbf{Step 2): Estimates for }$X_{i}X_{j}X_{I}u$\textbf{.} With
\eqref{eq:ReprXalfau} at hand, we now turn to e\-sta\-blish \emph{$L^{p}$ and
mean-oscillation estimates} for $X_{i}X_{j}X_{I}u$, which are the analog of
the estimates for $X_{i}X_{j}u$ proved in Theorems \ref{Thm Lp constant}%
\thinspace-\thinspace\ \ref{Thm Krylov main step}, respectively.

To begin with, taking into account (\ref{eq:ReprXalfau}) (and following the
scheme of the proof Theorem \ref{Thm main}), we arbitrarily fix a $k$-tuple
$J=(j_{1},\ldots,j_{k})$, and we introduce the following operator
\[
T_{\varepsilon,R}^{A,J}(f)=-\int_{\mathbb{R}^{n}}K_{J,\varepsilon,R}%
^{A,i,j}(x,y)f(y)\,dy,
\]
which is well-defined for all $f\in C_{0}^{\infty}(\mathbb{R}^{n})$. Since we
have already observed that the kernel $K_{J,\varepsilon,R}^{A,i,j}$ satisfies
the same properties of $K_{\varepsilon,R}^{A}$ (hence, in particular, it
satisfies properties (a)\thinspace-\thinspace(c) in Theorem
\ref{Thm spazio omogeneo}), we derive that $T_{\varepsilon,R}^{A,J}$ can be
\emph{con\-ti\-nu\-ously extended to $L^{p}(\mathbb{R}^{n})$}, and there
exists $c>0$ (independent of $\varepsilon,R$) such that%
\[
\Vert T_{\varepsilon,R}^{A,J}(f)\Vert_{L^{p}(\mathbb{R}^{n})}\leq c\Vert
f\Vert_{L^{p}(\mathbb{R}^{n})}.
\]
This, together with the representation formula \eqref{eq:reprXluByPart} and
the uniform bounds of the constants $c_{ij}^{A,J}$ w.r.t.\thinspace$A$ (see
\textbf{Step 1)} above), allows us to repeat the argument in the proof of
Theorem \ref{Thm Lp constant}, thus obtaining the following facts:

\begin{itemize}
\item[a)] the operator $T^{A,J}$ defined as
\[
T^{A,J}f(x)=-\lim_{\varepsilon\rightarrow0^{+}}\lim_{R\rightarrow+\infty
}T_{\varepsilon,R}^{A,J}(f)(x)\qquad(f\in C_{0}^{\alpha}\left(  \mathbb{R}%
^{n}\right)  )
\]
can be extended to a linear continuous operator on $L^{p}(\mathbb{R}^{n})$,
with operator norm bounded by the same constant $c$;

\item[b)] there exists a constant $c>0$, depending on $A$ only through the
number $\nu$ in \eqref{ellipticity}, such that, for every $u\in C_{0}^{\infty
}\left(  \mathbb{R}^{n}\right)  $, we have
\begin{equation}
\left\Vert X_{i}X_{j}X_{I}u\right\Vert _{L^{p}\left(  \mathbb{R}^{n}\right)
}\leq c\sum_{\left\vert J\right\vert =k}\left\Vert X_{J}L_{A}u\right\Vert
_{L^{p}\left(  \mathbb{R}^{n}\right)  }. \label{eq:LpestimConstantHO}%
\end{equation}

\end{itemize}

We explicitly notice that the above \eqref{eq:LpestimConstantHO} is nothing
but the \emph{higher-order} analog of the $L^{p}$-estimate for $L_{A}$
contained in Theorem \ref{Thm Lp constant}.

On the other hand, owing to the above a), we can also apply Theorem
\ref{Thm sharp} and proceed exactly as in the proof of Theorem
\ref{Thm Krylov main step}; this, together with the representation formula
\eqref{eq:reprXluByPart}, gives the following mean-oscillation estimate (which
is the higher-order analog of \eqref{Krylov} in Theorem
\ref{Thm Krylov main step})
\begin{equation}%
\begin{split}
&  \frac{1}{|B_{r}(\overline{x})|}\int_{B_{r}(\overline{x})}|X_{i}X_{j}%
X_{I}u(x)-(X_{i}X_{j}X_{I}u)_{B_{r}}|dx\\
&  \qquad\leq c\sum_{\left\vert J\right\vert =k}\Big\{\frac{1}{\kappa
}\mathcal{M}(X_{J}L_{A}u)(x_{0})\\
&  \qquad\qquad+\kappa^{\frac{q}{p}}\Big(\frac{1}{|B_{kr}(\overline{x})|}%
\int_{B_{\kappa r}(\overline{x})}|X_{J}L_{A}u(x)|^{p}dx\Big)^{1/p}\Big\}\\
&  \qquad\leq c\Big\{\frac{1}{\kappa}\sum_{\left\vert J^{\prime}\right\vert
=k+2}\mathcal{M}(X_{J^{\prime}}u)(x_{0})\\
&  \qquad\qquad+\kappa^{\frac{q}{p}}\sum_{\left\vert J\right\vert
=k}\Big(\frac{1}{|B_{\kappa r}(\overline{x})|}\int_{B_{\kappa r}(\overline
{x})}|X_{J}L_{A}u(x)|^{p}dx\Big)^{1/p}\Big\},
\end{split}
\label{eq:KrylovHO}%
\end{equation}
holding true for every $u\in C_{0}^{\infty}\left(  \mathbb{R}^{n}\right)  $,
$r>0$, $x_{0},\overline{x}\in\mathbb{R}^{n}$ with $x_{0}\in B\left(
\overline{x},r\right)  $. Here, $\kappa\geq2$ is arbitrarily chosen, and $q$
is as in \eqref{eq.defq}. \medskip

\noindent\textbf{Step 3): Local estimates for }$X_{i}X_{j}X_{I}u$\textbf{ in
terms of }$Lu$\textbf{.} Now we have proved the estimates
\eqref{eq:LpestimConstantHO}\thinspace-\thinspace\eqref{eq:KrylovHO} for the
\emph{constant-coefficient operator $L_{A}$} (where $A$ is an arbitrary matrix
satisfying \eqref{ellipticity}), and following the scheme of the proof of
Theorem \ref{Thm main}, we turn our attention to the
\emph{variable-coefficient operator}
\[
L=\sum_{h,l=1}^{m}a_{hl}(x)X_{h}X_{l}%
\]
(with $a_{hl}\in W_{X}^{k,\infty}(\mathbb{R}^{n})$, for some $k\geq1$).
According to section \ref{sec:estimatesVMO}, the first step in this direction
is to use \eqref{eq:KrylovHO} to derive \emph{mean-oscillation estimates} for
$X_{i}X_{j}X_{[\alpha]}u$ in terms of $Lu$ (see Theorem \ref{Thm 2}); to this
end, it suffices to proceed exactly as in the proof of Theorem \ref{Thm 2},
taking into account the following argument. \vspace*{0.1cm}

First of all, if $A=(\bar{a}_{ij})$ is any (constant) $m\times m$ matrix
satisfying \eqref{ellipticity}, we can clearly write $L_{A}u=Lu+(L_{A}-L)u$;
hence, for every multiindex $J\ $with $\left\vert J\right\vert =k$ we have
\begin{equation}%
\begin{split}
&  X_{J}L_{A}u=X_{J}(Lu)+\sum_{h,l=1}^{m}X_{J}\left[  (\overline{a}%
_{hl}-a_{hl}(x))X_{h}X_{l}u\right] \\
&  \qquad=\sum_{h,l=1}^{m}(\overline{a}_{hl}-a_{hl}(x))X_{J}X_{h}X_{l}u\\
&  \qquad\qquad-\sum_{s=1}^{k}\mathcal{P}_{s}(X_{1},\ldots,X_{m})(a_{hl}%
)\cdot\mathcal{Q}_{k-s}(X_{1},\ldots,X_{m})(X_{h}X_{l}u),
\end{split}
\label{eq:XbetaLAuProduct}%
\end{equation}
where $\mathcal{P}_{s}(t_{1},\ldots,t_{m}),\,\mathcal{Q}_{k-s}(t_{1}%
,\ldots,t_{m})$ (with $1\leq s\leq k$) are suitable \emph{polynomials in the
non-commuting variables $t_{1},\ldots,t_{m}$}, homogeneous of degree $s$ and
$k-s$, respectively, and whose coefficients are \emph{absolute constants}.

By combining \eqref{eq:KrylovHO}\thinspace-\thinspace
\eqref{eq:XbetaLAuProduct}, and by proceeding exactly as in the proof of
The\-orem \ref{Thm 2}, we then obtain the following mean-oscillation estimate%
\begin{align*}
&  \frac{1}{\left\vert B_{r}\left(  \overline{x}\right)  \right\vert }%
\int_{B_{r}\left(  \overline{x}\right)  }\left\vert X_{i}X_{j}X_{I}u\left(
x\right)  -\left(  X_{i}X_{j}X_{I}u\right)  _{B_{r}}\right\vert dx\\
&  \leq c\left\{  \frac{1}{\kappa}\sum_{\left\vert J^{\prime}\right\vert
=k+2}\mathcal{M}\left(  X_{J^{\prime}}u\right)  \left(  x_{0}\right)  \right.
\\
&  +\left.  \kappa^{\frac{q}{p}}\sum_{\left\vert J\right\vert =k}\left(
\frac{1}{\left\vert B_{\kappa r}\left(  \overline{x}\right)  \right\vert }%
\int_{B_{\kappa r}\left(  \overline{x}\right)  }\left\vert X_{J}%
L_{A}u\right\vert ^{p}dx\right)  ^{1/p}\right\}
\end{align*}

(\text{using \eqref{eq:XbetaLAuProduct}, and since $a_{ij}\in W^{k,\infty
}(\mathbb{R}^{n})\subseteq VMO(\mathbb{R}^{n})$ by Proposition }%
\ref{Prop VMO W1inf})%
\begin{equation}%
\begin{split}
&  \leq\frac{c}{\kappa}\sum_{\left\vert J^{\prime}\right\vert =k+2}%
\mathcal{M}(X_{J^{\prime}}u)\left(  x_{0}\right)  +c\kappa^{q/p}%
\sum_{\left\vert J\right\vert =k}\left(  \mathcal{M}\left(  |X_{J}%
Lu|^{p}\right)  \left(  x_{0}\right)  \right)  ^{1/p}\ \\
&  \qquad+c\kappa^{q/p}(a_{R}^{\sharp})^{1/p\tau}\sum_{\left\vert J^{\prime
}\right\vert =k+2}\left(  \mathcal{M}(|X_{J^{\prime}}u|^{p\alpha}%
)(x_{0})\right)  ^{1/p\sigma}\\
&  \qquad\qquad+c\kappa^{q/p}\sum_{\left\vert J^{\prime}\right\vert
=k+1}\left(  \mathcal{M}(|X_{J^{\prime}}u|^{p})(x_{0})\right)  ^{1/p},
\end{split}
\label{eq:LpNormLAuProduct}%
\end{equation}
holding true for every choice of

$\sigma,\tau\in(1,\infty)$ with $\sigma^{-1}+\tau^{-1}=1$; \vspace*{0.05cm}

\thinspace$R,r>0$; \vspace*{0.05cm}

\thinspace$x^{\ast},x_{0},\overline{x}\in\mathbb{R}^{n}$ with $x_{0}\in
B_{r}(\overline{x})$; \vspace*{0.05cm}

$u\in C_{0}^{\infty}(B_{R}\left(  x^{\ast}\right)  )$; \vspace*{0.05cm}

$\kappa\geq M$ (with $M$ the number in Theorem \ref{Thm sharp}).
\vspace*{0.1cm}

\noindent Here, the constant $c > 0$ depends on $p,\sigma$ and on the number
$\|a\|$. \medskip

With estimate \eqref{eq:LpNormLAuProduct} at hand, we can now proceed exactly
as in the proof of Theorem \ref{Thm local small Lp}: choosing $\kappa$ large
enough and $R$ small enough (recall that $W_{X}^{k,\infty}(\mathbb{R}%
^{n})\subseteq VMO(\mathbb{R}^{n})$ by Proposition \ref{Prop VMO W1inf}), we
finally obtain%
\begin{equation}
\Vert X_{i}X_{j}X_{I}u\Vert_{L^{p}(B_{R}(x^{\ast}))}\leq c\left\{  \Vert
Lu\Vert_{W_{X}^{k,p}(B_{R}(x^{\ast}))}+\Vert u\Vert_{W_{X}^{k+1,p}%
(B_{R}(x^{\ast}))}\right\}  , \label{LocalHO}%
\end{equation}
and this estimate holds for every $u\in C_{0}^{\infty}\left(  B_{R}\left(
x^{\ast}\right)  \right)  $, with $R$ small enough. \medskip

\noindent\textbf{Step 4): End of the proof.} Now we have established the
\emph{local higher-order estimate} \eqref{eq:LocalHO}, we can easily complete
the proof of the theorem.

Proceding like in the proof of Theorem \ref{Thm stima globale ver1}, from
(\ref{eq:LocalHO}) we deduce%
\begin{equation}
\Vert u\Vert_{W_{X}^{k+2,p}\left(  \mathbb{R}^{n}\right)  }\leq c\left\{
\Vert Lu\Vert_{W_{X}^{k,p}\left(  \mathbb{R}^{n}\right)  }+\Vert u\Vert
_{W_{X}^{k+1,p}\left(  \mathbb{R}^{n}\right)  }\right\}  \label{global 2}%
\end{equation}
for every $u\in C_{0}^{\infty}\left(  \mathbb{R}^{n}\right)  $. Next,
proceding like in the proof of Corollary \ref{Thm stima globale ver2}, from
(\ref{global 2}) we deduce
\begin{equation}
\Vert u\Vert_{W_{X}^{k+2,p}(B_{R}(x^{\ast}))}\leq c\left\{  \Vert
Lu\Vert_{W_{X}^{k,p}(B_{2R}(x^{\ast}))}+\Vert u\Vert_{W_{X}^{k+1,p}%
(B_{2R}(x^{\ast}))}\right\}  \label{global 3}%
\end{equation}
for every $u\in C^{\infty}\left(  \mathbb{R}^{n}\right)  $ and for every
$R>0$, with $c$ independent of $R$. Then, by the local approximation result in
Sobolev spaces $W_{X}^{k,p}$ (Theorem \ref{Thm approx}), we can infer that
(\ref{global 3}) still holds for every $u\in W_{X}^{k+2,p}\left(
\mathbb{R}^{n}\right)  $ and then, letting $R\rightarrow+\infty$, we get%
\begin{equation}
\Vert u\Vert_{W_{X}^{k+2,p}\left(  \mathbb{R}^{n}\right)  }\leq c\left\{
\Vert Lu\Vert_{W_{X}^{k,p}\left(  \mathbb{R}^{n}\right)  }+\Vert u\Vert
_{W_{X}^{k+1,p}\left(  \mathbb{R}^{n}\right)  }\right\}  \label{global 4}%
\end{equation}
for every $u\in W_{X}^{k+2,p}\left(  \mathbb{R}^{n}\right)  $. Iteration then
yields
\begin{align*}
\Vert u\Vert_{W_{X}^{k+2,p}(\mathbb{R}^{n})}  &  \leq c\big\{\Vert
Lu\Vert_{W_{X}^{k,p}(\mathbb{R}^{n})}+\Vert u\Vert_{W_{X}^{k+1,p}%
(\mathbb{R}^{n})}\big\}\\
&  \leq c\big\{\Vert Lu\Vert_{W_{X}^{k,p}(\mathbb{R}^{n})}+\Vert u\Vert
_{W_{X}^{k,p}(\mathbb{R}^{n})}\big\}\\
&  \vdots\\
&  \leq c\big\{\Vert Lu\Vert_{W_{X}^{k,p}(\mathbb{R}^{n})}+\Vert u\Vert
_{W_{X}^{2,p}(\mathbb{R}^{n})}\big\}\\
&  \leq c\big\{\Vert Lu\Vert_{W_{X}^{k,p}(\mathbb{R}^{n})}+\Vert u\Vert
_{L^{p}(\mathbb{R}^{n})}\big\}
\end{align*}
for every $u\in W_{X}^{k+2,p}\left(  \mathbb{R}^{n}\right)  $, where in the
last inequality we applied Theorem \ref{Thm main}. This ends the proof of
Theorem \ref{thm:mainHO}.
\end{proof}

\section{Regularization of solutions}

In this section we want to prove the following result which, together with
Theorem \ref{thm:mainHO}, will imply the full statement of Theorem
\ref{Thm main 2}.

\begin{theorem}
\label{Thm regularity}Under assumptions (H.1)-(H.4) stated in section 1,
assume that for, some $p\in\left(  1,\infty\right)  $ and positive integer
$k$, $a_{ij}\in W_{X}^{k,\infty}\left(  \mathbb{R}^{n}\right)  $, $u\in
W_{X}^{2,p}\left(  \mathbb{R}^{n}\right)  $, and $Lu\in W_{X}^{k,p}\left(
\mathbb{R}^{n}\right)  $. Then $u\in W_{X}^{k+2,p}\left(  \mathbb{R}%
^{n}\right)  $.
\end{theorem}

To prove the above result we will apply the known \emph{local }regularity
result which holds for nonvariational operators built on a general system of
H\"{o}rmander vector fields (not necessarily homogeneous):

\begin{theorem}
[{See \cite[Thm. 12.2 (a)]{BBbook}}]\label{Thm local regularity}Let
$X_{1},..,X_{m}$ be a system of H\"{o}rmander vector fields in a bounded
domain $\Omega\subset\mathbb{R}^{n}$, let $\left\{  a_{ij}\left(  x\right)
\right\}  _{i,j=1}^{m}$ satisfy our assumption (H.4) stated in section 1 and
let $L$ be as in (\ref{operators}). For every domain $\Omega^{\prime}%
\Subset\Omega$, $p\in\left(  1,\infty\right)  $, nonnegative integer $k$, if
$a_{ij}\in W_{X}^{k,\infty}\left(  \Omega\right)  $, $u\in W_{X}^{2,p}\left(
\Omega\right)  $, and $Lu\in W_{X}^{k,p}\left(  \Omega\right)  $, then $u\in
W_{X}^{k+2,p}\left(  \Omega^{\prime}\right)  $.
\end{theorem}

The quoted result in \cite[Thm. 12.2 (a)]{BBbook} actually contains also local
a priori estimates in $W_{X}^{k+2,p}\left(  \Omega^{\prime}\right)  $, but we
will not apply those local estimates, but the \emph{global} a priori estimates
proved in Theorem \ref{thm:mainHO}.

\bigskip

\begin{proof}
[Proof of Theorem \ref{Thm regularity}]We proceed by induction on $k.$

For $k=1$, assume $a_{ij}\in W_{X}^{1,\infty}\left(  \mathbb{R}^{n}\right)  $,
$u\in W_{X}^{2,p}\left(  \mathbb{R}^{n}\right)  $, and $Lu\in W_{X}%
^{1,p}\left(  \mathbb{R}^{n}\right)  $. For any $R>1$, let $\phi_{R}\in
C_{0}^{\infty}\left(  B_{2R}\left(  0\right)  \right)  ,0\leq\phi_{R}%
\leq1,\phi_{R}=1$ in $B_{R}\left(  0\right)  ,$ with the derivatives of
$\phi_{R}$, up to order $3$, bounded uniformly w.r.t. $R$. By Theorem
\ref{Thm local regularity}, $u\in W_{X}^{3,p}\left(  B_{2R}\left(  0\right)
\right)  $, then $u\phi\in W_{X}^{3,p}\left(  \mathbb{R}^{n}\right)  $ and we
can apply (\ref{eq:HOEstimate}) for $k=1$ getting, for some constant $c>0$
independent of $R,$%
\begin{align*}
\left\Vert u\right\Vert _{W_{X}^{3,p}\left(  B_{R}\left(  0\right)  \right)
}  &  \leq c\left\{  \left\Vert L\left(  u\phi\right)  \right\Vert
_{W_{X}^{1,p}\left(  \mathbb{R}^{n}\right)  }+\left\Vert u\phi\right\Vert
_{L^{p}\left(  \mathbb{R}^{n}\right)  }\right\} \\
&  \leq c\left\{  \left\Vert Lu\right\Vert _{W_{X}^{1,p}\left(  \mathbb{R}%
^{n}\right)  }+\left\Vert u\right\Vert _{W_{X}^{2,p}\left(  \mathbb{R}%
^{n}\right)  }\right\}  .
\end{align*}
Passing to the limit as $R\rightarrow+\infty$, we get%
\[
\left\Vert u\right\Vert _{W_{X}^{3,p}\left(  \mathbb{R}^{n}\right)  }\leq
c\left\{  \left\Vert Lu\right\Vert _{W_{X}^{1,p}\left(  \mathbb{R}^{n}\right)
}+\left\Vert u\right\Vert _{W_{X}^{2,p}\left(  \mathbb{R}^{n}\right)
}\right\}  ,
\]
where the right-hand side is finite by our assumptions, hence $u\in
W_{X}^{3,p}\left(  \mathbb{R}^{n}\right)  $.

Assume we have proved the assertion up to $k-1$ and let us prove it for $k$.
So, assume $a_{ij}\in W_{X}^{k,\infty}\left(  \mathbb{R}^{n}\right)  $, $u\in
W_{X}^{2,p}\left(  \mathbb{R}^{n}\right)  $, and $Lu\in W_{X}^{k,p}\left(
\mathbb{R}^{n}\right)  $. Then by inductive assumption $u\in W_{X}%
^{k+1,p}\left(  \mathbb{R}^{n}\right)  .$ For any $R>1$, let $\phi_{R}\in
C_{0}^{\infty}\left(  B_{2R}\left(  0\right)  \right)  ,0\leq\phi_{R}%
\leq1,\phi_{R}=1$ in $B_{R}\left(  0\right)  ,$ with the derivatives of
$\phi_{R}$, up to order $k+1$, bounded uniformly w.r.t. $R$. By Theorem
\ref{Thm local regularity}, $u\in W_{X}^{k+2,p}\left(  B_{2R}\left(  0\right)
\right)  $. Then $u\phi\in W_{X}^{k+2,p}\left(  \mathbb{R}^{n}\right)  $ and
we can apply (\ref{eq:HOEstimate}) getting, for some constant $c>0$
independent of $R,$%
\begin{align*}
\left\Vert u\right\Vert _{W_{X}^{k+2,p}\left(  B_{R}\left(  0\right)  \right)
}  &  \leq c\left\{  \left\Vert L\left(  u\phi\right)  \right\Vert
_{W_{X}^{k,p}\left(  \mathbb{R}^{n}\right)  }+\left\Vert u\phi\right\Vert
_{L^{p}\left(  \mathbb{R}^{n}\right)  }\right\} \\
&  \leq c\left\{  \left\Vert Lu\right\Vert _{W_{X}^{k,p}\left(  \mathbb{R}%
^{n}\right)  }+\left\Vert u\right\Vert _{W_{X}^{k+1,p}\left(  \mathbb{R}%
^{n}\right)  }\right\}  .
\end{align*}
Passing to the limit as $R\rightarrow+\infty$, we get%
\[
\left\Vert u\right\Vert _{W_{X}^{k+2,p}\left(  \mathbb{R}^{n}\right)  }\leq
c\left\{  \left\Vert Lu\right\Vert _{W_{X}^{k,p}\left(  \mathbb{R}^{n}\right)
}+\left\Vert u\right\Vert _{W_{X}^{k+1,p}\left(  \mathbb{R}^{n}\right)
}\right\}  ,
\]
where the right-hand side is finite by our assumptions, hence $u\in
W_{X}^{k+2,p}\left(  \mathbb{R}^{n}\right)  $.

This completes the proof of Theorem \ref{Thm regularity}, and then of Theorem
\ref{Thm main 2}.
\end{proof}

\section{Appendix. Uniform estimates on $\Gamma^{A}$}

The aim of this Appendix is to complete the justification of the statements of
Theorems \ref{Thm lifting} and \ref{th.teoremone}. Actually, these theorems
have been proved in \cite[Thms 3.2 and 4.4]{BBlift} and \cite[Thm.1.3]{BBB1}
when $A$ is the identity matrix, while in this paper we have stated them for a
general symmetric positive definite matrix $A$, pretending that the constants
involved in the statements depend on $A$ only through the \textquotedblleft
ellipticity constant\textquotedblright\ $\nu$.

\bigskip

\begin{proof}
[Proof of Theorem \ref{Thm lifting}]Assume that $X=\{X_{1},\ldots,X_{m}\}$
satisfies (H.1)-(H.3). Let $\mathfrak{a}=\mathrm{Lie}(X)$, so that
$N=\mathrm{dim}(\mathfrak{a})>n$. Then, by \cite[Thm 1.1]{BBlift}, point (1)
of Theorem \ref{Thm lifting} holds. Here we will keep the same notation of
that statement.

For any $\nu\geq1$, let us denote by $\mathcal{M}_{\nu}$ the set of the
$m\times m$ symmetric \emph{constant }matrices $A$ satisfying the
\emph{ellipticity condition} (\ref{ellipticity}). For every fixed
$A=(a_{i,j})_{i,j=1}^{m}\in\mathcal{M}_{\nu}$ let%
\begin{equation}
L_{A}=\sum_{i,j=1}^{m}a_{i,j}X_{i}X_{j}. \label{eq.defiOpLLA}%
\end{equation}
Since $A$ is symmetric and positive definite, it admits a \emph{unique}
symmetric and positive definite square root, say $S$. As a con\-se\-quen\-ce,
writing $S=(s_{i,j})_{i,j=1}^{m}$, we have
\[
L_{A}=\sum_{j=1}^{m}Y_{j}^{2},\text{ \ \ where $Y_{j}=\sum_{i=1}^{m}%
s_{i,j}X_{i}$}.
\]
On the other hand, since $S$ is non-singular, it is not difficult to recognize
that the family $Y=\{Y_{1},\ldots,Y_{m}\}\subseteq\mathcal{X}(\mathbb{R}^{n})$
satisfies assumptions (H.1)-(H.3). In particular, since the $s_{i,j}$'s are
constant, one has
\begin{equation}
\mathrm{Lie}(Y)=\mathfrak{a}. \label{eq.LieYa}%
\end{equation}
Gathering these facts, we can apply point (1) of Theorem \ref{Thm lifting}
also to the family $Y$: there exist a homogeneous Carnot group $\mathbb{F}$
and a system
\[
\widetilde{Y}=\{\widetilde{Y}_{1},\ldots,\widetilde{Y}_{m}\}
\]
of Lie-generators for the Lie algebra $\mathrm{Lie}(\mathbb{F})$ such that,
for every $i=1,\ldots,m$, the vector field $\widetilde{Y}_{i}$ is a lifting of
$Y_{i}$ in the sense of (\ref{lifting}). \vspace*{0.05cm}The key observation
is that, in view of \cite[Remark 2.3]{BiB1}, the construction of $\mathbb{F}$
does not really depend on $Y$, but only on the Lie algebra $\mathrm{Lie}(Y)$
and on the choice of an \emph{adapted basis}. As a consequence, using
(\ref{eq.LieYa}) and choosing the \emph{same adapted basis} used for the
construction of $\mathbb{G}$, we obtain%
\[
\mathbb{F}=\mathbb{G}\qquad\text{and}\qquad\widetilde{Y}_{j}=\sum_{i=1}%
^{m}s_{i,j}\widetilde{X}_{i}.
\]
Summing up, the couple $\left(  \mathbb{G},\widetilde{X}\right)  $ associated
with $X$ provides a `lifting pair' for the family $Y=\{Y_{1},\ldots,Y_{m}\}$
(in the sense of point (1) of Theorem \ref{Thm lifting}) \emph{which is
independent of the fixed matrix $A\in\mathcal{M}_{\nu}$.} This means in
particular that%
\[
\sum_{j=1}^{m}\widetilde{Y}_{j}^{2}=\sum_{i,j=1}^{m}a_{i,j}\widetilde{X}%
_{i}\widetilde{X}_{j}.
\]
Therefore, if $\widetilde{\Gamma}_{A}$ is the homogeneous fundamental solution
vanishing at infinity of the operator $\sum_{j=1}^{m}\widetilde{Y}_{j}^{2}$ on
the group $(\mathbb{G},\widetilde{X})$, the fundamental solution of $L_{A}$ is
given, according to \cite[Thm. 1]{BBlift}, by (\ref{sec.one:mainThm_defGamma}%
), and points (2) and (3) of Theorem \ref{Thm lifting} hold.
\end{proof}

\bigskip

In order to prove the uniform estimates on the derivatives of $\Gamma_{A}$
contained in Theorem \ref{th.teoremone}, we will need the following known fact:

\begin{theorem}
[Uniform estimates on $\widetilde{\Gamma}_{A}$](See \cite[Thm.6.18]{BBbook})
Let $\mathbb{G}$ be a Carnot group in $\mathbb{R}^{N}$ with generators
$\widetilde{X}_{1},...,\widetilde{X}_{m}$, let $A$ be an $m\times m$ symmetric
constant matrix satisfying (\ref{ellipticity}) and let $\Gamma_{A,\mathbb{G}}$
be the homogeneous fundamental solution, vanishing at infinity, of
$\widetilde{L}_{A}=\sum_{i,j=1}^{m}a_{i,j}\widetilde{X}_{i}\widetilde{X}_{j}$.
Then for every positive integer $r$, there exists a constant $c$, only
depending on the group, the ellipticity constant $\nu$ and the integer $r$,
such that:%
\[
\sup_{z\in\mathbb{R}^{N},\left\Vert z\right\Vert =1}\sup_{\left\vert
\alpha\right\vert \leq r}\left\vert \frac{\partial^{\alpha}\Gamma
_{A,\mathbb{G}}}{\partial z^{\alpha}}\left(  z\right)  \right\vert \leq
c\left(  r,\mathbb{G},\nu\right)  .
\]
(Here we use the notation $\Gamma_{A,\mathbb{G}}$ as in
(\ref{sec.one:mainThm_defGamma2})).
\end{theorem}

By homogeneity, the sup on the unit sphere $\left\{  \left\Vert z\right\Vert
=1\right\}  $ can be replaced by the sup on any compact set $K$ excluding the
origin, with the constant $c$ also depending on $K.$

This bound implies that, for every fixed homogeneous norm $\left\Vert
\cdot\right\Vert $ in $\mathbb{R}^{N}$,%
\begin{equation}
\left\vert \widetilde{X}_{i_{1}}\cdots\widetilde{X}_{i_{s}}\Gamma
_{A,\mathbb{G}}\left(  z\right)  \right\vert \leq\frac{c}{\left\Vert
z\right\Vert ^{Q-2+s}} \label{uniform bound Gamma}%
\end{equation}
with $c$ depending on the matrix $A$ only through the number $\nu$.

\bigskip

\begin{proof}
[Proof of Theorem \ref{th.teoremone}]Point (I) of Theorem \ref{th.teoremone}
is contained in \cite[Thm.1.3]{BBB1}, once we know formula
(\ref{sec.one:mainThm_defGamma}).

In order to prove point (II) of Theorem \ref{th.teoremone}, assuring the
uniformity of the constant $C_{r}$ as the matrix $A$ ranges in the ellipticity
class $\mathcal{M}_{\nu}$, we need to revise the proof of point (II) contained
in \cite{BBB1}, exploiting at a crucial point the uniform bound
(\ref{uniform bound Gamma}). Let us distinguish three cases, according to kind
of $x$- or $y$-derivatives acting on $\Gamma_{A}$.

\textsc{Case I:} $Z_{1}\cdots Z_{r}=X_{i_{1}}^{x}\cdots X_{i_{r}}^{x}$. Using
the integral re\-pre\-sen\-ta\-tion formula in point (I) of Theorem
\ref{th.teoremone} for the $X$-derivatives of $\Gamma_{A}(\cdot;y)$, we can
write, for $x\neq y$,%
\begin{equation}
Z_{1}\cdots Z_{r}\Gamma_{A}(x;y)=\int_{\mathbb{R}^{p}}\left(  \widetilde{X}%
_{i_{1}}\cdots\widetilde{X}_{i_{r}}\Gamma_{A,\mathbb{G}}\right)  \left(
(y,0)^{-1}\ast(x,\eta)\right)  d\eta. \label{eq.reprWiGammaCaseI}%
\end{equation}
Hence, by the uniform bound (\ref{uniform bound Gamma}), we have%
\[
\left\vert Z_{1}\cdots Z_{r}\Gamma_{A}(x;y)\right\vert \leq c\left(
r,\nu,\mathbb{G}\right)  \,\int_{\mathbb{R}^{p}}d_{\widetilde{X}}%
^{2-Q-r}\left(  0,(y,0)^{-1}\ast(x,\eta)\right)  d\eta.
\]
Then we exploit a result in \cite[Cor.\thinspace4.11]{BBB1}: there exists a
constant $C>0$ such that, for every $x\neq y\in\mathbb{R}^{n}$, one has
\begin{equation}
\int_{\mathbb{R}^{p}}d_{\widetilde{X}}^{2-Q-r}\left(  0,(y,0)^{-1}\ast
(x,\eta)\right)  d\eta\leq C\,\frac{d_{X}^{2-r}(x,y)}{\left\vert B_{X}%
(y,d_{X}(x,y))\right\vert }. \label{eq.estimIntdXhat}%
\end{equation}
Combining the last two inequalities we get%
\[
\left\vert Z_{1}\cdots Z_{r}\Gamma_{A}(x;y)\right\vert \leq C\,\frac
{d_{X}^{2-r}(x,y)}{\left\vert B_{X}(y,d_{X}(x,y))\right\vert }\leq C^{\prime
}\,\frac{d_{X}^{2-r}(x,y)}{\left\vert B_{X}(x,d_{X}(x,y))\right\vert }%
\]
where in the last inequality we have exploited the equivalence between
$\left\vert B_{X}(y,d_{X}(x,y))\right\vert $ and $\left\vert B_{X}%
(x,d_{X}(x,y))\right\vert $, which follows by the doubling property.

\textsc{Case II:} $Z_{1}\cdots Z_{r}=X_{j_{1}}^{y}\cdots X_{j_{r}}^{y}$. We
start with the integral re\-pre\-sen\-ta\-tion formula in point (I) of Theorem
\ref{th.teoremone} for the $X$-derivatives of $\Gamma_{A}(x;\cdot)$%
\[
X_{i_{1}}^{y}\cdots X_{i_{s}}^{y}\left(  \Gamma_{A}(x;\cdot)\right)
(y)=\int_{\mathbb{R}^{p}}\left(  \widetilde{X}_{i_{1}}\cdots\widetilde{X}%
_{i_{s}}\Gamma_{A,\mathbb{G}}\right)  \left(  (x,0)^{-1}\ast(y,\eta)\right)
\,d\eta
\]
and then we can argue exactly as in Case I, getting the same bound.

\textsc{Case III:} $Z_{1}\cdots Z_{r}=X_{j_{1}}^{y}\cdots X_{j_{k}}%
^{y}\,X_{i_{1}}^{x}\cdots X_{i_{h}}^{x}$ (with $h+k=r$). This is a bit more
delicate. We start from the integral representation formula in point (I) of
Theorem \ref{th.teoremone} for the mixed $X$-de\-ri\-va\-ti\-ves of
$\Gamma_{A}$, writing%
\begin{align}
&  X_{j_{1}}^{x}\cdots X_{j_{t}}^{x}X_{i_{1}}^{y}\cdots X_{i_{s}}^{y}%
\Gamma_{A}(x;y)\nonumber\\
&  =\int_{\mathbb{R}^{p}}\left(  \widetilde{X}_{j_{1}}\cdots\widetilde{X}%
_{j_{t}}\left(  \left(  \widetilde{X}_{i_{1}}\cdots\widetilde{X}_{i_{s}}%
\Gamma_{A,\mathbb{G}}\right)  \circ\iota\right)  \right)  \left(
(y,0)^{-1}\ast(x,\eta)\right)  d\eta\, \label{repr Case 3}%
\end{align}
where $\iota$ denotes the inversion map of the Lie group $\mathbb{G}$. First
of all, by repeatedly using \cite[Lem.\thinspace5.3]{BiB1}, we can write
\begin{equation}
\widetilde{X}_{j_{1}}\cdots\widetilde{X}_{j_{t}}\left(  \left(  \widetilde{X}%
_{i_{1}}\cdots\widetilde{X}_{i_{s}}\Gamma_{A,\mathbb{G}}\right)  \circ
\iota\right)  =\left(  Z_{j_{1}}\cdots Z_{j_{t}}\,\widetilde{X}_{i_{1}}%
\cdots\widetilde{X}_{i_{s}}\,\widetilde{\Gamma}_{A}\right)  \circ\iota,
\label{eq.derKernelLemma}%
\end{equation}
where $Z_{j_{1}},\ldots,Z_{j_{t}}$ are suitable smooth vector fields
$D_{\lambda}$-homogeneous of degree $1$ but not necessarily left-invariant.
Since all the vector fields in (\ref{eq.derKernelLemma}) are $D_{\lambda}%
$-ho\-mo\-ge\-neous of degree $1$, we are entitled to apply \cite[Lemma
3.4]{BiB2}, obtaining
\begin{equation}
Z_{j_{1}}\cdots Z_{j_{t}}\,\widetilde{X}_{i_{1}}\cdots\widetilde{X}_{i_{s}%
}\,\Gamma_{A,\mathbb{G}}=\!\!\!\sum_{r\,\leqslant\,|\alpha|\,\leqslant
\,r\sigma_{n}}\!\!\!\!\gamma_{\alpha}(z)\,\mathcal{Q}_{\alpha}(\widetilde{X}%
_{1},\ldots,\widetilde{X}_{m})\widetilde{\Gamma}_{A}.
\label{eq.derKernelLemmaBis}%
\end{equation}
Here, $\mathcal{Q}_{\alpha}(t_{1},\ldots,t_{m})$ is a homogeneous polynomial
of (Euclidean) degree $|\alpha|$ in the non-commuting variables $t_{1}%
,\ldots,t_{m}$, and $\gamma_{\alpha}$ is a $D_{\lambda}$-homogeneous
polynomial of degree $|\alpha|-r$. We now observe that, on account of
(\ref{uniform bound Gamma}), we have
\[
\left\vert \mathcal{Q}_{\alpha}(\widetilde{X}_{1},\ldots,\widetilde{X}%
_{m})\Gamma_{A,\mathbb{G}}(z)\right\vert \leq c_{\nu,r}\,d_{\widetilde{X}%
}^{2-Q-|\alpha|}(0,z)\qquad(z\neq0),
\]
where the constant $c_{\nu,r}$ can be chosen independently of $\alpha$.
Moreover, since $\gamma_{\alpha}$ is (smooth and) $D_{\lambda}$-homogeneous of
degree $|\alpha|-r$, we get
\begin{equation}
\left\vert \gamma_{\alpha}(z)\right\vert \leq c_{r}\,d_{\widetilde{X}%
}^{|\alpha|-r}(0,z)\qquad(z\neq0), \label{eq.estimgammaalphadX}%
\end{equation}
where $c=c_{r}>0$ is another constant which can be chosen independently of
$\alpha$. Gathering all these facts, we obtain the following estimate%
\[
\left\vert \widetilde{X}_{j_{1}}\cdots\widetilde{X}_{j_{t}}\left(  \left(
\widetilde{X}_{i_{1}}\cdots\widetilde{X}_{i_{s}}\Gamma_{A,\mathbb{G}}\right)
\circ\iota\right)  \right\vert \leq c_{\nu,r}\,d_{\widetilde{X}}%
^{2-Q-r}\left(  0,z^{-1}\right)  =c_{\nu,r}\,d_{\widetilde{X}}^{2-Q-r}\left(
0,z\right)  .
\]
By (\ref{repr Case 3}) this implies%
\[
\left\vert X_{j_{1}}^{x}\cdots X_{j_{t}}^{x}X_{i_{1}}^{y}\cdots X_{i_{s}}%
^{y}\Gamma_{A}(x;y)\right\vert \leq c_{\nu,r}\int_{\mathbb{R}^{p}%
}d_{\widetilde{X}}^{2-Q-r}\left(  0,(y,0)^{-1}\ast(x,\eta)\right)  d\eta,
\]
which by (\ref{eq.estimIntdXhat}) implies the desired uniform bound. So we are done.
\end{proof}

\bigskip

\noindent\textbf{Acknowledgements. }{The Authors are members of the research
group \textquotedblleft Grup\-po Na\-zio\-nale per l'Analisi Matematica, la
Probabilit\`{a} e le loro Applicazioni\textquotedblright\ of the Italian
\textquotedblleft Istituto Na\-zio\-na\-le di Alta
Matematica\textquotedblright. The first Author is partially supported by the
PRIN 2022 project 2022R537CS \emph{$NO^{3}$ - Nodal Optimization, NOnlinear
elliptic equations, NOnlocal geometric problems, with a focus on regularity},
funded by the European Union - Next Generation EU}; the second Author is
partially supported by the PRIN 2022 project \emph{Partial differential
equations and related geometric-functional inequalities}, financially
supported by the EU, in the framework of the \textquotedblleft Next Generation
EU initiative\textquotedblright.

\bigskip

\noindent\textbf{Address. }S. Biagi, M. Bramanti:

\noindent Dipartimento di Matematica. Politecnico di Milano. \newline\noindent
Via Bonardi 9. 20133 Milano. Italy.

\noindent e-mail: stefano.biagi@polimi.it; marco.bramanti@polimi.it


\bigskip

\begin{thebibliography}{99}                                                                                               %


\bibitem {BBlift}S. Biagi, A. Bonfiglioli: The existence of a global
fundamental solution for homogeneous Hormander operators via a global lifting
method. Proc. Lond. Math. Soc. 114 (2017), 855-889.

\bibitem {BBB1}S. Biagi, A. Bonfiglioli, M. Bramanti: Global estimates for the
fundamental solution of homogeneous H\"{o}rmander operators. Ann. Mat. Pura
Appl. (4) 201 (2022), no. 4, 1875--1934.

\bibitem {BBB2}S. Biagi, A. Bonfiglioli, M. Bramanti: Global estimates in
Sobolev spaces for homogeneous H\"{o}rmander sum of squares. Journal of Math.
Anal. and Appl. Volume 498, Issue 1, 1 June 2021, 124935. Published online on
January 13, 2021.

\bibitem {BiB1}S. Biagi, M. Bramanti: Global Gaussian estimates for the heat
kernel of homogeneous sums of squares. Potential Anal. 59 (2023), no. 1, 113--151.

\bibitem {BiB2}S. Biagi, M. Bramanti: Non-divergence operators structured on
homogeneous H\"{o}rmander vector fields: heat kernels and global Gaussian
bounds. Advances in Differential Equations, 26 (2021), no. 11-12, 621--658.

\bibitem {BB1}M. Bramanti, L. Brandolini: $L^{p}$ estimates for uniformly
hypoelliptic operators with discontinuous coefficients on homogeneous groups.
Rend. Sem. Mat. Univ. Politec. Torino 58 (2000), no. 4, 389--433 (2003).

\bibitem {BB2}M. Bramanti, L. Brandolini: $L^{p}$ estimates for nonvariational
hypoelliptic operators with VMO coefficients. Trans. Amer. Math. Soc. 352
(2000), no. 2, 781--822.

\bibitem {BBbook}M. Bramanti, L. Brandolini: H\"{o}rmander operators. World
Scientific Publishing Co. Pte. Ltd., Hackensack, NJ, 2023, xxviii+693 pp.

\bibitem {BCsing}M. Bramanti, M. C. Cerutti: Commutators of singular integrals
on homogeneous spaces. Boll. Un. Mat. Ital. B (7) 10 (1996), no. 4, 843-883.

\bibitem {BC}M. Bramanti, M. C. Cerutti: $W^{1,2,p}$ solvability for the
Cauchy-Dirichlet problem for parabolic equations with VMO coefficients. Comm.
Partial Differential Equations 18 (1993), no. 9-10, 1735--1763.

\bibitem {BCM}M. Bramanti, M. C. Cerutti, M. Manfredini: $L^{p}$ estimates for
some ultraparabolic operators with discontinuous coefficients. J. Math. Anal.
Appl. 200 (1996), no. 2, 332--354.

\bibitem {BCLP}M. Bramanti, G. Cupini, E. Lanconelli, E. Priola: Global Lp
estimates for degenerate Ornstein-Uhlenbeck operators. Mathematische
Zeitschrift, 266, n. 4 (2010), pp. 789-816.

\bibitem {BF}M. Bramanti, M. S. Fanciullo: The local sharp maximal function
and BMO on locally homogeneous spaces. Annales Academiae Scientiarum Fennicae.
Mathematica. Volumen 42, (2017), 1--20.

\bibitem {BT}M. Bramanti, M. Toschi: The sharp maximal function approach to
estimates for operators structured on H\"{o}rmander vector fields. Rev. Mat.
Complut. (2016) 29: 531-557.

\bibitem {CFL1}F. Chiarenza, M. Frasca, P. Longo: Interior $W^{2,p}$ estimates
for nondivergence elliptic equations with discontinuous coefficients. Ricerche
Mat. 40 (1991), no. 1, 149--168.

\bibitem {CFL2}F. Chiarenza, M. Frasca, P. Longo: $W^{2,p}$-solvability of the
Dirichlet problem for nondivergence elliptic equations with VMO coefficients.
Trans. Amer. Math. Soc. 336 (1993), no. 2, 841--853.

\bibitem {CW}R. R. Coifman, G. Weiss: Analyse harmonique non-commutative sur
certains espaces homog\`{e}nes. Lecture Notes in Math., Vol. 242.
Springer-Verlag, Berlin-New York, 1971, v+160 pp.

\bibitem {HY}H. Dong, T. Yastrzhembskiy: Global $L^{p}$ estimates for kinetic
Kolmogorov-Fokker-Planck equations in nondivergence form. Arch. Ration. Mech.
Anal. 245 (2022), no. 1, 501--564.

\bibitem {Fo2}G.B. Folland: Subelliptic estimates and function spaces on
nilpotent Lie groups. Ark. Mat. \textbf{13} (1975), 161--207.

\bibitem {K}N. V. Krylov: Parabolic and elliptic equations with VMO
coefficients. Comm. Partial Differential Equations 32 (2007), no. 1-3, 453--475.

\bibitem {NSW}A. Nagel, E. M. Stein, S. Wainger: Balls and metrics defined by
vector fields. I. Basic properties. Acta Math. 155 (1985), no. 1-2, 103--147.

\bibitem {PS}G. Pradolini, O. Salinas: Commutators of singular integrals on
spaces of homogeneous type. Czechoslovak Math. J. 57(132) (2007), no. 1, 75--93.

\bibitem {RS}L. P. Rothschild, E. M. Stein: Hypoelliptic differential
operators and nilpotent groups. Acta Math. 137 (1976), no. 3-4, 247--320.
\end{thebibliography}
\end{document}